\documentclass[final,notitlepage,12pt,reqno,tbtags]{amsart}
\usepackage{graphicx}
\usepackage{calc}
\usepackage{url}
\usepackage{mathrsfs}
\usepackage{rotating}
\newlength{\depthofsumsign}
\makeatletter
\let\I\@undefined
\makeatother
\usepackage{geometry}
\usepackage{indentfirst}
\usepackage[normalem]{ulem}
\usepackage{float}
\usepackage{amsthm}

\usepackage{enumitem}[2011/09/28]
\setenumerate{align=left, leftmargin=0pt,labelsep=.5em, labelindent=0\parindent,listparindent=\parindent,itemindent=*}
\usepackage{array}
\usepackage{empheq}
\usepackage{natbib}
\usepackage[all]{xy}

\usepackage{multirow}

\newbox\shell
\newcommand{\dia}[2]{\setbox\shell=\hbox{\begin{picture}(180,120)(-90,-60)#1
\put(-90,-60){\makebox(180,120)[b]{\large #2}}\end{picture}}\dimen0=\ht
\shell\multiply\dimen0by7\divide\dimen0by16\raise-\dimen0\box\shell\hfill}

\usepackage{multirow}

\usepackage{caption}[2012/02/19]

\usepackage{longtable,lscape}

\usepackage{mathptmx}       
\DeclareSymbolFont{operators}{OT1}{txr}{m}{n}
\SetSymbolFont{operators}{bold}{OT1}{txr}{bx}{n}
\def\operator@font{\mathgroup\symoperators}
\DeclareSymbolFont{italic}{OT1}{txr}{m}{it}
\SetSymbolFont{italic}{bold}{OT1}{txr}{bx}{it}
\DeclareSymbolFontAlphabet{\mathrm}{operators}
\DeclareMathAlphabet{\mathbf}{OT1}{txr}{bx}{n}
\DeclareMathAlphabet{\mathit}{OT1}{txr}{m}{it}
\SetMathAlphabet{\mathit}{bold}{OT1}{txr}{bx}{it}
\DeclareSymbolFont{letters}{OML}{txmi}{m}{it}
\SetSymbolFont{letters}{bold}{OML}{txmi}{bx}{it}
\DeclareFontSubstitution{OML}{txmi}{m}{it}
\DeclareSymbolFont{lettersA}{U}{txmia}{m}{it}
\SetSymbolFont{lettersA}{bold}{U}{txmia}{bx}{it}
\DeclareFontSubstitution{U}{txmia}{m}{it}
\DeclareSymbolFontAlphabet{\mathfrak}{lettersA}
\DeclareSymbolFont{symbols}{OMS}{txsy}{m}{n}
\SetSymbolFont{symbols}{bold}{OMS}{txsy}{bx}{n}
\DeclareFontSubstitution{OMS}{txsy}{m}{n}

\usepackage{amssymb,bm,amsmath}
\usepackage{amsfonts}

\usepackage[OT2,T1,T2A]{fontenc}

\usepackage[english,french,german,russian]{babel}
\usepackage{appendix}

\DeclareMathOperator{\IKM}{\mathbf{IKM}}
\DeclareMathOperator{\IvKM}{\mathbf{\widetilde IKM}}

\DeclareMathOperator{\IKvM}{\mathbf{I\widetilde KM}}

\DeclareMathOperator{\IpKM}{\mathbf{\acute IKM}}

\DeclareMathOperator{\IKpM}{\mathbf{I\acute KM}}
\DeclareMathOperator{\Span}{span}
\DeclareMathOperator{\Adj}{adj}

\DeclareMathOperator{\D}{d}
\DeclareMathOperator{\I}{Im}
\DeclareMathOperator{\R}{Re}

\DeclareMathOperator{\cof}{cof}
\DeclareMathOperator{\Kl}{Kl}
\DeclareMathOperator{\Sym}{Sym}

\def\XXint#1#2#3{{\setbox0=\hbox{$#1{#2#3}{\int}$}
     \vcenter{\hbox{$#2#3$}}\kern-.5\wd0}}

\def\RotSymbol#1#2#3{\rotatebox[origin=c]{#1}{$#2#3$}}

\def\isosc{\mathpalette{\resizebox{.1em}{.52em}{|}\hspace{-.11em}\RotSymbol{-180}}\angle}

\bibpunct{[}{]}{,}{n}{}{;}

\def\eor{\hfill$ \square$}

\theoremstyle{plain}
\newtheorem{theorem}{Theorem}[section]
\newtheorem{proposition}[theorem]{Proposition}
\newtheorem{lemma}[theorem]{Lemma}
\newtheorem{corollary}[theorem]{Corollary}

\newenvironment{remark}[1][Remark]{\begin{trivlist}
\item[\hskip \labelsep {\bfseries #1}]}{\end{trivlist}}

\theoremstyle{definition}

\numberwithin{equation}{section}

\usepackage{color,scalerel,stackengine,colonequals}

\stackMath
\newcommand\reallywidehat[1]{%
\savestack{\tmpbox}{\stretchto{%
  \scaleto{%
    \scalerel*[\widthof{\ensuremath{#1}}]{\kern-.6pt\bigwedge\kern-.6pt}%
    {\rule[-\textheight/2]{1ex}{\textheight}}
  }{\textheight}%
}{0.5ex}}%
\stackon[1pt]{#1}{\tmpbox}%
}

\begin{document}

\pagenumbering{roman}
\selectlanguage{english}
\title[Broadhurst--Roberts quadratic relations]{Wro\'nskian algebra and Broadhurst--Roberts quadratic relations}
\author[Yajun Zhou]{Yajun Zhou}
\address{Program in Applied and Computational Mathematics (PACM), Princeton University, Princeton, NJ 08544} \email{yajunz@math.princeton.edu}\curraddr{ \textsc{Academy of Advanced Interdisciplinary Studies (AAIS), Peking University, Beijing 100871, P. R. China}}\email{yajun.zhou.1982@pku.edu.cn}

\date{\today}
\thanks{\textit{Keywords}:   Bessel moments,  Feynman integrals, Wro\'nskian matrices, Bernoulli numbers\\\indent\textit{Subject Classification (AMS 2020)}: 11B68, 33C10, 34M35 (Primary)    81T18, 81T40 (Secondary)\\\indent* This research was supported in part  by the Applied Mathematics Program within the Department of Energy
(DOE) Office of Advanced Scientific Computing Research (ASCR) as part of the Collaboratory on
Mathematics for Mesoscopic Modeling of Materials (CM4).
}

\dedicatory{\begin{center}\footnotesize{\textit{To the memory of Dr.\ W} \emph{(1986--2020)}}\end{center}}
\begin{abstract}
     Through algebraic manipulations on Wro\'nskian matrices whose entries are reducible to Bessel moments, we present a new analytic proof of the quadratic relations conjectured by Broadhurst and Roberts, along with some generalizations. In the Wro\'nskian framework, we reinterpret the de Rham intersection pairing through polynomial coefficients in Vanhove's differential operators, and compute the Betti intersection pairing via linear sum rules for on-shell and off-shell Feynman diagrams at threshold momenta.  From the ideal generated by Broadhurst--Roberts quadratic relations, we derive new non-linear sum rules for on-shell Feynman diagrams, including an infinite family of determinant identities that are compatible with Deligne's conjectures for critical values of motivic   $L$-functions. \end{abstract}
\maketitle

\tableofcontents

\clearpage

\pagenumbering{arabic}

\section{Introduction}
In perturbative  quantum field theory, one frequently encounters Feynman diagrams in the form of on-shell Bessel moments  \cite{LaportaRemiddi1996,Groote2007,BBBG2008,Laporta2008,Laporta:2017okg,Broadhurst2016,BroadhurstMellit2016,Broadhurst2017Paris,BroadhurstRoberts2018,BroadhurstRoberts2019,FresanSabbahYu2020b}  \begin{align} \IKM(a,b;n)\colonequals \int_0^\infty[I_0(t)]^a[K_0(t)]^{b}t^{n}\D t\end{align}for certain non-negative integers $a,b,n\in\mathbb Z_{\geq0}$, where $ I_0(t)=\frac{1}{\pi}\int_0^\pi e^{t\cos\theta}\D\theta$ and $K_0(t)=\int_0^\infty e^{-t\cosh u}\D u $ are modified Bessel functions of zeroth order. These Bessel moments satisfy a wealth of algebraic relations, which had been discovered by  Broadhurst and collaborators through numerical experimentations \cite{Broadhurst2016,BroadhurstMellit2016,Broadhurst2017Paris,BroadhurstRoberts2018,BroadhurstRoberts2019}, before becoming formally proven truths  \cite{HB1,Zhou2017WEF,Zhou2017BMdet,Zhou2018ExpoDESY,Zhou2019BMdimQ,FresanSabbahYu2020b}.

The Broadhurst--Mellit determinant formulae \cite[Conjectures 4 and 7]{Broadhurst2016} were among the earliest verified \cite{Zhou2017BMdet} non-linear algebraic relations for Bessel moments:
the $k\times k$
matrices  $\mathbf M_k\colonequals (\IKM(a,2k+1-a;2b-1))_{1\leq a,b\leq k}$  and  $\mathbf N_k\colonequals (\IKM(a,2k+2-a;2b-1))_{1\leq a,b\leq k}$ have determinants\begin{align}
\det\mathbf M_k=\prod_{j=1}^k\frac{(2j)^{k-j}\pi^j}{\sqrt{(2j+1)^{2j+1}}},\quad \det\mathbf N_k=\frac{2\pi^{(k+1)^2/2}}{\Gamma((k+1)/2)}\prod_{j=1}^{k+1}\frac{(2j-1)^{k+1-j}}{(2j)^j},\label{eq:detMdetN}
\end{align}  where $ \Gamma(x)\colonequals \int_0^\infty t^{x-1}e^{-t}\D t$ for $x>0$.

These determinant formulae are refined by the Broadhurst--Roberts quadratic relations   $\mathbf B_{m}^{}=\mathbf P_{m}^{}\mathbf D_{m}^{}\mathbf P^{\mathrm T} _m$ \cite{Broadhurst2017Paris,BroadhurstRoberts2018,BroadhurstRoberts2019} for the period matrix $ \mathbf P_{m}\colonequals((-1)^{b-1}\pi^{a-\frac{m+3}{2}}\IKM(a,m+2-a;2b-1))_{1\leq a,b\leq\left\lfloor \frac{m+1}2 \right\rfloor}$, in which the Betti matrix  $\mathbf B_{m}\in\mathbb Q^{\left\lfloor \frac{m+1}{2} \right\rfloor\times \left\lfloor \frac{m+1}{2} \right\rfloor}$ and the de Rham matrix $\mathbf D_{m}\in\mathbb Q^{\left\lfloor \frac{m+1}{2} \right\rfloor\times \left\lfloor \frac{m+1}{2} \right\rfloor}$ are filled with rational numbers.  Fres\'an, Sabbah and  Yu have recently proved an equivalent form  \cite[Theorem 1.4]{FresanSabbahYu2020b} of the Broadhurst--Roberts quadratic relations, with an algebro-geometric \textit{tour de force} that draws on their profound insights into   certain exponential motives \cite{FresanJossen2020}, whose associated Galois representations are intimately related to moments of Kloosterman sums \cite{FresanSabbahYu2018,FresanSabbahYu2020a}  (analogs of  on-shell Bessel moments in positive characterisitic).

In this paper, we present an analytic proof of the Broadhurst--Roberts quadratic relations, by revisiting the Wro\'nskian method in our verification of the Broadhurst--Mellit determinant formulae  \cite{Zhou2017BMdet}. Our Wro\'nskian-based proof also produces a new family of quadratic relations, which embody the Broadhurst--Roberts relations  for on-shell Bessel moments as special cases.

In \S\ref{sec:W_alg}, writing   $ D^n
f(u)\colonequals \D^nf(u)/\D u^n$ for $n\in\mathbb Z_{>0}$, and $ D^0f(u)\colonequals f(u)$, we will consider a Wro\'nskian matrix $\mathbf W[f_1(u),\dots,f_m(u)]=(D^{i-1}f_j(u))_{1\leq i,j\leq m}$  whose entries are reducible to on-shell Bessel moments (more precisely, entries of $ \mathbf P_{m}$ and $ \mathbf P_{m-2}$) as $u\to 1^-$, and whose determinant $ \det \mathbf W\colonequals W[f_1(u),\dots,f_m(u)]=\det(D^{i-1}f_j(u))_{1\leq i,j\leq m}$ is explicitly computable through our proof of the Broadhurst--Mellit determinant formulae  \cite{Zhou2017BMdet}. In \S\ref{sec:AdjW_alg}, our key observation is the adjugate relation $ \mathbf W^\star \mathbf W=\mathbf W\mathbf W^\star=W\mathbf I_m$, where the  matrix adjugate $ \mathbf W^\star\colonequals \Adj \mathbf W$ of $ \mathbf W[f_1(u),\dots,f_m(u)]$ factorizes into  $  \mathbf W^\star=\mathbf S\mathbf W^{\mathrm T}\mathbf U$, with $ \mathbf S\in\mathbb Q^{m\times m}$ and $ \mathbf U/W\in\mathbb Q(u)^{m\times m}$. In \S\ref{sec:BR_quad}, we explore the  $ u\to1^- $ limit of the matrix equation $ \mathbf S\mathbf W^{\mathrm T}\mathbf U\mathbf W=W\mathbf I_m$, and establish quadratic relations among   on-shell Bessel moments.

For a generic parameter $u\in(0,1)$,
the individual elements of the  Wro\'nskian matrix $\mathbf W[f_1(u),\dots,f_m(u)]$ in \S\ref{sec:W_alg} will be representable by  off-shell Bessel moments \cite[Definition 2.1]{Zhou2017BMdet}\begin{align}
\IvKM(a+1,b;n|u)\colonequals {}&\int_0^\infty I_{0}(\sqrt{u}t)[I_0(t)]^a[K_0(t)]^{b}t^{n}\D t,\label{eq:IvKM_defn}\\\IKvM(a,b+1;n|u)\colonequals {}&\int_0^\infty K_{0}(\sqrt{u}t)[I_0(t)]^a[K_0(t)]^{b}t^{n}\D t,\label{eq:IKvM_defn}
\end{align}  and  differentiated Bessel moments
 \cite[Definition 2.3]{Zhou2017BMdet}\begin{align}
\IpKM(a+1,b;n|u)\colonequals {}&+\int_0^\infty I_1(\sqrt{u}t)[I_0(t)]^a[K_0(t)]^bt^{n+1}\D t,\label{eq:IpKM_defn}\\\IKpM(a,b+1;n|u)\colonequals {}&-\int_0^\infty K_1(\sqrt{u}t)[I_0(t)]^a[K_0(t)]^bt^{n+1}\D t,\label{eq:IKpM_defn}
\end{align}
where $I_{1}(x)=+\D I_0(x)/\D x $ and $ K_1(x)=-\D K_0(x)/\D x$. Accordingly, our efforts in  \S\ref{sec:AdjW_alg} will culminate in a quadratic  relation involving these generalized Bessel moments and some explicitly computable rational functions (in lieu of the rational numbers in the Betti matrix  $\mathbf B$ and the de Rham matrix $\mathbf D$ of the Broadhurst--Roberts formulation),  as  stated below.
\begin{theorem}[Wro\'nskian algebra]\label{thm:W_alg}For $m\in\mathbb Z_{>0},u\in\big(0,\frac{5-3(-1)^{m}}{2}\big)$,  normalize off-shell Feynman diagrams as \begin{align}
\mathcal F_{m,j}(u)\colonequals \begin{cases}\frac{\IvKM(1,m+1;1|u)+(m+1)\IKvM(1,m+1;1|u)}{(m+2)\pi^{(m+1)/2}}, & j=1, \\
\frac{\IvKM(j,m+2-j;1|u)}{\pi^{(m+1)/2+1-j}}, & j\in\mathbb Z\cap\left[2,\left\lfloor \frac{m}{2} \right\rfloor+1\right] ,\\
\frac{\IKvM(j-\left\lfloor \frac{m}{2} \right\rfloor,m+2+\left\lfloor \frac{m}{2} \right\rfloor-j;1|u)}{\pi^{(m+1)/2+\left\lfloor m/2 \right\rfloor+1-j}}, & j\in\mathbb Z\cap\left[\left\lfloor \frac{m}{2} \right\rfloor+2,m\right],  \\
\end{cases}\label{eq:Fmj_defn}
\end{align}and define a polynomial\begin{align}
\mathcal L_{m}(u)\colonequals u^{\left\lfloor\frac {m+1}2\right\rfloor}\prod_{\substack{n\in\mathbb Z\cap[1,m+1]\\n\equiv m+1\hspace{-.5em}\pmod2}}(u-n^2),\label{eq:Lm_u_defn}
\end{align}where    $ \lfloor x\rfloor$ is the greatest integer less than or equal to $x$. For $ u\in\big(0,\frac{5-3(-1)^{m}}{2}\big)$, the Wro\'nskian matrix $ \mathbf W_m(u)\colonequals\mathbf W[\mathcal F_{m,1}(u),\dots,\mathcal F_{m,m}(u)]$ has determinant \begin{align}
W_{m}(u)\colonequals\det \mathbf W_m(u)=\frac{m+1}{m+2}\frac{\smash[t]{(-1)^{\left\lfloor m/4\right\rfloor }}}{2^{ m (m-1)/2}}\frac{[(m+1)!]^m}{\prod _{n=1}^{m+1} n^n}\frac{1}{|\mathcal L_{m}(u)|^{m/2}}\label{eq:detWm}
\end{align}and satisfies a quadratic relation \begin{align}\mathbf
W_m^{}(u)\mathbf S^{}_m\mathbf W^{\mathrm T}_m(u)=\frac{[\mathbf V_m(u)]^{-1}}{|\mathcal L_{m}(u)|}
\end{align}for matrices $\mathbf V_m^{\vphantom{\mathrm T}}(u)=(-1)^{m+1}\mathbf V_m^{{\mathrm T}}(u)\in\mathbb Q(u)^{m\times m}$ and $ \mathbf S_m^{\vphantom{\mathrm  T}}=(-1)^{m+1}\mathbf S_m^{{\mathrm  T}}\in\mathbb Q^{m\times m}$.

Explicitly, the Vanhove matrix  $\mathbf V_m^{\vphantom{\mathrm T}}(u)\in\mathbb Q(u)^{m\times m}$ is upper-left triangular, with entries\footnote{Throughout this article, we write  ${ n\choose m}\colonequals \frac{n!}{m!(n-m)!}$ for the standard binomial coefficient with $ n\in\mathbb Z_{\geq0},m\in\mathbb Z\cap[0,n]$. By convention, we also define $ {n\choose k}=0$ when $ n\in\mathbb Z_{\geq0},k\in\mathbb Z\smallsetminus[0,n]$. Empty sums like $ \sum_{n=A}^{B}(\cdots)$ with $ A>B$ are treated as zero.}\begin{align}
( \mathbf V^{\vphantom{\mathrm T}}_{m}
(u))_{a,b}\colonequals \sum_{n=a+b-1}^{m}(-1)^{a+n+m+1}{n-a\choose b-1}\frac{D^{n-a-b+1}\ell_{m,n}(u)}{\mathcal L_{m}(u)},\label{eq:Vmat_defn}
\end{align}where the  polynomials\footnote{These polynomials can be constructed by a finite enumeration over integer parameters $ \alpha_1,\dots,\alpha_n$ in \eqref{eq:VVpoly_defn}. See Table \ref{tab:Vanhove_Lm} for a partial list of the Vanhove--Verrill polynomials. }  $ \ell _{m,j}(u)\in\mathbb Z[u]$ are defined through the Vanhove--Verrill relation (cf.\ \cite[\S3.6.1]{Vanhove2019}  and \cite[\S3.6.1]{Vanhove2020}):\begin{align}
\sum_{j=0}^m\ell_{m,j}(u)D^j\colonequals {}&u\widehat \vartheta^m+\sum_{k=1}^{\left\lfloor\frac {m\vphantom{b}}2\right\rfloor+1}u^{1-k}\sum_{\substack{n\in\mathbb Z\cap[1,k],\alpha_{k+1}=1\\\alpha_{n}\in\mathbb Z\cap[1,m+1]\\\alpha_{n+1}\leq \alpha_{n}-2}}(\widehat \vartheta-k)^{m+1-\alpha_1}\prod_{n=1}^{k}\alpha_{n}(\alpha_{n}-m-2)(\widehat \vartheta-k+n)^{\alpha_{n}-\alpha_{n+1}},\label{eq:VVpoly_defn}
\end{align}for  $(\widehat \vartheta f)(u)=D^1[uf(u)]$; the matrix  $ \mathbf S_m^{\vphantom{\mathrm  T}}\in\mathbb Q^{m\times m}$ can be partitioned into $ {{\mathbf S_m}}=\left(\begin{smallmatrix}\mathbf S_m^A &\mathbf  S_m^B \\
\mathbf S_m^C &\mathbf  S_m^D \\
\end{smallmatrix}\right)$ for\footnote{ The Kronecker delta is defined as  $ \delta_{m,n}=1$ when $ m=n$ and $ \delta_{m,n}=0$ when $ m\neq n$. } \begin{align}(\mathbf S_m^A)_{a,b}\colonequals{}&\frac{1+(-1)^{a+b+m+1}}{2^{1-m}} \frac{[1+(m+1) \delta _{a,1}][1+(m+1)\delta_{b,1}]{\sum\limits _{s=1}^{\left\lfloor \frac m2\right\rfloor+2-a} (-1)^s} \binom{m+2-a}{\left\lfloor \frac {m+1}2\right\rfloor+s} \binom{\left\lfloor \frac m2\right\rfloor+1-s}{b-1}}{(-1)^{\left\lfloor\frac{a\vphantom{b}}{2}\right\rfloor+\left\lfloor \frac{b}{2}-\frac{1+(-1)^{m}}{4}\right\rfloor -\left\lfloor \frac m2\right\rfloor-1}(a-1)!(m+2-a)!} \label{eq:SAmat_defn}\end{align}
where  $ a,b\in\mathbb Z\cap\left[1,\left\lfloor\frac{m}2\right\rfloor+1\right]$, \begin{align}\begin{split}&(\mathbf S_m^B)_{a,b'}=(-1)^{m+1}(\mathbf S_m^C)_{b',a}\\\colonequals{}&\frac{1+(-1)^{a+b'+m}}{2^{1-m}} \frac{[1+(m+1) \delta _{a,1}]{\sum\limits _{s=1}^{\left\lfloor \frac m2\right\rfloor+2-a} (-1)^s }\binom{m+2-a}{\left\lfloor \frac {m+1}2\right\rfloor+s}\left[\binom{\left\lfloor \frac m2\right\rfloor+1-s}{b'+1}+(-1)^{b'} \binom{\left\lfloor \frac {m+1}2\right\rfloor+s}{b'+1}\right]}{(-1)^{(a-1)m+\left\lfloor\frac{a\vphantom{b}}{2}-\frac{1+(-1)^{m}}{4}\right\rfloor+\left\lfloor \frac{\smash{b'}\vphantom1}{2}+\frac{1+(-1)^{m}}{4}\right\rfloor -\left\lfloor \frac {m+1}2\right\rfloor}(a-1)!(m+2-a)!}\end{split}\label{eq:SBmat_defn}\end{align}where $ a\in\mathbb Z\cap\left[1,\left\lfloor\frac{m}2\right\rfloor+1\right],b'\in\mathbb Z\cap\left[ 1, \left\lfloor \frac{ m-1}2\right\rfloor\right]$, and \begin{align}(\mathbf  S_m^D )_{a',b'}\colonequals{}&\frac{1+(-1)^{m+1}}{2}\frac{(-4)^{\left\lfloor \frac {m+1}2\right\rfloor-1}}{\left(\left\lfloor \frac {m+1}2\right\rfloor!\right)^{2}}\frac{[1+(-1)^{a'}][1+(-1)^{b'}]}{(-1)^{\left\lfloor \frac{\smash{a'}\vphantom1}{2}\right\rfloor+\left\lfloor \frac{\smash{b'}\vphantom1}{2}\right\rfloor }}\binom{\left\lfloor \frac {m+1}2\right\rfloor}{a'+1}\binom{\left\lfloor \frac {m+1}2\right\rfloor}{b'+1}\label{eq:SDmat_defn}\end{align} where $a', b'\in\mathbb Z\cap\left[ 1, \left\lfloor \frac {m-1}2\right\rfloor\right]$.
 \end{theorem}

\begin{table}\caption{A partial list of Vanhove's operators $ \widetilde L_{m}=\sum_{j=0}^m\ell_{m,j}(u)D^j$\label{tab:Vanhove_Lm}}
\begin{scriptsize}\begin{tabular}{c|l}\hline\hline
$m$ & $\vphantom{\frac\int1}\widetilde L_m$ \\\hline
$1$ & $(u-4) uD^{1}+(u-2)D^{0}$\\
$2$ & $(u-9) (u-1) uD^{2}+(3 u^2-20 u+9)D^{1}+(u-3)D^{0}$\\
$3$ & $(u-16) (u-4) u^2D^{3}+6 u (u^2-15 u+32)D^{2}+(7 u^2-68 u+64)D^{1}+(u-4)D^{0}$\\
$4$ & $(u-25) (u-9) (u-1) u^2D^{4}+2 u (5 u^3-140 u^2+777 u-450)D^{3}+(25 u^3-518 u^2+1839 u-450)D^{2}+(3 u-5) (5 u-57)D^{1}+(u-5)D^{0}$\\
$5$ & $(u-36) (u-16) (u-4) u^3D^{5}+5 u^2 (3 u^3-140 u^2+1568 u-3456)D^{4}+u (65 u^3-2408 u^2+19836 u-27648)D^{3}$\\&$+6 (15 u^3-406 u^2+2078 u-1152)D^{2}+(31 u^2-516 u+1020)D^{1}+(u-6)D^{0}$\\
$6$ & $(u-49) (u-25) (u-9) (u-1) u^3D^{6}+3 u^2 (7 u^4-504 u^3+9870 u^2-51664 u+33075)D^{5}$\\&$+2 u (70 u^4-4179 u^3+64749 u^2-247993 u+99225)D^{4}+2 (175 u^4-8232 u^3+92394 u^2-217012 u+33075)D^{3}$\\&$+(301 u^3-10218 u^2+69561 u-62020)D^{2}+3 (21 u^2-428 u+1071)D^{1}+(u-7)D^{0}$\\
$7$ & $(u-64) (u-36) (u-16) (u-4) u^4D^{7}+28 u^3 (u^4-105 u^3+3276 u^2-32800 u+73728)D^{6}$\\&$+2 u^2 (133 u^4-11928 u^3+307950 u^2-2432512 u+3981312)D^{5}+30 u (35 u^4-2590 u^3+52420 u^2-297856 u+294912)D^{4}$\\&$+3 (567 u^4-32820 u^3+474944 u^2-1616896 u+589824)D^{3}+6 (161 u^3-6645 u^2+57096 u-70912)D^{2}$\\&$+(127 u^2-3076 u+9312)D^{1}+(u-8)D^{0}$\\
$8$ & $(u-81) (u-49) (u-25) (u-9) (u-1) u^4D^{8}+12 u^3 (3 u^5-440 u^4+20482 u^3-345620 u^2+1762035 u-1190700)D^{7}$\\&$+6 u^2 (77 u^5-9856 u^4+391358 u^3-5455164 u^2+21884733 u-10716300)D^{6}$\\&$+6 u (441 u^5-48048 u^4+1569150 u^3-17075168 u^2+48951657 u-14288400)D^{5}$\\&$+3 (2317 u^5-207295 u^4+5264259 u^3-40651797 u^2+68940108 u-7144200)D^{4}$\\&$+6 (1295 u^4-89980 u^3+1614657 u^2-7325514 u+4441878)D^{3}+(3025 u^3-148126 u^2+1547883 u-2474334)D^{2}$\\&$+(255 u^2-7172 u+25461)D^{1}+(u-9)D^{0}$\\\hline\hline
\end{tabular}\end{scriptsize}

\end{table}

 In \S\ref{sec:BR_quad}, we push the theorem above to the      $u\to 1^{-}$ limit, in the following form.
\begin{theorem}[Broadhurst--Roberts quadratic relations]\label{thm:BRquad}Define the  Bessel--Feynman moments $ \mathcal F^\ell_{m,j}(u)$ by replacing all the occurrences of $ \IvKM(\cdot,\cdot;1|u)$ [resp.\ $ \IKvM(\cdot,\cdot;1|u)$] in \eqref{eq:Fmj_defn} with  $ \IvKM(\cdot,\cdot;2\ell-1|u)$ [resp.\ $ \IKvM(\cdot,\cdot;2\ell-1|u)$]. For each $ m\in\mathbb Z_{>1}$, the Broadhurst--Roberts period matrix  $\mathbf P_m\colonequals\linebreak((-1)^{b-1}\mathcal F^b_{m,a}(1))_{1\leq a,b\leq \left\lfloor \frac{m+1}{2} \right\rfloor}$ satisfies a quadratic relation \begin{align}
\mathbf P^{}_m\mathbf D^{}_m\mathbf P^{\mathrm T}_m=\mathbf B_m,
\end{align}for  de Rham matrix $ \mathbf D_m^{}=(-1)^{m+1}\mathbf D_m^{\mathrm T}\in\mathbb Q^{\left\lfloor \frac{m+1}{2} \right\rfloor\times \left\lfloor \frac{m+1}{2} \right\rfloor}$ and   Betti matrix $ \mathbf B_m^{}=(-1)^{m+1}\mathbf B_m^{\mathrm T}\in\mathbb Q^{\left\lfloor \frac{m+1}{2} \right\rfloor\times \left\lfloor \frac{m+1}{2} \right\rfloor}$.

Explicitly, the de Rham matrix $ \mathbf D_{m-2}^{}$ is upper-left triangular, with entries\begin{align}
(\mathbf D_{m-2})_{a,b}\colonequals\lim_{u\to1^-} \frac{|\mathcal L_{m}(u)|((\pmb{\boldsymbol \beta}_{m}^{-1\vphantom{\mathrm T}})^{\mathrm T} \mathbf V^{\vphantom{\mathrm T}}_{m}
(u)\pmb{\boldsymbol \beta}_{m}^{-1})_{a+\left\lfloor \frac{m-1}{2}\right\rfloor+1,b+\left\lfloor \frac{m-1}{2}\right\rfloor+1}}{4(m+2)(-1)^{\left\lfloor \frac{m-1}{2}\right\rfloor}}\label{eq:deRhamD}
\end{align}  computable via the Vanhove matrix $ \mathbf V_m(u)$ [see \eqref{eq:Vmat_defn}] and  the Bessel matrix  $ \pmb{\boldsymbol \beta}_{m}\in\mathbb Z^{m\times m}$  for\begin{align}(\pmb{\boldsymbol \beta}_{m})_{a,b}\colonequals {}&\begin{cases}(-4)^{a-1}\frac{(a-1)!}{(b-a)!}{a-1\choose b-a}, & a\in\mathbb Z\cap\left[1,\left\lfloor\frac{m+1}2\right\rfloor\right], \\[5pt]
(-4)^{a-\left\lfloor\frac{m+1}2\right\rfloor-1}\frac{2\left(a-\left\lfloor\frac{m+1}2\right\rfloor-1\right)!}{\left(b-a+\left\lfloor\frac{m+1}2\right\rfloor-1\right)!}  \Big(\begin{smallmatrix}a-\left\lfloor\frac{m+1}2\right\rfloor \\ b-a+\left\lfloor\frac{m+1}2\right\rfloor\end{smallmatrix}\Big) ,& a\in\mathbb Z\cap\left[\left\lfloor\frac{m+1}2\right\rfloor+1,m\right]; \\
\end{cases}\label{eq:Bessel_mat_beta}
\end{align}the Betti matrix  $ \mathbf B_{m-2}^{}$ satisfies \begin{align}\begin{split}(\mathbf B_{m-2})_{a,b}\colonequals{}&\frac{(\mathbf S^{-1}_{m})_{a+\left\lfloor \frac{m}{2}\right\rfloor+1,b+\left\lfloor \frac{m}{2}\right\rfloor+1}-(\mathbf S^{-1}_{m})_{a+1,b+\left\lfloor \frac{m}{2}\right\rfloor+1}-(\mathbf S^{-1}_{m})_{a+\left\lfloor \frac{m}{2}\right\rfloor+1,b+1}}{4(m+2)(-1)^{\left\lfloor \frac{m-1}{2}\right\rfloor}}\\={}&\frac{(-1)^{a+1}}{2^{m+1}}\frac{   (m-a)! (m-b)!}{(m+1-a-b)!}\frac{\mathsf B_{m+1-a-b}}{ (-1)^{\left\lfloor \frac{m-a-b}{2} \right\rfloor }},
\end{split}\label{eq:BettiB}\end{align}where $ {{\mathbf S_m}}^{}=\left(\begin{smallmatrix}\mathbf S_m^A &\mathbf  S_m^B \\
\mathbf S_m^C &\mathbf  S_m^D \\
\end{smallmatrix}\right)=(-1)^{m+1}{\mathbf S_m^{\mathrm T}}$ is defined through \eqref{eq:SAmat_defn}--\eqref{eq:SDmat_defn}, and the Bernoulli numbers  $\mathsf B_n,n\in\mathbb Z_{\geq0}$ are generated by\begin{align}
\frac{t}{e^{t}-1}\colonequals \sum_{n=0}^\infty \mathsf B_n\frac{t^n}{n!},\quad|t|<\pi.\label{eq:Bn_defn}
\end{align} \end{theorem}

We note that Lee and Pomeransky  \cite[Appendix C]{LeePomeranksy2019} have recently produced some quadratic relations
for off-shell Bessel moments, but not for differentiated Bessel moments. There are technical difficulties when one attempts to push their results to the on-shell limit \cite[\S6.4]{LeePomeranksy2019} and recover Broadhurst--Roberts quadratic relations therefrom.

For each $ m\in\mathbb Z_{>1}$, our Betti matrix $\mathbf  B_m $ agrees with the original formulation of Broadhurst--Roberts \cite[(5.4)--(5.6)]{BroadhurstRoberts2018}, up to    reordering of rows and columns. Instead of considering  Betti intersection pairing that inspired the conjecture of  Broadhurst--Roberts \cite{Broadhurst2017Paris,BroadhurstRoberts2018,BroadhurstRoberts2019} and the proof by Fres\'an--Sabbah--Yu \cite[Theorem 1.4]{FresanSabbahYu2020b}, we generate $\mathbf  B_m $ by inverting
   $  \mathbf  S_{m+2}  $. The latter matrix $  \mathbf  S_{m+2}  $ draws on explicit sum rules (with rational coefficients) for on-shell and off-shell Bessel moments, which extend our earlier works \cite{HB1,Zhou2018ExpoDESY,Zhou2018LaportaSunrise} on the same topic.

In addition to elucidating the relations between Betti matrices and sum rules for individual Bessel moments in \S\ref{subsec:onshell_quad}, we will also establish a family of  sum rules for minor determinants of $ {\mathbf M}_k,k\in\mathbb{Z}_{>1} $ in \S\ref{subsec:app_Feynman} (see Proposition \ref{prop:det_refl} for details):\begin{align}\begin{split}&
\det(\IKM(2a,2(k-a)+1;2b-1))_{1\leq a,b\leq\left\lfloor \frac{k}{2} \right\rfloor}\\={}&\frac{1}{\pi^{\left\lfloor \frac{k+1}{2}\right\rfloor}}\frac{\sqrt{[(2k+1)!!]^{2-(-1)^k}}}{2^{\left\lfloor \frac{k}{2}\right\rfloor}(k-1)!!(k!!)^{1-(-1)^{k}}}\det(\IKM(2a-1,2(k-a+1);2b-1))_{1\leq a,b\leq\left\lfloor \frac{k+1}{2} \right\rfloor}.\label{eq:det_refl}\end{split}
\end{align}Such sum rules may have some arithmetic interest: Deligne's conjecture \cite{Deligne1979} for the motive $ \mathrm M_{2k+1}$ (as defined in \cite[\S7.c]{FresanSabbahYu2020b}) predicts that  the determinants on both sides of \eqref{eq:det_refl} are $ \mathbb Q^\times\pi^{\mathbb Z}$ multiples of critical $ L$-values, while these  critical values for $ L(\mathrm M_{2k+1},s)$ are related to each other by a functional equation \cite[Theorem 1.2]{FresanSabbahYu2018}. No such analogs appear to exist in minor determinants of $ {\mathbf N}_k ,k\in\mathbb{Z}_{>1}$, for both algebraic and analytic reasons.

Our algorithmic construction of the de Rham matrix   $\mathbf D_m $  is based on derivatives of the polynomials $ \ell _{m+2,j}(u)$ at $u=1$, which are readily adaptable to systems different from Bessel moments (see \S\ref{subsec:cofactor}). Within the same Wro\'nskian framework, one can establish several other equivalent representations [see \eqref{eq:deRhamD_1}--\eqref{eq:deRhamd_2} for details] of    $\mathbf D_m $, along with new algebraic relations [see \eqref{eq:M0M_quad}  for details] connecting the period matrix   $\mathbf P_m$  to \begin{align}
\smash{\underset{^\circ }{\mathbf{P}}}_m\colonequals {}&\left( \frac{(-1)^{b-1}}{\pi^{\frac{m+3}{2}-a}}\int_{0}^{\infty}\left\{\frac{[I_{0}(t)]^{a}\log t}{[K_{0}(t)]^{a}}-\frac{\delta_ {a,1}}{m+2}\right\}[K_0(t)]^{m+2}t^{2b-1}\D t\right)_{1\leq a,b\leq\left\lfloor \frac{m+1}2 \right\rfloor}.\label{eq:P0_defn}
\end{align} Outside our Wro\'nskian framework, there are still other approaches to  de Rham representations, such as the  Broadhurst--Roberts   combinatorial recursion formulae \cite[(5.7)--(5.12)]{BroadhurstRoberts2018} and the Fres\'an--Sabbah--Yu self-duality pairing \cite[\S3.b]{FresanSabbahYu2020b}. A direct proof of  the equivalence between Wro\'nskian and non-Wro\'nskian approaches to de Rham matrices  will not appear in this article. Nevertheless, we will outline general principles that may reconcile our representation of de Rham matrices with those of Broadhurst--Roberts and  Fres\'an--Sabbah--Yu, in \S\ref{subsec:deRham_alt}.

\noindent{\textbf{Acknowledgments}} The author thanks David Broadhurst and Javier Fres\'an for their  feedback on the first draft. The author is grateful to Pierre Vanhove for his invitation to \textit{S\'eminaire Motifs et int\'egrales de Feynman}. The author is indebted to two anonymous referees for their suggestions on improving the presentation of this manuscript.

\section{$\mathbf W$ algebra\label{sec:W_alg}}
In \S\ref{subsec:W_Vanhove}, we  open by recapitulating our previous work \cite{Zhou2017BMdet} on  the Broadhurst--Mellit determinant formulae, with an extended discussion on the ordinary differential equations satisfied by off-shell Bessel moments $ \mathcal F_{m,j}(u)$.

In \S\ref{subsec:BM_W}, we elucidate the relation between  Wro\'nskian  matrices and on-shell Bessel moments. We will achieve this by studying the   Wro\'nskian  matrices $ \mathbf W_m(u)\colonequals\mathbf W[\mathcal F_{m,1}(u),\dots,\mathcal F_{m,m}(u)]$ for $m\in\mathbb Z_{>0}$ in the $ u\to1^-$ limit, similar to  procedures reported in
\cite{Zhou2017BMdet}.

In \S\ref{subsec:sum_rules}, we establish, among other things, several families of $\mathbb Q $-linear dependence relations concerning on-shell and off-shell Bessel moments, when the parameter $ u$ is a perfect square. This sets the stage for the analysis of the  adjugate matrices $ \mathbf W^\star_{m}(u)$ for $m\in\mathbb Z_{>0}$, which will unfold in \S\ref{sec:AdjW_alg}.\subsection{Wro\'nskian matrices and Vanhove operators\label{subsec:W_Vanhove}}For $ m\in\mathbb Z_{>0}$, consider the functions $ \mathcal F_{m,j}(u),j\in\mathbb Z\cap[1,m]$ introduced in \eqref{eq:Fmj_defn}, which are normalized versions of the notations in  \cite[Definition 4.1]{Zhou2017BMdet}. In our previous work  \cite[\S4]{Zhou2017BMdet}, we have manipulated
Wro\'nskian matrices
 $ \mathbf W_m(u)\colonequals\mathbf W[\mathcal F_{m,1}(u),\dots,\mathcal F_{m,m}(u)]$ of all sizes, and have shown that their determinants $ W_{m}(u)\colonequals\det \mathbf W_m(u)$  [see \eqref{eq:detWm}] satisfy (cf.~\cite[(4.36) and (4.62)]{Zhou2017BMdet}):\begin{align}
W_{2k-1}(u)={}&\frac{(-1)^{\frac{(k-1)(k-2)}{2}}k[\Gamma(k/2)]^{2}}{\pi ^{k^{2}}u^{k(2k-1)/2}(2k+1)}\frac{(\det\mathbf N_{k-1})^2}{2^{(k-1)(2k-1)+1}}\prod _{j=1}^k\left[\frac{(2j)^2}{(2j)^2-u}\right]^{k-\frac{1}{2}},\quad \forall u\in(0,4),\\W_{2k}(u)={}&\frac{(-1)^{\frac{k(k-1)}{2}}}{\pi^{k(k+1)}}\frac{(2k+1)(\det \mathbf M_k)^{2}}{2^{(2k-1)k+1}u^{k^{2}}(k+1)}\prod _{j=1}^{k+1}\left[\frac{(2j-1)^{2}}{(2j-1)^2-u}\right]^{k},\quad \forall u\in(0,1),
\end{align}for each $ k\in\mathbb Z_{\geq2}$.  In particular, these collateral results for the Broadhurst--Mellit determinant formulae  \eqref{eq:detMdetN} tell us that the  Wro\'nskian matrix $\mathbf W_m(u) $  is non-singular for $  u\in\big(0,\frac{5-3(-1)^{m}}{2}\big)$, hence the linear independence of the functions $ \mathcal F_{m,1}(u),\dots,\mathcal F_{m,m}(u)$  on the same interval.

Furthermore, we may deduce from \cite[Lemma 4.2]{Zhou2017BMdet} that the aforementioned linearly independent functions form complete sets of fundamental solutions to certain ordinary differential equations:\footnote{Using the same techniques of integration by parts as in  \cite[Lemma 4.2]{Zhou2017BMdet}, one can show that $ \widetilde L_m$ annihilates $\IKvM(a,m+2-a;1|u) $ for $ a\in\mathbb Z\cap[2,m+1]$, when $u$ belongs to a certain open interval that ensures the  convergence of the respective integral. } $ C^{\infty}\big(0,\linebreak\tfrac{5-3(-1)^{m}}{2}\big)\cap\ker\widetilde L_{m}=\Span_{\mathbb C}\big\{\mathcal F_{m,j}(u),u\in\big(0,\tfrac{5-3(-1)^{m}}{2}\big)\big|j\in\mathbb Z\cap[1,m]\big\}$,  where Vanhove's operator  $ \widetilde L_{m}$ assumes the following form (omitting lower order terms)  \cite[(9.11)--(9.12)]{Vanhove2014Survey}:\begin{align}
\widetilde L_m={}&\mathcal L_m(u)D^{m}+\frac{m}{2}\frac{\D \mathcal L_m(u)}{\D u}D^{m-1}+\cdots,\label{eq:VL_m}
\end{align}with the polynomial $ \mathcal L_m(u)$ defined in \eqref{eq:Lm_u_defn}. The leading and sub-leading terms in \eqref{eq:VL_m}  explain why $ W_{m}(u)$  is a constant multiple of $ |\mathcal L_m(u)|^{-m/2}$ [see \eqref{eq:detWm}], according to standard manipulations of Wro\'nskian determinants \cite[\S5.2, p.~119]{Ince1956ODE}.

In general,  Vanhove's operator $ \widetilde L_m,m\in\mathbb Z_{>0}$ takes the form $ \widetilde L_m=\sum_{j=0}^m\ell_{m,j}(u)D^j $, where $ \ell _{m,j}(u)\in\mathbb Z[u]$. These polynomials $\ell_{m,j}(u)$ can be deduced from the closed-form representations  in the next proposition, which in turn is amplified and extrapolated from  \cite[\S3.6.1]{Vanhove2019}  and \cite[\S3.6.1]{Vanhove2020}.\begin{proposition}[Vanhove's operators in explicit form]\label{prop:Vanhove_Verrill}For each $ m\in\mathbb Z_{>0}$, we normalize Verrill's polynomials    as $\mathscr  V_{m,0}(t)=t^m$ and\begin{align}\begin{split}
\mathscr V_{m,k}(t)\colonequals {}&t^{m}\sum_{\substack{n\in\mathbb Z\cap[1,k]\\\alpha_{n}\in\mathbb Z\cap[1,m+1]\\\alpha_{n+1}\leq \alpha_{n}-2}}\prod_{n=1}^{k}\alpha_{n}(\alpha_{n}-m-2)\left(\frac{t-n}{t-n+1} \right)^{\alpha_{n}-1}\\={}&\sum_{\substack{n\in\mathbb Z\cap[1,k],\alpha_{k+1}=1\\\alpha_{n}\in\mathbb Z\cap[1,m+1]\\\alpha_{n+1}\leq \alpha_{n}-2}}t^{m+1-\alpha_1}\prod_{n=1}^{k}\alpha_{n}(\alpha_{n}-m-2)(t-n)^{\alpha_{n}-\alpha_{n+1}}\in\mathbb Z[t]\end{split}
\end{align}for $ k\in\mathbb Z_{>0}$, with empty sums being  $0$ and empty products being $1$.

If we define $(\widehat \vartheta f)(u)=D^1[uf(u)]$ and write   $ \lfloor x\rfloor$ for the greatest integer less than or equal to $x$, then we have the following explicit representation for Vanhove's differential operator of $m$-th order:\begin{align}
\widetilde L_m=(-1)^{m}\sum_{k=0}^{\left\lfloor\frac{ m\vphantom{b}}2\right\rfloor+1}u^{1-k}\mathscr V_{m,k}(k-\widehat \vartheta ).\label{eq:Vanhove_Verrill}
\end{align} \end{proposition}\begin{proof}Using the arguments in  \cite[Lemma 4.2]{Zhou2017BMdet}, one can show that Vanhove's operator $ \widetilde L_m$ annihilates\begin{align}
\IKvM(m+1,1;1|u)\colonequals \int_0^\infty K_0(\sqrt{u}t)[I_0(t)]^{m+1}t\D t,\quad u\in ((m+1)^{2},\infty).
\end{align}Thanks to a  convergent Taylor series  \cite[(2.4)]{BSWZ2012}\begin{align}
[I_0(t)]^{m+1}={}&\sum_{k=0}^\infty\mathsf W_{m+1}(2k)\left( \frac{t^k}{2^{k}k!} \right)^2, \text{ where }\mathsf W_{m+1}(2k)=\sum_{\substack{a_1,a_2,\dots, a_{m+1}\in\mathbb Z_{\geq0}\\a_1+a_2+\cdots+a_{m+1}=k}}\left( \frac{k!}{a_1!a_2!\cdots a_{m+1}!} \right)^2,\label{eq:I0_pow}
\end{align} we may  invoke the dominated convergence theorem, and integrate termwise, as done by Glasser--Montaldi  \cite[(A3)]{GlasserMontaldi1993}:\begin{align}\IKvM(m+1,1;1|u)={}&\int_0^\infty K_0(\sqrt{u}t)\left[\sum_{k=0} ^{\infty}\mathsf W_{m+1}(2k)\left( \frac{t^k}{2^{k}k!} \right)^2\right]t\D t=\sum_{k=0} ^{\infty}\frac{\mathsf W_{m+1}(2k)}{u^{k+1}}.
\end{align}

We  write the right-hand side of \eqref{eq:Vanhove_Verrill} as $ \widetilde R_m=\sum_{j=0}^mr_{m,j}(u)D^j$, with $ r_{m,j} (u)\in\mathbb Q(u)$ being rational functions in $u$. By virtue  of Verrill's recursion \cite[Theorem 1]{Verrill2004} $ \sum_{k=0}^{\left\lfloor\frac m2\right\rfloor+1}\mathscr V_{m,k}(n)\mathsf W_{m+1}(2(n-k))=0$, we know that $ \widetilde R_m$ annihilates $ \IKvM(m+1,1;1|u),u\in ((m+1)^{2},\infty)$, so there exists a rational function  $ \rho(u)\in\mathbb Q(u)$ such that $ \rho(u)\widetilde L_m=\widetilde R_m$. Moreover, we recover Vanhove's leading term $ \mathcal L_m(u)D^{m}$ in \eqref{eq:VL_m}   from \begin{align}
r_{m,m}(u)={}&u^{m+1}+u^{\left\lfloor\frac {m+1}2\right\rfloor}\sum_{k=1}^{\left\lfloor\frac {m\vphantom{b}}2\right\rfloor+1}u^{\left\lfloor\frac {m\vphantom{b}}2\right\rfloor+1-k}\sum_{\substack{n\in\mathbb Z\cap[1,k]\\\alpha_n\in\mathbb Z\cap[1,m+1]\\\alpha_{n+1}\leq \alpha_n-2}}\prod_{n=1}^k\alpha_{n}(\alpha_{n}-m-2)=u^{\left\lfloor\frac {m+1}2\right\rfloor}\prod_{\substack{n\in\mathbb Z\cap[1,m+1]\\n\equiv m+1\hspace{-.5em}\pmod2}}(u-n^2)
\end{align} where the last combinatorial identity owes to Djakov--Mityagin (see \cite[Theorem 4.1]{Djakov2004} or \cite[Theorem 8]{Djakov2007}) and Zagier \cite[Theorem 2.7]{BSWZ2012}. Hence $ \rho(u)=1$ and $ \mathcal L_m(u)=\ell_{m,m}(u)$. \end{proof}\begin{remark}For an equivalent formulation of Vanhove's operator in the variable $ s=\frac1u$ and its proof, see the recent work of B\"onisch--Fischbach--Klemm--Nega--Safari \cite[\S3.1]{Klemm2020}. \eor\end{remark}

By Vanhove's original construction of  $ \widetilde L_m=\sum_{j=0}^m\ell_{m,j}(u)D^j $ in \cite[\S9]{Vanhove2014Survey}, we know that  $ \ell _{m,j}(u)$ is always a polynomial in $u$ with integer coefficients. In the next corollary, we will use the explicit formula \eqref{eq:Vanhove_Verrill} to show that such  polynomials have certain divisibility properties, thereby paving the way for the asymptotic analysis (to be elaborated in Proposition \ref{prop:u0_sum}) of solutions to Vanhove's differential equations
in the $u\to0^+$ regime.\begin{corollary}[Regular singular point]\label{cor:reg_sing}For all $ m\in\mathbb Z_{>0}$, we have $ u^{m-j-\left\lfloor\frac {m+1}2\right\rfloor}\ell_{m,j}(u)\in \mathbb Z[u]$ for $ j\in\mathbb Z\cap[0,m]$. As a consequence, for each $ m\in\mathbb Z_{>0}$, the point $u=0$ is a regular singular point for  Vanhove's differential equation $ \widetilde L_m f(u)=0$, so that for a suitably chosen $ \delta_m\in(0,1)$, every solution $f(u),u\in(0,\delta_m)$  is expressible as a convergent series in (fractional) powers of $ u$ and $\log u$.\end{corollary}\begin{proof}By  \eqref{eq:Vanhove_Verrill}, we have $\widetilde L_mu^{j}=(-1)^{m}\sum_{k=0}^{\left\lfloor\frac{ m\vphantom{b}}2\right\rfloor+1}u^{j+1-k}\mathscr V_{m,k}(k-j-1 )
$ for all $ j\in\mathbb R$. Hence, we obtain $ \lim_{u\to0^+}u^{m+1-j-\left\lfloor\frac {m+1}2\right\rfloor}\widetilde L_mu^{j}=0$.

From   $ \widetilde L_m=\sum_{j=0}^m\ell_{m,j}(u)D^j $, we have  $ u^{m+1-j-\left\lfloor\frac {m+1}2\right\rfloor}\widetilde L_mu^{j}= u^{m+1-j-\left\lfloor\frac {m+1}2\right\rfloor}\sum_{n=0}^j\ell_{m,n}(u)D^n u^j$ for  $ j\in\mathbb Z\cap[0,m]$. In view of the output from the last paragraph, we can deduce $\lim_{u\to0^+}u^{m+1-j-\left\lfloor\frac {m+1}2\right\rfloor} \ell_{m,j}(u)=0$ for all $j\in\mathbb Z\cap[0,m]$, beginning with the $j=0$ case, and proceeding with finitely many recursions. This shows that  $ u^{m-j-\left\lfloor\frac {m+1}2\right\rfloor}\ell_{m,j}(u)\in \mathbb Z[u]$ for $ j\in\mathbb Z\cap[0,m]$, as claimed.

Now that  $ u^{m-j}\ell_{m,j}(u)/\ell_{m,m}(u)$ is holomorphic in a non-void $ \mathbb C$-neighborhood of  $u=0$ for every $ j\in\mathbb Z\cap[0,m-1]$, we know that the differential equation $ \widetilde L_m f(u)=0,u\in(0,1)$ has a regular singular point at $u=0$, according to the Fuchs condition \cite[\S15.3]{Ince1956ODE}. The existence of generalized power series expansion then follows immediately.  \end{proof}

For any differential operator $
\widehat O=\sum_{k=0}^{n}a_{k}(u)D^{k}$ with $a_k\in C^\infty(a,b),k\in\mathbb Z\cap[0,n] $,
its formal adjoint (or ``formal dual'') $ \widehat O^* $ acts on every smooth function $ f\in C^\infty(a,b)$ in the following manner: $ \widehat O^*f(u)=\sum_{k=0}^{n}(-1)^{k}D^{k}[a_{k}(u)f(u)]$. This formal adjoint is the unique operator satisfying $ \langle \widehat Of,g\rangle_{(a,b)}=\langle f,\widehat O^*g\rangle_{(a,b)}$  for every smooth function $ f\in C^\infty(a,b)$ and every compactly supported smooth function $ g\in C^\infty_{c}(a,b)$, where $ \langle f,g\rangle_{(a,b)}=\int_a^bf(u)g(u)\D u$.

We study the formal adjoint of each Vanhove operator in the next proposition, which will become useful later in Proposition \ref{prop:Vanhove_dual}.\begin{proposition}[Parity of Vanhove operators]\label{prop:parity_VL}For every $ m\in\mathbb Z_{>1}$, we have $\widetilde L_m^*=(-1)^m\widetilde L_m $. \end{proposition}
\begin{proof}We compute  each summand in \begin{align}
\widetilde L_m^{*}=(-1)^{m}\sum_{k=0}^{\left\lfloor\frac {m\vphantom{b}}2\right\rfloor+1}[u^{1-k}\mathscr V_{m,k}(k-\widehat \vartheta )]^{*}\label{eq:Vanhove_Verrill_dual}\tag{\ref{eq:Vanhove_Verrill}$^*$}\end{align}separately.

For $ k=0$, we have $ [u\mathscr V_{m,0}(-\widehat \vartheta )]^{*}=(-1)^m(u\widehat \vartheta ^{m})^{*}=u\widehat \vartheta^{m}$ through integration by parts.

For $ k\in\mathbb Z_{>0}$, we  operate on power functions $ u^{-1-r},r\in\mathbb R$:{\allowdisplaybreaks\begin{align}
u^{1-k}\mathscr V_{m,k}(k-\widehat \vartheta )\frac{1}{u^{r+1}}={}&\sum_{\substack{n\in\mathbb Z\cap[1,k],\alpha_{0}=m+1,\alpha_{k+1}=1\\\alpha_n\in\mathbb Z\cap[1,m+1],\alpha_{n+1}\leq \alpha_n-2}}\frac{\prod_{n=1}^{k}\alpha_{n}(\alpha_{n}-m-2)}{u^{r+k}}\prod_{j=0}^{k}(r+k-j)^{\alpha_{j}-\alpha_{j+1}},\\(-1)^{m}[u^{1-k}\mathscr V_{m,k}(k-\widehat \vartheta )]^{*}\frac{1}{u^{r+1}}={}&\sum_{\substack{n\in\mathbb Z\cap[1,k],\beta_{0}=m+1,\beta_{k+1}=1\\\beta_n\in\mathbb Z\cap[1,m+1],\beta_{n+1}\leq \beta_n-2}}\frac{\prod_{n=1}^{k}\beta_{n}(\beta_{n}-m-2)}{u^{r+k}}\prod_{j=0}^{k}(r+j)^{\beta_{j}-\beta_{j+1}}.
\end{align}}One can identify the left-hand sides of the last  two equations, by pairing up the dummy indices  $ \beta_n=m+2-\alpha_{k+1-n}$ for $ n\in\mathbb Z\cap[1,k]$.

 Since the parameter $ r$ in the last paragraph is an arbitrary real number, we have $ [u^{1-k}\mathscr V_{m,k}(k-\widehat \vartheta )]^{*}=(-1)^{m}u^{1-k}\mathscr V_{m,k}(k-\widehat \vartheta )$ for $k\in\mathbb Z_{\geq0}$,  finally leading us to the claimed parity relation $\widetilde L_m^*=(-1)^m\widetilde L_m $. \end{proof}

An immediate consequence of the last proposition is a combinatorial constraint  on the coefficients $ \ell_{m,j}(u)$ in  Vanhove's operator $ \widetilde L_m =\sum_{j=0}^m\ell_{m,j}(u)D^j $ for each $ m\in\mathbb Z_{>1}$.
\begin{corollary}[Combinatorial constraint on $ \ell_{m,j}(u)$] For each $ m\in\mathbb Z_{>1}$, we have \begin{align}
(-1)^{m}\ell_{m,j}(u)=\sum_{n=j}^{m}(-1)^n{n\choose j} D^{n-j}
\ell_{m,n}(u)\label{eq:VV_poly_rec}\end{align}for $j\in\mathbb Z\cap[0,m] $, which incorporates Vanhove's sub-leading  term [cf.~\eqref{eq:VL_m}]\begin{align}
\ell_{m,m-1}(u)=\frac{m}{2}D^{1}\ell_{m,m}(u)
\label{eq:sublead}\end{align}as a special case.\end{corollary} \begin{proof}With the product rule for differentiation, we compute \begin{align}
(-1)^m\widetilde L_mf(u)=\widetilde L_{m}^*f(u)={}&\sum_{n=0}^{m}(-1)^nD^n[\ell_{m,n}(u)f(u)]=\sum_{n=0}^{m}(-1)^n\sum _{j=0}^n{n\choose j}[D^{n-j}\ell_{m,n}(u)]D^{j}f(u),
\end{align}  whose coefficients for $ D^{j}f(u)$ boil down to  \eqref{eq:VV_poly_rec}.

For $ j=m-1 $, the constraint   \eqref{eq:VV_poly_rec} reduces to $(-1)^{m}\ell_{m,m-1}(u)=(-1)^{m}mD^1\ell_{m,m}(u)+(-1)^{m-1} \ell_{m,m-1}(u)$, which is equivalent to \eqref{eq:sublead}, a vital result for the computations of Wro\'nskian determinants.  \end{proof}\begin{remark}We refer our readers to Corollary \ref{cor:VvSs} in \S\ref{subsec:cofactor}, for non-trivial applications of the generic recursion  \eqref{eq:VV_poly_rec}. \eor\end{remark}

During our proof of Proposition~\ref{prop:Vanhove_Verrill}, we have considered solutions to Vanhove's differential equations $\widetilde L_m f(u)=0,u\in (m+1,\infty)$. Later on in \S\S\ref{subsec:threshold}--\ref{subsec:offshell_quad}, we will also find it helpful to solve $ \widetilde L_{m}f(u)=0,u\in\mathbb R\smallsetminus\big\{\big(2j-\frac{1+(-1)^{m}}{2}\big)^2\big|j\in\mathbb Z\cap\big[1,\left\lfloor \frac m2\right\rfloor+1\big]\big\}$ for each $m\in\mathbb Z_{>0}$.  To solve Vanhove's differential equations  in such scenarios, we need to extend our previous notations for the  Bessel--Feynman moments $ \mathcal F^\ell_{m,j}(u)$ to  $ \mathcal F^\ell_{m,j}(u)=\pi^{j-(m+1)/2-\left\lfloor \frac m2\right\rfloor -1}\IKvM\big(j-\left\lfloor \frac m2\right\rfloor,m+2+\left\lfloor \frac m2\right\rfloor-j;2\ell-1\big|u\big)$ for  $  j\in\mathbb Z\cap\big[\left\lfloor \frac m2\right\rfloor+2,m+\left\lfloor \frac m2\right\rfloor+1\big]$.
In addition, it is also appropriate to  enlist the differentiated Bessel moments in \eqref{eq:IpKM_defn}--\eqref{eq:IKpM_defn},  and define $ \acute{\mathcal F}^\ell_{m,j}(u)$ by updating each occurrence of  $ \IvKM(\cdot,\cdot;2\ell-1|u)$ [resp.\ $ \IKvM(\cdot,\cdot;2\ell-1|u)$]  in   $ \mathcal F^\ell_{m,j}(u)$  with  $ \IpKM(\cdot,\cdot;2\ell-1|u)$ [resp.\ $ \IKpM(\cdot,\cdot;2\ell-1|u)$].\subsection{Block diagonalization of on-shell Wro\'nskian matrices\label{subsec:BM_W}}Before moving on, we need to reduce the higher-order derivatives in our formulation of  Wro\'nskian matrices into off-shell Bessel moments \eqref{eq:IvKM_defn}--\eqref{eq:IKvM_defn} and differentiated Bessel moments \eqref{eq:IpKM_defn}--\eqref{eq:IKpM_defn}.

As in \cite[\S4.1]{Zhou2017BMdet}, we apply the Bessel differential operator iteratively, to derive\begin{align}
\frac{\mathcal F^\ell_{m,j}(u)}{4^{\ell-1}}={}&(uD^{2}+D^{1})^{\ell-1}\mathcal F_{m,j}(u),\\\frac{\acute{\mathcal F}^\ell_{m,j}(u)}{4^{\ell-1}\sqrt{u}}={}&2D^{1}(uD^{2}+D^{1})^{\ell-1}\mathcal F_{m,j}(u).
\end{align} In the  next lemma, we will spell out the differential operators $ (uD^{2}+D^{1})^{\ell-1}$ and $ D^{1}(uD^{2}+D^{1})^{\ell-1}$ as linear differential operators with polynomial coefficients.\begin{lemma}[Iterations of Bessel differential operators]\label{lm:Bessel_diff_iter}For $ n\in\mathbb Z_{\geq0}$, we have\begin{align}
(uD^{2}+D^{1})^{n}={}&\sum_{m=0}^n\frac{n!}{m!}{n\choose m}u^{m}D^{n+m},\label{eq:Bessel_it_pow}\\D^{1}(uD^{2}+D^{1})^{n}={}&\sum_{m=0}^n\frac{n!}{m!}{{n+1}\choose m+1}u^{m}D^{n+m+1}.\end{align}\end{lemma}\begin{proof}It is clear that \eqref{eq:Bessel_it_pow} is valid for $n=0 $. Suppose that it holds for $ n=N\in \mathbb Z_{\geq0}$, then{\allowdisplaybreaks\begin{align}\begin{split}&
(uD^{2}+D^{1})^{N+1}f(u)=(uD^2+D^1)\sum_{m=0}^N\frac{N!}{m!}{N\choose m}u^{m}D^{N+m}f(u)\\={}&\sum_{m=0}^N\frac{N!}{m!}{N\choose m}
u^{m-1}[m^2D^{0}+(2 m +1)u D^{1}+u^2 D^{2}]D^{N+m}f(u)\\={}&u^{N+1}D^{2N}f(u)+\sum_{m=0}^N\left[(m+1)^2\frac{N!}{(m+1)!}{N\choose{ m+1}}+(2 m +1)\frac{N!}{m!}{N\choose m}
 \right.\\&\left.+\frac{N!}{(m-1)!}{N\choose{ m-1}}
\right]u^{m}D^{N+1+m}f(u).\end{split}
\end{align}}Since the last bracketed expression in equal to $ \frac{(N+1)!}{m!}{{N+1}\choose m}$, we know that  \eqref{eq:Bessel_it_pow}  also holds for $n=N+1$.

Now that we have verified \eqref{eq:Bessel_it_pow}  by induction, we may differentiate its right-hand side termwise, to obtain\begin{align}\begin{split}&
D^{1}(uD^{2}+D^{1})^{n}f(u)=D^1\sum_{m=0}^n\frac{n!}{m!}{n\choose m}u^{m}D^{n+m}f(u)
\\={}&\sum_{m=0}^n\frac{n!}{m!}{n\choose m}u^{m-1}(mD^{0}+uD^{1})D^{n+m}f(u)
\\={}&\sum_{m=0}^n\left[ (m+1)\frac{n!}{(m+1)!}{n\choose{ m+1}}+\frac{n!}{m!}{n\choose m} \right]u^{m}D^{n+m+1}f(u).
\end{split}\end{align}Here, the last bracketed expression can be identified with  $\frac{n!}{m!}{n+1\choose m+1}$, hence our claim. \end{proof}

   \begin{proposition}[Bessel moment representations of Wro\'nskian matrices]\label{prop:Bessel_mat}We have\begin{align}
\begin{split}&
 (\mathcal F_{m,j}^{1}(u),(-1)^{2-1}\mathcal F_{m,j}^{2}(u),\ldots,(-1)^{\left\lfloor\frac{ m+1}{2}\right\rfloor-1}\mathcal F_{m,j}^{\left\lfloor\frac{ m+1}{2}\right\rfloor}(u),\\&\sqrt{u}\acute{\mathcal F}_{m,j}^1(u),(-1)^{2-1}\sqrt{u}\acute{\mathcal F}_{m,j}^2(u),\ldots,(-1)^{k-2}\sqrt{u}\acute{\mathcal F}_{m,j}^{\left\lfloor\frac{ m}{2}\right\rfloor}(u))^{\mathrm T} \\
={}&\pmb{\boldsymbol \beta}_{m}(u)(D^0\mathcal F_{m,j}(u),\ldots,D^{m-1}\mathcal F_{m,j}(u))^{\mathrm T},\end{split}
\end{align} where the $ (a,b)$-th entries of the square matrix $ \pmb{\boldsymbol \beta}_{m}(u)\in\mathbb Z[u]^{m\times m}$ are given by\begin{align}
(\pmb{\boldsymbol \beta}_{m}(u))_{a,b}={}&\begin{cases}(-4)^{a-1}\frac{(a-1)!}{(b-a)!}{a-1\choose b-a}u^{b-a}, & a\in\mathbb Z\cap\left[1,\left\lfloor\frac{m+1}2\right\rfloor\right], \\[5pt]
(-4)^{a-\left\lfloor\frac{m+1}2\right\rfloor-1}\frac{2\left(a-\left\lfloor\frac{m+1}2\right\rfloor-1\right)!}{\left(b-a+\left\lfloor\frac{m+1}2\right\rfloor-1\right)!}  \Big(\begin{smallmatrix}a-\left\lfloor\frac{m+1}2\right\rfloor \\ b-a+\left\lfloor\frac{m+1}2\right\rfloor\end{smallmatrix}\Big) u^{b-a+\left\lfloor\frac{m+1}2\right\rfloor},& a\in\mathbb Z\cap\left[\left\lfloor\frac{m+1}2\right\rfloor+1,m\right]. \\
\end{cases}
\end{align}Here, it is understood that $ {N\choose M}=0$ and $ \frac{1}{(-N-1)!}=0$ if $ N\in\mathbb Z_{\geq0}$ and $ M\in\mathbb Z\smallsetminus[0,N]$.

\end{proposition}\begin{proof}This is essentially a transcription of the results from Lemma \ref{lm:Bessel_diff_iter}. [Note: we have abbreviated $ \pmb{\boldsymbol \beta}_{m}(1)$ as $ \pmb{\boldsymbol \beta}_{m}$ in \eqref{eq:Bessel_mat_beta}, by slight abuse of notation.] \end{proof}\begin{remark}It is not hard to show, through elementary manipulations of determinants, that $ |\det\pmb{\boldsymbol \beta}_{m}(u)|=2^{m(m-1)/2}|u|^{\left\lfloor \vphantom{\frac11}{m^2}/4 \right\rfloor}$. Therefore, the matrix $ \pmb{\boldsymbol \beta}_{m}(u)$ is invertible when $ u\neq0$.\eor\end{remark}

As we may recall from \cite[\S4]{Zhou2017BMdet}, upon suitable column manipulations, some elements of
the matrix $\pmb{\boldsymbol \beta}_{m}(u)\mathbf W_{m}(u) $  are  expressible as on-shell Bessel moments when $u\to1^-$. Such reductions have led us to factorizations of Wro\'nskian determinants into products of  Broadhurst--Mellit determinants  \cite[\S4]{Zhou2017BMdet}. In the next proposition, we transfer our earlier results onto
the matrix $\pmb{\boldsymbol \beta}_{m}(u)\mathbf W_{m}(u) $ explicitly.\begin{proposition}[Block tridiagonalization of Wro\'nskian matrices]\label{prop:lim_W} For $ k\in\mathbb Z_{>0}$, define $ \pmb{\boldsymbol \alpha}_{k}\in\mathbb Z^{(2k-1)\times(2k-1)}$ as \begin{align}
(\pmb{\boldsymbol \alpha}_{k})_{a,b}\colonequals \begin{cases}\delta_{a,b}-\delta_{a+k-1,b}, & a\in\mathbb Z\cap[2,k]; \\
\delta_{a,b}, & a\in\mathbb Z\cap(\{1\}\cup[k+1,2k-1]),\\
\end{cases}\label{eq:Amat_defn}
\end{align}and $\pmb{\boldsymbol\psi}_{k}\in\mathbb Z^{2k\times( 2k-1) },\pmb{\boldsymbol \varrho}_k\in\mathbb Z^{(2k-1)\times 2k}$ as\begin{align}
(\pmb{\boldsymbol\psi}_{k})_{a,b}\colonequals {}&\begin{cases}\delta_{a,b}(1-\delta_{a,k+1}), & a\in\mathbb Z\cap[1,k+1]; \\
\delta_{a,b+1}, & a\in\mathbb Z\cap[k,2k],\\
\end{cases}\\(\pmb{\boldsymbol \varrho}_k)_{a,b}\colonequals {}&\delta_{a,b},
\end{align} so that  the $(k+1)$-st column and the bottom row of $ \mathbf X\in\mathbb R^{2k\times 2k}$ are dropped to form $ \pmb{\boldsymbol \varrho}_k\mathbf X\pmb{\boldsymbol\psi}_{k}\in\mathbb R^{(2k-1)\times(2k-1)}$. For  period matrices $ \mathbf P_{m}\colonequals((-1)^{b-1}\pi^{a-\frac{m+3}{2}}\IKM(a,m+2-a;2b-1))_{1\leq a,b\leq\left\lfloor \frac{m+1}2 \right\rfloor}$ of all sizes,  we have\begin{align}
\pmb{\boldsymbol \beta}_{2k-1}(1)\mathbf W_{2k-1}(1)\pmb{\boldsymbol \alpha}_{k}^{\vphantom{\mathrm T}}={}&\begin{pmatrix}\mathbf P_{2k-1}^{\mathrm T} &    \\[5pt]
\acute{\mathbf{P}}_{2k-1}^{\mathrm T} & -\mathbf{P}_{2k-3}^{\mathrm T}  \\
\end{pmatrix}\label{eq:Mat1}
\intertext{and}
\pmb{\boldsymbol \varrho}_k\pmb{\boldsymbol \beta}_{2k}(1)\lim_{u\to1^-}\mathbf W_{2k}(u)\pmb{\boldsymbol\psi}_{k}\pmb{\boldsymbol \alpha}_{k}={}&\begin{pmatrix}\mathbf P_{2k}^{\mathrm T} &    \\[5pt]
\acute{\mathbf{P}}_{2k}^{\mathrm T} & -\mathbf{P}_{2k-2}^{\mathrm T}  \\
\end{pmatrix},\label{eq:Nat1}
\end{align} where the unwritten blocks in the top-right positions are   filled with zeros, and $\acute{\mathbf P}_m\colonequals\linebreak ((-1)^{b-1}\acute{\mathcal F}_{m,a}^b(1))_{1\leq a\leq\left\lfloor \frac{m+1}{2} \right\rfloor,1\leq b\leq\left\lfloor \frac{m-1}{2} \right\rfloor} $.\end{proposition}\begin{proof}Multiplying $ \pmb{\boldsymbol \alpha}_{k}$ to the right of a $ (2k-1)\times(2k-1)$ matrix, we are effectively subtracting its 2nd to $ k$-th columns from its last $ (k-1)$ columns. By definition, we have $ \mathcal{F}^\ell_{m,j}(1)-\mathcal{F}^\ell_{m,j+\left\lfloor \frac{m+1}2 \right\rfloor-1}(1)=0,\forall j\in\mathbb Z\cap\big[2,\left\lfloor \frac{m+1}2 \right\rfloor\big]$, accounting for the top-right blocks with vanishing entries in \eqref{eq:Mat1} and \eqref{eq:Nat1}. Meanwhile, the Wro\'nskian relation $ I_1(t)K_0(t)+I_0(t)K_1(t)=\frac{1}{t}$ leads us to an identity  $ \acute{\mathcal{F}}^\ell_{m,j}(1)-\acute{\mathcal F}^\ell_{m,j+\left\lfloor \frac{m+1}2 \right\rfloor-1}(1)=-\mathcal{F}^\ell_{m-2,j-1}(1),\forall j\in\mathbb Z\cap\big[2,\left\lfloor \frac{m+1}2 \right\rfloor\big]$, accounting for the bottom-right blocks  in \eqref{eq:Mat1} and \eqref{eq:Nat1}. It is clear that the bottom-left blocks $ \acute{\mathbf{P}}_{2k-1}^{\mathrm T}  $ and $ \acute{\mathbf{P}}_{2k}^{\mathrm T}$ originate from the intact entries in $ \pmb{\boldsymbol \beta}_{2k-1}(1)\mathbf W_{2k-1}(1)$ and $ \pmb{\boldsymbol \varrho}_k\pmb{\boldsymbol \beta}_{2k}(1)\lim_{u\to1^-}\mathbf W_{2k}(u)\pmb{\boldsymbol\psi}_{k}$.\end{proof}The configurations in \eqref{eq:Mat1} and \eqref{eq:Nat1} were precisely the forms that we prepared \cite[Propositions 4.3 and 4.6]{Zhou2017BMdet} for  our proof of the Broadhurst--Mellit determinant formulae. To get ready for later developments in \S\ref{subsec:onshell_quad} of the current article, we need to go one step further,  eliminating the bottom-left blocks $ \acute{\mathbf{P}}_{2k-1}^{\mathrm T}  $ and $ \acute{\mathbf{P}}_{2k}^{\mathrm T}$ in the next proposition.
\begin{proposition}[Block diagonalization of Wro\'nskian matrices]\label{prop:block_diag_W}Define square matrices $\pmb{\boldsymbol \vartheta}_{m}$, $\pmb{\boldsymbol \varphi}_{m}\in\linebreak\mathbb Q^{\left(2\left\lfloor \frac{m+1}{2}\right\rfloor-1\right)\times\left(2\left\lfloor \frac{m+1}{2}\right\rfloor-1\right)}$ as follows:{\allowdisplaybreaks\begin{align}(
\pmb{\boldsymbol \vartheta}_{m})_{a,b}\colonequals{}&\begin{cases}\delta_{a,b}, & a \in\mathbb Z\cap\left[1,\left\lfloor \frac{m+1}{2}\right\rfloor\right],\\
\delta_{a,b}+\frac{2\left(a-\left\lfloor \frac{m+1}{2}\right\rfloor\right)}{m+2}\delta_{a-\left\lfloor \frac{m+1}{2}\right\rfloor,b}, & a\in\mathbb Z\cap\left[\left\lfloor \frac{m+1}{2}\right\rfloor+1,2\left\lfloor \frac{m+1}{2}\right\rfloor-1\right];
\end{cases}\\(
\pmb{\boldsymbol \varphi}_{m})_{a,b}\colonequals{}&\begin{cases}\delta_{a,b}, & a \in\mathbb Z\cap\left[1,\left\lfloor \frac{m+1}{2}\right\rfloor\right],\\
\delta_{a,b}+\left(1-\frac{b}{m+2}\right)\delta_{a-\left\lfloor \frac{m+1}{2}\right\rfloor+1,b}, & a\in\mathbb Z\cap\left[\left\lfloor \frac{m+1}{2}\right\rfloor+1,2\left\lfloor \frac{m+1}{2}\right\rfloor-1\right].
\end{cases}
\end{align}}For $ k\in\mathbb Z_{>1}$, we have \begin{align}
\pmb{\boldsymbol \vartheta}_{2k-1}\pmb{\boldsymbol \beta}_{2k-1}(1)\mathbf W_{2k-1}(1)\pmb{\boldsymbol \alpha}_{k}^{\vphantom{\mathrm T}}\pmb{\boldsymbol \varphi}_{2k-1}={}&\begin{pmatrix}\mathbf P_{2k-1}^{\mathrm T} &    \\[5pt]
 & -\mathbf P_{2k-3}^{\mathrm T}   \\
\end{pmatrix}\intertext{and}\pmb{\boldsymbol \vartheta}_{2k}\pmb{\boldsymbol \varrho}_k\pmb{\boldsymbol \beta}_{2k}(1)\lim_{u\to1^-}\mathbf W_{2k}(u)\pmb{\boldsymbol\psi}_{k}\pmb{\boldsymbol \alpha}_{k}\pmb{\boldsymbol \varphi}_{2k}={}&\begin{pmatrix}\mathbf P_{2k}^{\mathrm T}  &  \\[5pt]
& -\mathbf P_{2k-2}^{\mathrm T}   \\
\end{pmatrix}.
\end{align}\end{proposition}\begin{proof}
 Clearly, it would suffice to check that  the matrix elements of  $ \acute{\mathbf P}_m$ can be alternatively represented by\begin{align}
(\acute{\mathbf P}_m)_{a,b}\colonequals {}&\begin{cases}(-1)^{b}\frac{2b}{m+2}\mathcal{F}^b_{m,1}(1), & a=1, \\[5pt]
(-1)^{b}\left[\frac{2b}{m+2}\mathcal{F}^b_{m,a}(1)-\left( 1-\frac{a}{m+2} \right)\mathcal{F}^b_{m-2,a-1}(1)\right], & a\in\mathbb Z\cap\big[2,\left\lfloor \frac{m+1}{2} \right\rfloor\big]; \\
\end{cases}\label{eq:P'alt}
\end{align}

Integrating by parts, we have \begin{align}\begin{split}\acute{\mathcal F}_{m,1}^{\ell}(1)\colonequals {}&
\frac1m\int_0^\infty I_1(t)[K_0(t)]^{m-1}t^{2\ell}\D t-\frac{m-1}{m}\int_0^\infty K_1(t)I_{0}(t)[K_0(t)]^{m-2}t^{2\ell -1}\D t\\={}&\frac1m\int_0^\infty t^{2\ell }\frac{\D}{\D t}\{I_{0}(t)[K_0(t)]^{m-1}\}\D t=-\frac{2\ell}{m} \mathcal{F}_{m,1}^{\ell}(1).\end{split}\label{eq:xi'1}
\end{align}Likewise, for $ j\in\mathbb Z\cap\big[2,\left\lfloor \frac{m+1}{2} \right\rfloor\big]$, we have {\allowdisplaybreaks\begin{align}\begin{split}
\acute{\mathcal F}_{m,j}^\ell(1): ={}&\int_0^\infty I_1(t)[I_{0}(t)]^{j-1}[K_0(t)]^{m-j}t^{2\ell}\D t=\frac{1}{j}\int_0^\infty [K_0(t)]^{m-j}t^{2\ell}\frac{\D}{\D t}[I_{0}(t)]^{j}\D t\\={}&-\frac{2\ell}{j}\mathcal{F}_{m,j}^{\ell}(1)+\frac{m-j}{j}\int_0^\infty K_1(t)[I_{0}(t)]^{j}[K_0(t)]^{m-j-1}t^{2\ell}\D t\\={}&-\frac{2\ell}{j}\mathcal{F}_{m,j}^{\ell}(1)+\frac{m-j}{j}\left[ \mathcal{F}_{m-2,j-1}^\ell(1)-\acute{\mathcal F}_{m,j}^\ell(1) \right].\end{split}
\end{align}}This can be rearranged into\begin{align}
\acute{\mathcal F}_{m,j}^\ell(1)=\left(1-\frac{j}{m}\right)\mathcal{F}_{m-2,j-1}^\ell(1)-\frac{2\ell}{m}\mathcal{F}_{m,j}^{\ell}(1),\label{eq:xi'j}
\end{align}which completes our task.
\end{proof}\subsection{Sum rules for on-shell and off-shell Bessel moments\label{subsec:sum_rules}}When $ u=\big(2j-\frac{1+(-1)^{m}}{2}\big)^2,j\in\mathbb Z\cap[1,\lfloor m/2\rfloor+1]$, certain functions among $ \mathcal{F}_{m,j}^{\ell}(u),\acute{\mathcal F}_{m,j}^{\ell}(u)$ for  $j\in\mathbb Z\cap[1,m+\lfloor m/2\rfloor+1]$  assume finite values, and are $ \mathbb Q$-linearly dependent. We refer to such  $ \mathbb Q$-linear dependence relations (together with their modest generalizations) as sum rules for Bessel moments at thresholds. Here, we choose the word ``threshold'' to acknowledge the fact that certain  functions among $ \mathcal{F}_{m,j}^{\ell}(u),\acute{\mathcal F}_{m,j}^{\ell}(u)$ for $j\in\mathbb Z\cap[1,m+\lfloor m/2\rfloor+1]$  go unbounded as  $ u\to\big(2j-\frac{1+(-1)^{m}}{2}\big)^2\pm0^+,j\in\mathbb Z\cap[1,\lfloor m/2\rfloor+1]$. (See Lemma \ref{lm:loc} for quantitative characterization of such divergent behavior near the threshold.)

These sum rules with rational coefficients will be vital to the  analysis  of Wro\'nskian cofactors in \S\ref{subsec:threshold}, which revolves around the identity \begin{align}
\begin{pmatrix*}[r](-1)^{m+1}W\big[\reallywidehat{\mathcal{F}_{m,1}(u)},\dots,{\mathcal{F}_{m,r}(u),}\dots,\mathcal{F}_{m,m}(u)\big] \\
\vdots\phantom{spacespacespace} \\
(-1)^{m+r}W\big[\mathcal{F}_{m,1}(u),\dots,\reallywidehat{\mathcal{F}_{m,r}(u)},\dots,\mathcal{F}_{m,m}(u)\big] \\
\vdots\phantom{spacespacespace} \\
W\big[\mathcal{F}_{m,1}(u),\dots,{\mathcal{F}_{m,r}(u),}\dots,\reallywidehat{\mathcal{F}_{m,m}(u)}\big] \\
\end{pmatrix*}= \frac{\Lambda_m}{|\mathcal L_{m}(u)|^{\frac{m-2}{2}}}\mathbf S_m\begin{pmatrix*}[r]\mathcal{F}_{m,1}(u) \\[3pt]
\vdots\phantom{11\,\,} \\[3pt]
\mathcal{F}_{m,r}(u) \\[3pt]
\vdots\phantom{11\,\,} \\[3pt]
\mathcal{F}_{m,m}(u) \\
\end{pmatrix*}.\label{eq:Sm_fit}
\end{align}Here on the left-hand side,  a caret indicates removal of the term underneath while evaluating the Wro\'nskian determinant of $(m-1)$ functions; on the right-hand side, we have [see \eqref{eq:detWm}]\begin{align}
\Lambda_m\colonequals W_{m}(u)|\mathcal L_{m}(u)|^{m/2}=\frac{m+1}{m+2}\frac{\smash[t]{(-1)^{\left\lfloor m/4\right\rfloor }}}{2^{ m (m-1)/2}}\frac{[(m+1)!]^m}{\prod _{n=1}^{m+1} n^n}\label{eq:Lambda_m_defn}
\end{align}and $ \mathbf S_m\in\mathbb Q^{m\times m}$ is a matrix filled with rational numbers. At the end of this section, we will illustrate the usefulness of $ \mathbb Q$-linear sum rules with an explicit computation of $ \mathbf S_3$ (Corollary \ref{cor:S3_comp}), while deferring the treatment of generic $\mathbf S_m$ until the next section.

To facilitate  future discussion of sum rules, we introduce the following  formal notations {for $m\in\mathbb Z_{>0},s\in\mathbb Z\cap\big[1,\left\lfloor \frac{m}{2}\right\rfloor+1\big]$:\footnote{It is understood that expressions like $ \sum_{j=1}^0(\cdots)$ are empty sums, hence vanishing.} }{\allowdisplaybreaks\begin{align}\begin{split}\mathcal I_{m,s}^\ell(u)\colonequals {}&(m+2)\mathcal{F}_{m,1}^{\ell}(u)+\sum_{j=1}^{\left\lfloor\frac{\left\lfloor m/2\right\rfloor+1-s}{2}\right\rfloor}(-1)^{j}{\left\lfloor \frac{m}{2}\right\rfloor+1-s\choose 2j}\mathcal{F}_{m,2j+1}^{\ell}(u)\\{}&-(-1)^{m}i\sum_{j=0}^{\left\lfloor\frac{\left\lfloor m/2\right\rfloor-s}{2}\right\rfloor}(-1)^{j}{\left\lfloor \frac{m}{2}\right\rfloor+1-s\choose 2j+1} \mathcal F_{m,2j+2}^{\ell}(u),
\label{eq:GI_sum}\end{split}\\\begin{split}\mathcal K_{m,s}^\ell(u)\colonequals {}&-\sum_{j=1}^{\left\lfloor\frac{\left\lfloor(m+1)/2\right\rfloor-1+s}{2}\right\rfloor}(-1)^{j}\left[{\left\lfloor \frac{m}{2}\right\rfloor+1-s\choose 2j+1}+{\left\lfloor \frac{m+1}{2}\right\rfloor+s\choose 2j+1}\right]\mathcal F_{m,2j+\left\lfloor \frac{m}{2}\right\rfloor+1}^{\ell}(u)\\&-(-1)^{m}i\sum_{j=1}^{\left\lfloor\frac{\left\lfloor(m+1)/2\right\rfloor+s}{2}\right\rfloor}(-1)^{j}\left[{\left\lfloor \frac{m}{2}\right\rfloor+1-s\choose 2j}-{\left\lfloor \frac{m+1}{2}\right\rfloor+s\choose 2j}\right]\mathcal F_{m,2j+\left\lfloor \frac{m}{2}\right\rfloor}^{\ell}(u).\end{split}
\label{eq:GK_sum}\end{align}}The domains of definition for the real-valued functions $ \R\mathcal I_{m,s}^\ell(u)$, $ \I \mathcal I_{m,s}^\ell(u)$ and so forth will be clear from context.\footnote{Note that $ \R\mathcal I_{m,s}^\ell(u)$ and  $ \I \mathcal I_{m,s}^\ell(u)$ do not necessarily share their domains of definition, neither do  $ \R\mathcal K_{m,s}^\ell(u)$ and  $ \I \mathcal K_{m,s}^\ell(u)$. In view of this, the notations   $ \mathcal I_{m,s}^\ell(u)$,  $  \mathcal K_{m,s}^\ell(u)$ and so forth are convenient short-hands, instead of complex-valued functions in the usual sense.
} We also write $ \acute{\mathcal I}_{m,s}^\ell(u)$ [resp.\ $ \acute{\mathcal K}_{m,s}^\ell(u)$] for expressions that trade $\mathcal F $   in  $ {\mathcal I}_{m,s}^\ell(u)$ [resp.\ $ {\mathcal K}_{m,s}^\ell(u)$] with $ \acute{\mathcal F}$, and set
  $ \mathsf F_{k,s}^\ell=\mathcal I^\ell_{2k-1,s}-\mathcal K^\ell_{2k-1,s}$, $ \acute{\mathsf F}_{k,s}^\ell=\acute{\mathcal I}^\ell_{2k-1,s}-\acute{\mathcal K}^\ell_{2k-1,s}$,  $ \mathsf G_{k,s}^\ell=\mathcal I^\ell_{2k,s}-\mathcal K^\ell_{2k,s}$,  $ \acute{\mathsf G}_{k,s}^\ell=\acute{\mathcal I}^\ell_{2k,s}-\acute{\mathcal K}^\ell_{2k,s}$.

\begin{proposition}[Sum rules at thresholds]\label{prop:u_sq}In what follows, assume that $ k\in\mathbb Z_{>1}$. \begin{enumerate}[leftmargin=*,  label=\emph{(\alph*)},ref=(\alph*),
widest=a, align=left]
\item For $ s\in\mathbb Z\cap[1,k],\ell\in\mathbb Z\cap\big[1,\left\lfloor\frac{k}{2}\right\rfloor\big]$, we have $\mathsf F_{k,s}^\ell((2s)^2)=0 $; for $
s\in\mathbb Z\cap[1,k],\ell\in\mathbb Z\cap\big[1,\left\lfloor\frac{k-1}{2}\right\rfloor\big]$, we have $ \acute{\mathsf F}_{k,s}^\ell((2s)^2)=0$.

\item
For $ s\in\mathbb Z\cap[1,k+1],\ell\in\mathbb Z\cap\big[1,\left\lfloor\frac{k}{2}\right\rfloor\big]$, we have $ \mathsf G_{k,s}^\ell((2s-1)^2)=0$; for $
s\in\mathbb Z\cap[1,k+1],\ell\in\mathbb Z\cap\big[1,\left\lfloor\frac{k-1}{2}\right\rfloor\big]$, we have  $ \acute{\mathsf G}_{k,s}^\ell((2s-1)^2)=0$.
\end{enumerate}
\end{proposition}\begin{proof}As in \cite[\S2.1, \S3]{Zhou2018ExpoDESY} and \cite[Remark after Lemma 3.2]{Zhou2018LaportaSunrise}, we will prove  these sum rules
by studying contour integrals involving the cylindrical Hankel functions  $ H^{(1)}_0(z), H_0^{(2)}(z),H_1^{(1)}(z),H_1^{(2)}(z)$ for $ z\in\mathbb C\smallsetminus(-\infty,0]$, whose  asymptotic behavior reads\begin{align}
\left\{\begin{array}{c}
H_n^{(1)}(z)=\sqrt{\frac{2}{\pi z}}e^{+i\left(z-\frac{n\pi}{2}-\frac{\pi}{4}\right)}\left[1+O\left( \frac{1}{|z|} \right)\right] \\
H_n^{(2)}(z)=\sqrt{\frac{2}{\pi z}}e^{-i\left(z-\frac{n\pi}{2}-\frac{\pi}{4}\right)}\left[1+O\left( \frac{1}{|z|} \right)\right] \\
\end{array}\right.\label{eq:H0H1_asympt}
\end{align}for $ n\in\{0,1\}$, as  $ |z|\to\infty,-\pi<\arg z<\pi$.
 Furthermore,  these Hankel functions are related to modified Bessel functions $I_0(t),K_0(t),I_1(t),K_1(t) $ for $ t>0$ in the following manner:{\allowdisplaybreaks\begin{align}
\left\{\begin{array}{l}
H^{(1)}_n(+it)=-\frac{2i^{1-n}}{\pi} K_n(t)\\
H^{(1)}_n(-it)=\frac{2}{i^{n}} I_n(t)-\frac{2i^{n+1}}{\pi} K_n(t)\\
\end{array}\right.\text{ and  }\,\,\left\{\begin{array}{l}
H^{(2)}_n(+it)=2i^nI_n(t)+\frac{2i^{1-n}}{\pi}K_n(t) \\
H^{(2)}_n(-it)=\frac{2i^{n+1}}{\pi}K_n(t) \\
\end{array}\right.\label{eq:HnInKn}
\end{align}}where $ n\in\{0,1\}$.
\begin{enumerate}[leftmargin=*,  label={(\alph*)},ref=(\alph*),
widest=a, align=left]
\item
Armed with the aforementioned preparations, we  use two methods to evaluate the  contour integral\begin{align}
\int_{-i\infty}^{i\infty}H_n^{(2)}(2sz)[H_0^{(1)}(z)]^{2s}[H_0^{(1)}(z)H_0^{(2)}(z)]^{k-s}z^{2\ell+n-1}\D z,\label{eq:H_vert_int}
\end{align}which converges for  $ k\in\mathbb Z_{>1},
s\in\mathbb Z\cap[1,k],n\in\{0,1\},\ell\in\mathbb Z\cap\big[1,\left\lfloor\frac{k-n}{2}\right\rfloor\big]$. First, by the asymptotic behavior of cylindrical Hankel functions in \eqref{eq:H0H1_asympt}, we can close the contour to the right, and show that the integral vanishes. Second,  with the aid of     \eqref{eq:HnInKn}, we may recognize the contour  integral as a constant multiple of\begin{align}\begin{split}&\int_0^\infty[i^n\pi I_n(2st)+i^{1-n}K_n(2st)][-iK_0(t)]^{k+s}[\pi I_{0}(t)+iK_{0}(t)]^{k-s}t^{2\ell+n-1}\D t\\
&-(-1)^{n}\int_0^\infty i^{n+1}K_{n}(2st)[\pi I_{0}(t)-iK_{0}(t)]^{k+s}[iK_{0}(t)]^{k-s}t^{2\ell+n-1}\D t\end{split}
\end{align}For $ n=0$,  the last vanishing integral reduces to a sum rule\begin{align}\begin{split}&
\int_0^\infty[\pi I_0(2st)+iK_0(2st)]\sum_{j=0}^{k-s}{k-s\choose j}[\pi I_{0}(t)]^j[iK_{0}(t)]^{2k-j}t^{2\ell+n-1}\D t\\&-\int_0^\infty iK_{0}(2st)\sum_{j=0}^{k+s}{k+s\choose j}[\pi I_{0}(t)]^j[-iK_{0}(t)]^{2k-j}t^{2\ell+n-1}\D t=0,\end{split}
\end{align}which is equivalent to  $ \mathsf F_{k,s}^\ell((2s)^2)=0$. For $ n=1$, we can deduce  $\acute{\mathsf F}_{k,s}^\ell((2s)^2)=0$ instead.
\item
When   $ k\in\mathbb Z_{>1},
s\in\mathbb Z\cap[1,k+1],n\in\{0,1\},\ell\in\mathbb Z\cap\big[1,\left\lfloor\frac{k-n}{2}\right\rfloor\big]$, we can evaluate\begin{align}
\int_{-i\infty}^{i\infty}H_n^{(2)}((2s-1)z)[H_0^{(1)}(z)]^{2s-1}[H_0^{(1)}(z)H_0^{(2)}(z)]^{k+1-s}z^{2\ell+n-1}\D z=0
\end{align}by closing the contour to the right, while noting that the integrand has $ O(|z|^{-2})$ behavior for large $|z|$. This can be reinterpreted into the sum rules $ \mathsf G_{k,s}^\ell((2s-1)^2)=0$ and  $ \acute{\mathsf G}_{k,s}^\ell((2s-1)^2)=0$.
 \qedhere\end{enumerate}
\end{proof}

The sum rules in the last proposition will be essential to local analysis\footnote{Here, by ``local analysis'', we refer to asymptotic expansions for both sides  of  \eqref{eq:Sm_fit} (with appropriate real-analytic continuations if necessary) when $u$ approaches a threshold. Corollary \ref{cor:S3_comp} illustrates such local analysis in the case where  $ m=3$.} in \S\ref{subsec:threshold}, especially the determination of the blocks   $ \mathbf S_m^A$ [as given in \eqref{eq:SAmat_defn}] and $ \mathbf S_m^B=(-1)^{m+1}(\mathbf S_m^C)^{\mathrm T}$ [as given in \eqref{eq:SBmat_defn}] for the partitioned matrix  $  {{\mathbf S_m}}=\left(\begin{smallmatrix}\mathbf S_m^A &\mathbf  S_m^B \\
\mathbf S_m^C &\mathbf  S_m^D \\
\end{smallmatrix}\right)$ that fits into \eqref{eq:Sm_fit}.
To compute the block $ \mathbf  S_m^D$, and to complete the proof of Theorem \ref{thm:W_alg} in \S\ref{subsec:offshell_quad}, we will  need to investigate another  threshold $u\to0^+$ in the next  proposition.

In what follows, we define rational numbers  $ \mathsf H_{k,n}$ for $ k,n\in\mathbb Z_{>0}$ through the following asymptotic expansion (cf.~\cite[(66)]{Zhou2018ExpoDESY}):\begin{align}h_{k,N}(z)\colonequals
\left[\frac{\pi}{2}H^{(1)}_0(z)H^{(2)}_0(z)\right]^k-\sum^N_{n=1}\frac{\mathsf H_{k,n}}{z^{2n+k-2}}=O\left( \frac{1}{|z|^{2N+k}} \right),\label{eq:hmn_defn}
\end{align}where $ k,N\in\mathbb Z_{>0}$.  These rational numbers also turn up in the relations for generalized Crandall numbers (cf.~\cite[(65)]{Zhou2018ExpoDESY}): \begin{align}(-1)^{n+1}
\left(\frac{\pi}{2}\right)^k {\mathsf H_{k,n}}={}&\int_0^\infty \frac{[\pi I_0(t)+i K_{0}(t)]^{k}-[\pi I_0(t)-i K_{0}(t)]^{k}}{\pi i}[K_0(t)]^{k}t^{2n+k-3}\D t.\label{eq:Crandall_num}
\end{align}By the convention that  empty sums must vanish, we retroactively define $ h_{k,0}(z)=\big[\frac{\pi}{2}H^{(1)}_0(z)H^{(2)}_0(z)\big]^k$ and $ \mathsf H_{k,0}=0$. Furthermore, we introduce a function\begin{align}
\mathsf L_{k}(u,t)\colonequals {}&\sum _{n=1}^{\infty } \frac{\pi^{1-k} u^{n+\frac{k}{2}-1}t^{2n+k-3} }{2^{2 n-1}\left[ \Gamma \left(n+\frac{k}{2}\right)\right]^2}\equiv\frac{  u^{k/2} t^{k-1}}{  2\pi^{k-1}\left[\Gamma \left(\frac{k}{2}+1\right)\right]^2}{_1F_2}\left(\left.\begin{array}{c}
1 \\
\frac{k}{2}+1,\frac{k}{2}+1 \\
\end{array}\right|\frac{ut^2 }{4}\right).
\end{align} For each fixed $u\in(0,1)$, we have $\mathsf L_{k}(u,t)=O(1) $ as $ t\to0^+$, and $\mathsf L_{k}(u,t)=O(e^{\sqrt{u}t}/\sqrt{t}) $  as $ t\to\infty$, so the dominated convergence theorem allows us to  conclude that either side of  \begin{align}
\int_0^\infty \frac{[\pi I_0(t)+i K_{0}(t)]^{k}-[\pi I_0(t)-i K_{0}(t)]^{k}}{\pi i}[K_0(t)]^{k}\mathsf L_{k}(u,t)\D t={}&\sum _{n=1}^{\infty } \frac{\pi^{} (-1)^{n+1} \mathsf H_{k,n}u^{n+\frac{k}{2}-1}}{2^{2 n+k-1}\left[ \Gamma \left(n+\frac{k}{2}\right)\right]^2}
\end{align}represents a continuous function in the variable $u\in(0,1)$.\begin{proposition}[``Sum rules'' at threshold $ u\to0^+$]\label{prop:u0_sum}\begin{enumerate}[leftmargin=*,  label=\emph{(\alph*)},ref=(\alph*),
widest=a, align=left]
\item For $\ell\in\mathbb Z_{>0} $, we have\footnote{The asymptotic behavior of $ \mathcal{F}^\ell_{2k-1,2k-1}(u)$ and  $ \acute{\mathcal F}^\ell_{2k-1,2k-1}(u)$ was misstated as $ O(1/\sqrt{u})$ (irrespective of $\ell$) in \cite[Proposition 4.4]{Zhou2017BMdet}. The corrections here do not affect the  proof of  \cite[Proposition 4.4]{Zhou2017BMdet}, the latter of which only requires $ \lim_{u\to0^+}{u^{k/2}\mathcal{F}^\ell_{2k-1,2k-1}(u)}\log u=\lim_{u\to0^+}{u^{k/2}\sqrt{u}\acute{\mathcal F}^\ell_{2k-1,2k-1}(u)}\log u=0$ for $ \ell\in\mathbb Z\cap[1,k-1]$ to work properly. }{\allowdisplaybreaks\begin{align}
\mathcal F^\ell_{2k-1,2k-1}(u)=O\left( \frac{1}{(\sqrt{u})^{\max\{1,2\ell-k\}}} \right),\; \sqrt{u}\acute{\mathcal F}^\ell_{2k-1,2k-1}(u)=O\left( \frac{1}{(\sqrt{u})^{\max\{1,2\ell-k\}}} \right),\label{eq:bare_asympt}
\end{align}}as  $ u\to0^+$.

For each $k\in1+2\mathbb Z_{\geq0}$ and $ \ell\in\mathbb  Z_{>0}$, we have the following quantitative refinements of \eqref{eq:bare_asympt} for  $u\in(0,1)$: {\allowdisplaybreaks\begin{align}\begin{split}O_k^{\ell}(u)\colonequals {}&(2k+1)\mathcal{F}^\ell_{2k-1,1}(u)+\sum_{b=2}^{k}\frac{1+(-1)^{b-1}}{2(-1)^{\left\lfloor \frac{b-1}{2}\right\rfloor }}\binom{k}{b-1}{} \mathcal{F}^\ell_{2k-1,b}(u)+\sum_{b=k+1}^{2k-1}\frac{1+(-1)^{b-k}}{(-1)^{\left\lfloor \frac{b-k}{2}\right\rfloor }}\binom{k}{b-k+1}{} \mathcal{F}^\ell_{2k-1,b}(u)\\={}&4^\ell(uD^2+D^1)^{\ell}\int_0^\infty \frac{[\pi I_0(t)+i K_{0}(t)]^{k}-[\pi I_0(t)-i K_{0}(t)]^{k}}{(-1)^{\frac{k+1}{2}}2^{k}\pi i}[K_0(t)]^{k}\mathsf L_{k}(u,t)\D t;\end{split}\label{eq:Okl}\\\begin{split}
\acute O_k^\ell(u)\colonequals {}&(2k+1)\acute{\mathcal{F}}^\ell_{2k-1,1}(u)+\sum_{b=2}^{k}\frac{1+(-1)^{b-1}}{2(-1)^{\left\lfloor \frac{b-1}{2}\right\rfloor }}\binom{k}{b-1}{} \acute{\mathcal{F}}^\ell_{2k-1,b}(u)+\sum_{b=k+1}^{2k-1}\frac{1+(-1)^{b-k}}{(-1)^{\left\lfloor \frac{b-k}{2}\right\rfloor }}\binom{k}{b-k+1}{} \acute{\mathcal{F}}^\ell_{2k-1,b}(u)\\={}&4^\ell\sqrt{u}D^{1}(uD^2+D^1)^{\ell}\int_0^\infty \frac{[\pi I_0(t)+i K_{0}(t)]^{k}-[\pi I_0(t)-i K_{0}(t)]^{k}}{(-1)^{\frac{k+1}{2}}2^{k-1}\pi i}[K_0(t)]^{k}\mathsf L_{k}(u,t)\D t.\end{split}\tag{\ref{eq:Okl}$'$}\label{eq:O'kl}
\end{align}}

For each
 $k\in2\mathbb Z_{>0}$ and $ \ell\in\mathbb  Z_{>0}$, we have the following quantitative refinements of \eqref{eq:bare_asympt} in  the $ u\to0^+$ regime: {\allowdisplaybreaks\begin{align}\begin{split}E_k^{\ell}(u)\colonequals {}&\sum_{b=k+1}^{2k-1}\frac{1-(-1)^{b-k}}{(-1)^{\left\lfloor \frac{b-k}{2}\right\rfloor }}\binom{k}{b-k+1}{} \mathcal{F}^\ell_{2k-1,b}(u)+\frac{2\log\frac{\sqrt{u}}{2}}{\pi}\left[ (2k+1)\mathcal{F}^\ell_{2k-1,1}(u)+\vphantom{\binom{k}{b}}\right.\\&\left.+\sum_{b=k+1}^{2k-1}\frac{1+(-1)^{b-k}}{(-1)^{\left\lfloor \frac{b-k}{2}\right\rfloor }}\binom{k}{b-k+1}{} \mathcal{F}^\ell_{2k-1,b}(u) \right]-\frac{\log^{2}\frac{\sqrt{u}}{2}}{\pi^{2}}\sum_{b=2}^k \frac{1+(-1)^b}{(-1)^{\left\lfloor \frac{b}{2}\right\rfloor }} \binom{k}{b-1}\mathcal{F}^\ell_{2k-1,b}(u)\\={}&\mathscr E_k^\ell(u)+\sum_{n\in\mathbb Z\cap\left[1,\max\left\{0,\ell-\frac{k}{2}\right\}\right]}\frac{ 2^{2 (\ell- n-k)+1}\left[ \left(\ell-n-\frac{k}{2}\right)!\right]^2}{(-1)^{n+k/2}\pi u^{\ell-n+1-\frac{k}{2}}}\mathsf H_{k,n},\end{split}\label{eq:Ekl}\\\begin{split}
\acute E_k^{\ell}(u)\colonequals {}&\sum_{b=k+1}^{2k-1}\frac{1-(-1)^{b-k}}{(-1)^{\left\lfloor \frac{b-k}{2}\right\rfloor }}\binom{k}{b-k+1}{} \acute{\mathcal{F}}^\ell_{2k-1,b}(u)+\frac{2\log\frac{\sqrt{u}}{2}}{\pi}\left[ (2k+1)\acute{\mathcal{F}}^\ell_{2k-1,1}(u)+\vphantom{\binom{k}{b}}\right.\\&\left.+\sum_{b=k+1}^{2k-1}\frac{1+(-1)^{b-k}}{(-1)^{\left\lfloor \frac{b-k}{2}\right\rfloor }}\binom{k}{b-k+1}{} \acute{\mathcal{F}}^\ell_{2k-1,b}(u) \right]-\frac{\log^{2}\frac{\sqrt{u}}{2}}{\pi^{2}}\sum_{b=2}^k \frac{1+(-1)^b}{(-1)^{\left\lfloor \frac{b}{2}\right\rfloor }} \binom{k}{b-1}\acute{\mathcal{F}}^\ell_{2k-1,b}(u)\\={}&\frac{\acute{\mathscr E}_k^\ell(u)}{\sqrt{u}}+\sum_{n\in\mathbb Z\cap\left[1,\max\left\{0,\ell-\frac{k}{2}\right\}\right]}\frac{ 4^{\ell- n-k+1} \left(\ell-n-\frac{k}{2}\right)!\left(\ell-n-\frac{k}{2}+1\right)!}{(-1)^{n-1+k/2}\pi u^{\ell-n+\frac{3}{2}-\frac{k}{2}}}\mathsf H_{k,n},\end{split}\label{eq:E'kl}
\tag{\ref{eq:Ekl}$'$}\end{align}}where $ \mathscr E_k^\ell(u)$ and $ \acute{\mathscr E}_k^\ell(u)$ are holomorphic in a non-void $ \mathbb C$-neighborhood of $u=0$.

 Furthermore, if a function in $\Span_{\mathbb C}\{\mathcal F_{2k-1,b}(u),u\in(0,\varepsilon)|b\in\mathbb Z\cap(\{1\}\cup[k+1,2k-1])\} $ (where $ 0<\varepsilon<1$) extends to a  holomorphic function in a non-void $\mathbb C$-neighborhood of $u=0$, then it must be identically zero. \item If a function in $\Span_{\mathbb C}\{\mathcal F_{2k,b}(u),u\in(0,\varepsilon)|b\in\mathbb Z\cap(\{1\}\cup[k+2,2k])\} $ (where $ 0<\varepsilon<1$) extends to a  holomorphic function in a non-void $\mathbb C$-neighborhood of $u=0$, then it must be identically zero.\end{enumerate} \end{proposition}\begin{proof}\begin{enumerate}[leftmargin=*,  label={(\alph*)},ref=(\alph*),
widest=a, align=left]
\item From $ \sup_{t>0}t^{\varepsilon}I_0(t)K_0(t)<\infty$ for any $ \varepsilon\in(0,1]$, we can deduce that  (cf.~\cite[(2.27), (2.30)]{Zhou2017BMdet}\begin{align}
\int_0^\infty K_0(\sqrt{u}t)[I_0(t)K_0(t)]^kt^{2\ell-1}\D t={}&O\left(\int_0^\infty K_0(\sqrt{u}t)\D t\right)=O\left( \frac{1}{\sqrt{u}} \right),\\\int_0^\infty \sqrt{u}K_1(\sqrt{u}t)[I_0(t)K_0(t)]^kt^{2\ell}\D t={}&O\left(\int_0^\infty \sqrt{u}K_1(\sqrt{u}t)t\D t\right)=O\left( \frac{1}{\sqrt{u}} \right),
\end{align}when $ 2\ell-1\leq k$. Meanwhile, from Heaviside's integral formula \cite[\S13.21(8)]{Watson1944Bessel} for    $ \int_0^\infty K_0(t)t^{n-1}\D t$, we have
\begin{align}
\begin{split}&\frac{1}{2^{k}}\int_0^\infty K_0(\sqrt{u}t)t^{2\ell-1-k}\D t+\int_0^\infty K_0(\sqrt{u}t)\left\{[I_0(t)K_0(t)]^k-\frac{1}{(2t)^{k}}\right\}t^{2\ell-1}\D t\\={}&\frac{\left[\Gamma \left(\ell-\frac{k}{2}\right)\right]^2}{4^{k-\ell+1}(\sqrt{u})^{2\ell-k}} +O\left( \int_0^\infty K_0(\sqrt{u}t)t^{2\ell-3-k}\D t \right) \end{split}
\end{align}
  when $ 2\ell-1> k$, so the first half of \eqref{eq:bare_asympt} follows from induction on $\ell$. The second half of  \eqref{eq:bare_asympt} founds on a similar argument.

Recall the notation $ h_{k,N}(z)  $ for $k,N\in\mathbb Z_{>0} $ from \eqref{eq:hmn_defn} and set  $ h_{k,0}(z)=\big[\frac{\pi}{2}H^{(1)}_0(z)H^{(2)}_0(z)\big]^k$.  The following contour integral\begin{align}\begin{split}&
i(-1)^{\ell-1}\left( \frac{2}{\pi} \right)^{k+1}\int_0^\infty K_0(\sqrt{u}t)[K_0(t)]^k[K_0(t)-i\pi I_0(t)]^{k}t^{2\ell-1}\D t\\&-\sum_{n\in\mathbb Z\cap\left[1,\max\left\{0,\left\lfloor\ell-\frac{k-1}{2}\right\rfloor\right\}\right]}\frac{(2 i)^{2( \ell- n)+1-k} \left[\Gamma \left(\ell-n+1-\frac{k}{2}\right)\right]^2}{\pi(  \sqrt{u})^{2 (\ell-n+1)-k}}\mathsf H_{k,n}\\={}&\int_0^{i\infty}H_0^{(1)}(\sqrt{u}z)h_{k,\max\left\{0,\left\lfloor\ell-\frac{k-1}{2}\right\rfloor\right\}}(z)z^{2\ell-1}\D z\end{split}\label{eq:H0_u_to_0}
\end{align}can be identified with \begin{align}
\int_0^{\infty}H_0^{(1)}(\sqrt{u}x)h_{k,\max\left\{0,\left\lfloor\ell-\frac{k-1}{2}\right\rfloor\right\}}(x)x^{2\ell-1}\D x,\label{eq:H0_u_to_0_Wick}
\end{align}because Jordan's lemma allows us to rotate the contour of integration $90^\circ$ clockwise, from the positive $\I z$-axis to the positive $\R z$-axis, as in \cite[\S2]{Zhou2017PlanarWalks}.

When $k$ is odd, we can relate the real parts of the last two displayed equations
as follows:
{\allowdisplaybreaks\begin{align}\begin{split}&2^k (-1)^{\ell-1}\left[ (2k+1)\mathcal{F}^\ell_{2k-1,1}(u)-\frac{\IvKM(1,2k,2\ell-1|u)}{\pi^{k}}+\sum_{b=k+1}^{2k-1}\frac{1+(-1)^{b-k}}{(-1)^{\left\lfloor \frac{b-k}{2}\right\rfloor }}\binom{k}{b-k+1}{} \mathcal{F}^\ell_{2k-1,b}(u) \right]\\&-(-1)^{\frac{k-1}{2}}4^\ell(uD^2+D^1)^{\ell}\sum_{n\in\mathbb Z\cap\left[1,\max\left\{0,\ell-\frac{k-1}{2}\right\}\right]}\frac{\pi u^{n+\frac{k}{2}-1} (-1)^{n+1}}{2^{2 n+k-1}\left[ \Gamma \left(n+\frac{k}{2}\right)\right]^2}\mathsf H_{k,n}\\
={}&\int_0^{\infty}J_{0}(\sqrt{u}x)h_{k,\max\left\{0,\ell-\frac{k-1}{2}\right\}}(x)x^{2\ell-1}\D x,\end{split}\label{eq:k_odd_J0_int}
\end{align}}where $ J_0(z)=\frac{1}{2}[H_0^{(1)}(z)+H_0^{(2)}(z)]$ is the Bessel function of the first kind, satisfying $|J_0(x)|\leq 1$ for all real-valued $x$. Here, the function $ h_{k,\max\left\{0,\ell-\frac{k-1}{2}\right\}}(x)x^{2\ell-1}$ is bounded in the $ x\to0^+$ regime, and has $ O(x^{-2})$ behavior as $ x\to\infty$, so we can  invoke  the dominated convergence theorem to show that \begin{align}\lim_{u\to0^+}
\int_0^{\infty}J_{0}(\sqrt{u}x)h_{k,\max\left\{0,\ell-\frac{k-1}{2}\right\}}(x)x^{2\ell-1}\D x=\int_0^{\infty}h_{k,\max\left\{0,\ell-\frac{k-1}{2}\right\}}(x)x^{2\ell-1}\D x.
\end{align}After  rotating the contour of integration $90^\circ$ counterclockwise, we can identify the right-hand side of the last equation  with\begin{align}
\begin{split}&\R\int_0^{i\infty}h_{k,\max\left\{0,\ell-\frac{k-1}{2}\right\}}(z)z^{2\ell-1}\D z\\={}&(-1)^\ell\left( \frac{2}{\pi} \right)^{k}\R\int_0^{\infty}[K_0(t)]^k[K_0(t)-i\pi I_0(t)]^{k}t^{2\ell-1}\D t\\={}&2^k (-1)^\ell\lim_{u\to0^+}\left[\frac{\IvKM(1,2k,2\ell-1|u)}{\pi^{k}}+\sum_{b=2}^{k}\frac{1+(-1)^{b-1}}{2(-1)^{\left\lfloor \frac{b-1}{2}\right\rfloor }}\binom{k}{b-1}{} \mathcal{F}^\ell_{2k-1,b}(u)\right].\end{split}
\end{align}  So far, we have already arrived at a precursor to \eqref{eq:Okl}, namely $\lim_{u\to0^+} [O^\ell_k(u)-\smash{\underset{\widetilde{}}O}^\ell_k(u)]=0$, where $ \smash{\underset{\widetilde{}}O}^\ell_k(u)$ denotes the right-hand side of  \eqref{eq:Okl}. This precursor can be rewritten as \begin{align}\lim_{u\to0^+} (uD^2+D^1)^{\ell-1}[O^1_k(u)-\smash{\underset{\widetilde{}}O}^1_k(u)]=0,\label{eq:Okl_deriv0}\end{align} for all  $ \ell\in\mathbb Z_{>0}$, since we have $ O^\ell_k(u)=4^{\ell-1}(uD^2+D^1)^{\ell-1}O^1_k(u)$ by the Bessel differential equation.
Now that $O^1_k(u),u\in(0,\varepsilon_{k})$ is equal to a generalized power series  (by Corollary \ref{cor:reg_sing}) for a suitably chosen   $ \varepsilon_{k}\in(0,1)$, and so is $ \smash{\underset{\widetilde{}}O}^1_k(u),u\in(0,1)$ (by its construction and the dominated convergence theorem), the vanishing derivatives in \eqref{eq:Okl_deriv0} bring us $ O^1_k(u)= \smash{\underset{\widetilde{}}O}^1_k(u),u\in(0,\varepsilon_{k})$. By real-analytic continuations and differentiations, we can establish   \eqref{eq:Okl} and   \eqref{eq:O'kl}  in their entirety.

 When $k$ is even, we can identify the imaginary parts of  \eqref{eq:H0_u_to_0} and \eqref{eq:H0_u_to_0_Wick} as follows:{\allowdisplaybreaks\begin{align}\begin{split}
&2^k (-1)^{\ell-1}\left[ \frac{2\IKvM(0,2k+1,2\ell-1|u)}{\pi^{k+1}}-\sum_{b=k+1}^{2k-1}\frac{1-(-1)^{b-k}}{(-1)^{\left\lfloor \frac{b-k}{2}\right\rfloor }}\binom{k}{b-k+1}{} \mathcal{F}^\ell_{2k-1,b}(u) \right]\\{}&-\sum_{n\in\mathbb Z\cap\left[1,\max\left\{0,\ell-\frac{k}{2}\right\}\right]}\frac{(-4)^{ \ell- n}2^{1-k} \left[ \left(\ell-n-\frac{k}{2}\right)!\right]^2}{(-1)^{k/2}\pi u^{\ell-n+1-\frac{k}{2}}}\mathsf H_{k,n}\\={}&\int_0^{\infty}Y_{0}(\sqrt{u}x)h_{k,\max\left\{0,\ell-\frac{k}{2}\right\}}(x)x^{2\ell-1}\D x,\end{split}
\end{align}}where $ Y_0(z)=\frac{1}{2i}[H_0^{(1)}(z)-H_0^{(2)}(z)]$ is the Bessel function of the second kind. Meanwhile, the real parts of  \eqref{eq:H0_u_to_0} and \eqref{eq:H0_u_to_0_Wick}  leave us{\allowdisplaybreaks\begin{align}\begin{split}&2^k (-1)^{\ell-1}\left[ (2k+1)\mathcal{F}^\ell_{2k-1,1}(u)-\frac{\IvKM(1,2k,2\ell-1|u)}{\pi^{k}}+\sum_{b=k+1}^{2k-1}\frac{1+(-1)^{b-k}}{(-1)^{\left\lfloor \frac{b-k}{2}\right\rfloor }}\binom{k}{b-k+1}{} \mathcal{F}^\ell_{2k-1,b}(u) \right]\\
={}&\int_0^{\infty}J_{0}(\sqrt{u}x)h_{k,\max\left\{0,\ell-\frac{k}{2}\right\}}(x)x^{2\ell-1}\D x,\end{split}
\end{align}}which is subtly different from \eqref{eq:k_odd_J0_int}. Here, we point out that \begin{align}\begin{split}&
\IKvM(0,2k+1,2\ell-1|u)+\IvKM(1,2k,2\ell-1|u)\log\frac{\sqrt{u}}{2}\\={}&\int_0^\infty\left[ K_{0} (\sqrt{u}t)+I_{0}(\sqrt{u}t)\log\frac{\sqrt{u}}{2}\right][K_0(t)]^{2k}t^{2\ell-1}\D t\end{split}
\end{align}equals its own Maclaurin series for  $ u\in(0,(2k)^2)$, because the dominated convergence theorem allows us to perform termwise integration on  the  convergent series   \cite[\S3.71(14)]{Watson1944Bessel}\begin{align}
K_{0}(\sqrt{u}t)+I_0(\sqrt{u}t)\log\frac{\sqrt{u}t}{2}=\sum_{m=0}^\infty\frac{\psi^{(0)}(m+1)}{(m!)^2}\left( \frac{ut^{2}}{4} \right)^{m},\label{eq:K0_I0log}
\end{align} where $ \psi^{(0)}(m+1)\colonequals \frac1{m!}\int_0^\infty x^m e^{-x}\log x\D x$ goes like $\log m+O\left( \frac{1}{m} \right)$ as $m\to\infty$.
 As we may recall from \cite[\S3.51]{Watson1944Bessel}, the following limit\begin{align}
\lim_{u\to0^+}\left[ Y_{0}(\sqrt{u}x)-\frac{2}{\pi}J_{0}(\sqrt{u}x)\log\frac{\sqrt{u}}{2} \right]=\frac{2(\gamma_0+\log x)}{\pi}
\end{align} holds for all $x>0 $, where $ \gamma_0=-\psi^{(0)}(1)$ is the Euler--Mascheroni constant.   Meanwhile, we note that the function $
\widetilde h_{k,\ell}(x)x^{2\ell-1}\colonequals h_{k,\max\left\{0,\ell-\frac{k}{2}\right\}}(x)x^{2\ell-1}-\mathsf H_{k,{\max\left\{0,\ell+1-\frac{k}{2}\right\}}}h_{2,0}(x)x
$ is bounded in the $ x\to0^+$ regime, and has $ O(x^{-3})$ behavior as $ x\to\infty$, so the dominated convergence theorem  allows us to conclude that\begin{align}
\lim_{u\to0^+}\int_0^{\infty}\left[ Y_{0}(\sqrt{u}x)-\frac{2}{\pi}J_{0}(\sqrt{u}x)\log\frac{\sqrt{u}}{2} \right]\widetilde h_{k,\ell}(x)x^{2\ell-1}\D x=\frac{2}{\pi}\int_0^{\infty}\widetilde h_{k,\ell}(x)z^{2\ell-1}(\gamma_0+\log x)\D x
\end{align} is a finite real number.  To compensate for the difference $ h_{k,\max\left\{0,\ell-\frac{k}{2}\right\}}(x)x^{2\ell-1}-\widetilde h_{k,\ell}(x)x^{2\ell-1}$, we need  the explicit knowledge that (cf.~\cite[Lemma 3.2]{Zhou2018LaportaSunrise}){\allowdisplaybreaks\begin{align}
\begin{split}\mathcal{F}_{3,1}(u)={}&\frac{\sqrt{3}}{20\pi^{1/2}}\frac{1}{2\pi i }\int_{\frac14-i\infty}^{\frac14+i\infty}\frac{ \Gamma \left(\frac{1}{3}-s\right) \Gamma \left(\frac{2}{3}-s\right) \left[\Gamma (s)\right]^2}{ (4-u)\Gamma (1-s) \Gamma \left(s+\frac{1}{2}\right)}\left[ \frac{108 u}{(4-u)^3} \right]^{-s}\D s\\={}&-\frac{\log \frac{u}{(4-u)^3}}{10 (4-u)}+\frac{O\left( \frac{u}{(4-u)^3} \right)}{4-u},\end{split}\label{eq:F31log}\\\begin{split}\mathcal{F}_{3,3}(u)={}&\frac{\sqrt{3}}{8 \pi ^{5/2} }\frac{1}{2\pi i }\int_{\frac14-i\infty}^{\frac14+i\infty}\frac{\Gamma \left(\frac{1}{3}-s\right) \Gamma \left(\frac{1}{2}-s\right) \Gamma \left(\frac{2}{3}-s\right) [\Gamma (s)]^3}{4-u}\left[ \frac{108 u}{(4-u)^3} \right]^{-s}\D s-\mathcal{F}_{3,2}(u)\\={}&\frac{\pi ^2+3 \log ^2\frac{u}{(4-u)^3}}{24 \pi  (4-u)}+\frac{O\left( \frac{u}{(4-u)^3} \right)}{4-u}\end{split}\label{eq:F33log2}
\end{align}}hold in the $u\to0^+$ regime (where the logarithmic terms are attributed to residues at $s=0$), so as to  deduce\begin{align}\begin{split}&
\mathsf H_{k,{\max\left\{0,\ell+1-\frac{k}{2}\right\}}}\int_0^{\infty}\left[ Y_{0}(\sqrt{u}x)-\frac{2}{\pi}J_{0}(\sqrt{u}x)\log\frac{\sqrt{u}}{2} \right] h_{2,0}(x)x\D x\\={}&\mathsf H_{k,{\max\left\{0,\ell+1-\frac{k}{2}\right\}}}\left\{-8\mathcal{F}_{3,3}(u)-\frac{40}{\pi}\mathcal{F}_{3,1}(u)\log\frac{\sqrt{u}}{2}+\vphantom{\left(\frac{2}{\pi}\right)^3}\right.\\&\left.+\left( \frac{2}{\pi} \right)^3\int_0^\infty \left[ K_{0}(\sqrt{u}t)+I_{0}(\sqrt{u}t)\log\frac{\sqrt{u}}{2} \right] [K_{0}(t)]^{4}t\D t\right\}\\={}&\frac{ 2^{k}(-1)^{\ell}\log ^2\frac{\sqrt{u}}{2}}{ \pi^{2} }\sum _{b=2}^k \frac{1+(-1)^b}{(-1)^{\left\lfloor \frac{b}{2}\right\rfloor }} \binom{k}{b-1}\mathcal{F}^\ell_{2k-1,b}(0)+O(1).\end{split}
\end{align}Here, in the last step, we have combined  \eqref{eq:Crandall_num} with the generalized Bailey--Borwein--Broadhurst--Glasser sum rule \cite[(5)]{Zhou2018ExpoDESY} into the following form for
 $k\in2\mathbb Z_{>0}$ and $ \ell\in\mathbb  Z_{>0}$:\begin{align}
(-1)^{\ell-\frac{k}{2}}\mathsf H_{k,{\max\left\{0,\ell+1-\frac{k}{2}\right\}}}=\int_0^\infty \frac{[\pi I_0(t)+i K_{0}(t)]^{k}-[\pi I_0(t)-i K_{0}(t)]^{k}}{\pi i}\left[\frac{2K_0(t)}{\pi}\right]^{k}t^{2\ell-1}\D t.
\end{align}  Let  $ \smash{\underset{\widetilde{}}E}^\ell_k(u)$ be the finite sum on the right-hand side of \eqref{eq:Ekl}, then the efforts so far lead us to $ E^\ell_k(u)-\smash{\underset{\widetilde{}}E}^\ell_k(u)=O(1)$ as $ u\to0^+$.  Before proceeding further, we note that the sequence\begin{align}
\mathsf b_{n}(u,t)\colonequals \left[t^{2n}-4^n(uD^2+D^1)^n\right]\left[ K_{0}(\sqrt{u}t)\log\frac{\sqrt{u}}{2}+\frac{1}{2}I_{0}(\sqrt{u}t)\log^{2}\frac{\sqrt{u}}{2}\right]
\end{align}satisfies a recursion \begin{align}
4(uD^2+D^1)\mathsf b_{n}(u,t)-\mathsf b_{n+1}(u,t)=\frac{t^{2n}}{u}\left\{I_0(\sqrt{u}t)+2 \sqrt{u} t\left[I_1(\sqrt{u}t)\log \frac{\sqrt{u}}{2}- K_1(\sqrt{u}t)\right]\right\}
\end{align}  with the initial condition $ \mathsf b_{0}(u,t)=0$, and \cite[\S3.71(15)]{Watson1944Bessel} \begin{align}
 \sqrt{u} t\left[I_1(\sqrt{u}t)\log \frac{\sqrt{u}t}{2}- K_1(\sqrt{u}t)\right]=-1+\frac{1}{2}\sum_{m=0}^\infty\frac{\psi^{(0)} (m+1)+\psi ^{(0)}(m+2) }{m! (m+1)!}\left(\frac{ut^2 }{4}\right)^{m+1},
\end{align}  so we can show inductively  that \begin{align}
\mathsf b_{n}(u,t)-\sum _{m=1}^n \frac{4^{m-1} [(m-1)!]^2 }{u^m}t^{2 (n-m)}=
\sum_{m=0}^\infty \mathsf a_{n,m}(t) u^m\label{eq:b_n_Laurent}\end{align}  defines a holomorphic function in $u\in\mathbb C$ for each fixed $t\in(0,\infty)$ and $ n\in\mathbb Z_{\geq0}$. Therefore, we have $ E^\ell_k(u)-\smash{\underset{\widetilde{}}E}^\ell_k(u)=4^{\ell-1}(uD^2+D^1)^{\ell-1}E^1_k(u)+O(1)$ as $ u\to0^+$. By Corollary \ref{cor:reg_sing}, there exists a suitable $ \varepsilon_k\in(0,1)$  such that the function  $ E^1_k(u),u\in(0,\varepsilon_k)$ equals a generalized power series in (fractional) powers of $u$ and $\log u$; such a series is in fact a Maclaurin series, the coefficients of which are computable from the finite numbers $\lim_{u\to0^+}(uD^2+D^1)^{\ell-1}E^1_k(u) $ for  $\ell\in\mathbb Z_{>0}$. This reveals $ E^1_k(u)=\mathscr E^1_k(u)$ as a holomorphic function in a non-void $ \mathbb C$-neighborhood of $u=0$. The full version of \eqref{eq:Ekl}  then follows from repeated applications of the Bessel differentiation operator $ uD^2+D^1$ and invocations of the Laurent expansion for $ \mathsf b_n(u,t)$, as given in \eqref{eq:b_n_Laurent}. To prove \eqref{eq:E'kl}, we differentiate  \eqref{eq:Ekl}  while  noting that\begin{align}\begin{split}&2uD^1\left[ K_{0}(\sqrt{u}t)\log\frac{\sqrt{u}}{2}+\frac{1}{2}I_{0}(\sqrt{u}t)\log^{2}\frac{\sqrt{u}}{2}\right]-\sqrt{u}t\left[ -K_{1}(\sqrt{u}t)\log\frac{\sqrt{u}}{2}+\frac{1}{2}I_{1}(\sqrt{u}t)\log^{2}\frac{\sqrt{u}}{2}\right]
\\={}&K_{0}(\sqrt{u}t)+I_{0}(\sqrt{u}t)\log\frac{\sqrt{u}}{2}\end{split}\end{align}extends to a holomorphic function of $ u\in\mathbb C$.

According to \eqref{eq:K0_I0log}, we know that $ \pi\mathcal{F}_{2k-1,1}(u)+\frac{2k}{2k+1}\mathcal{F}_{2k-1,2}(u)\log\frac{\sqrt{u}}{2}$ and $ \pi\mathcal{F}_{2k-1,b}(u)+\linebreak\mathcal{F}_{2k-1,b-k+2}(u)\log\frac{\sqrt{u}}2$ for  $b\in\mathbb Z\cap[k+1,2k-2]$ are holomorphic in a non-void $ \mathbb C$-neighborhood of $ u=0$. From \eqref{eq:Okl}  and \eqref{eq:Ekl}, we know that the generalized power series expansion for $ \mathcal{F}_{2k-1,2k-1}(u)$ additionally involve $ u^{\mathbb Z+\frac{1}{2}}$ or $ u^{\mathbb Z}\log^2\frac{\sqrt{u}}{2}$ terms, according as $k$ is odd or even.

If a linear combination of these functions [say, $ c_1 \pi ^{k}\mathcal F_{2k-1,1}(u)+\sum_{b=k+1}^{2k-1}c_{b}^{\vphantom1}\pi^{2k-b}\mathcal{F}_{2k-1,b}(u)$] is  holomorphic near the origin, then it should not involve  $ u^{\mathbb Z+\frac{1}{2}}$ or $ u^{\mathbb Z}\log^2\frac{\sqrt{u}}{2}$ terms,  hence $c_{2k-1}=0$. Furthermore, we require that in the generalized power series expansion for    $ c_1 \pi ^{k}\mathcal F_{2k-1,1}(u)+\sum_{b=k+1}^{2k-2}c_{b}^{\vphantom1}\pi^{2k-b}\linebreak\mathcal{F}_{2k-1,b}(u)$, the coefficients for $ u^n\log\frac{\sqrt{u}}{2},n\in\mathbb Z\cap[0,k-2]$ should all vanish, that is,\begin{align}
\frac{2kc_{1}\pi ^{k-1} }{2k+1}\mathcal F_{2k-2,1}(1)+\sum_{b=k+1}^{2k-2}\frac{c_{b}^{\vphantom1}\mathcal F_{2k-2,b-k+1}(1)}{\pi^{b-2k+1}}=0,\quad\ell\in\mathbb Z\cap[1,k-1].
\end{align} Since $ \det\mathbf N_{k-1}=\det(\pi^{k+\frac{1}{2}-b}\mathcal F_{2k-2,b}^\ell(1))_{1\leq b,\ell\leq k-1}\neq0$, we must have $ c_b=0$ for $ b\in\mathbb Z\cap(\{1\}\cup[k+1,2k-2])$.

 \item In a non-void $ \mathbb C$-neighborhood of $ u=0$, we have holomorphic functions  $ \pi\mathcal{F}_{2k,1}(u)+\frac{2k+1}{2k+2}\mathcal{F}_{2k,2}(u)\log\frac{\sqrt{u}}{2}$ and $ \pi\mathcal{F}_{2k,b}(u)+\mathcal{F}_{2k,b-k+1}(u)\log\frac{\sqrt{u}}2$ for  $b\in\mathbb Z\cap[k+2,2k]$. The conclusion follows from    the coefficients for $ u^n\log\frac{\sqrt{u}}{2},n\in\mathbb Z\cap[0,k-1]$ and the fact that $ \det\mathbf M_k=\det(\pi^{k+1-b}\mathcal F_{2k-1,b}^\ell(1))_{1\leq b,\ell\leq k}\neq0$. \qedhere\end{enumerate}\end{proof}

\begin{corollary}[Determination of $ \mathbf S_3$]\label{cor:S3_comp}The matrix\begin{align}
\mathbf S_3=\begin{pmatrix*}[r]
 -25 &  &  \\
  & -\frac{4}{3} & 4 \\
  & 4 &  \\
\end{pmatrix*}
\end{align}fits into a special case of \eqref{eq:Sm_fit}, namely\begin{align}
\frac{\sqrt{|\mathcal L_3(u)|} }{\Lambda_3}\begin{pmatrix*}[r]W\big[\mathcal{F}_{3,2}(u),\mathcal{F}_{3,3}(u)\big]\\[5pt]
- W\big[\mathcal{F}_{3,1}(u),\mathcal{F}_{3,3}(u)\big] \\[5pt]
 W\big[\mathcal{F}_{3,1}(u),\mathcal{F}_{3,2}(u)\big] \\
\end{pmatrix*}={}&\mathbf S_3\begin{pmatrix*}[r]\mathcal{F}_{3,1}(u) \\[5pt]
\mathcal{F}_{3,2}(u) \\[5pt]
\mathcal{F}_{3,3}(u) \\
\end{pmatrix*}\label{eq:cof_Omega3}
\end{align}for  $|\mathcal L_3(u)|=u^2(4-u)(16-u),u\in(0,4)  $ and  $ \Lambda_3=\frac{1}{20}$. \end{corollary}\begin{proof}During the proof  of \cite[Proposition 5.5]{Zhou2017BMdet}, we have already pointed out that the left-hand side of \eqref{eq:cof_Omega3} is annihilated by $\widetilde L_3$ (which is a special case of Proposition~\ref{prop:Vanhove_dual} below), and have evaluated the first row of $\mathbf S_3$.

To compute the last row of   $ {\mathbf S}_3$, simply check that $ \IvKM(2,3;1|u)=\IKM(1,3;1)[1+O(u^{2})]=\frac{\pi^{2}}{16}[1+O(u^{2})]$ \cite[(55)]{BBBG2008} as well as  \begin{align}\begin{split}{}&\lim_{u\to0^+}
\sqrt{|\mathcal L_3(u)|} W\big[\mathcal{F}_{3,1}(u),\mathcal{F}_{3,2}(u)=\lim_{u\to0^+}\frac{32}{5\pi^3}\det\begin{pmatrix*}[r]D^0\IKvM(1,4;1|u) & D^0\IvKM(2,3;1|u) \\
uD^1\IKvM(1,4;1|u) & uD^1\IvKM(2,3;1|u) \\
\end{pmatrix*}\\={}&\lim_{u\to0^+}\frac{32}{5\pi^3}\det\begin{pmatrix*}[r]-\frac{\pi^{2}\log u}{32} +O(u\log u)& \frac{\pi^{2}}{16} +O(u^2)\\
-\frac{\pi^{2}}{32}+O(u\log u) & O(u) \\
\end{pmatrix*}=\frac{\pi}{80}=\frac{1}{5}\mathcal{F}_{3,2}(0).\end{split}\end{align}This accounts for the only  non-vanishing element in the  last row of  $ {\mathbf S}_3$, since any non-trivial linear combination of   $ \mathcal{F}_{3,1}(u)=-\frac{\log u}{40}[1+O(u)]$  (cf.~\cite[(55)]{BBBG2008} or \eqref{eq:F31log} above) and $\mathcal{F}_{3,3}(u)=\frac{\log^{2} u}{32\pi}+O(\log u)$ (cf.~\cite[(5.43)]{Zhou2017BMdet} or \eqref{eq:F33log2} above) in the $ u\to0^+$ regime is incompatible with such a finite limit.
[To  argue  for $ \mathbf S_{2k-1},k\in\mathbb Z_{>2}$, one will need sum rules \eqref{eq:Okl} and \eqref{eq:Ekl} in Proposition \ref{prop:u0_sum}, instead of the explicit formulae like  \eqref{eq:F31log} and  \eqref{eq:F33log2}.]

To compute the second row of  $ {\mathbf S}_3$, we quote in advance the symmetry ${\mathbf S}^{\vphantom{\mathrm T}}_3={\mathbf S}_3 ^{{\mathrm T}}$ from Corollary \ref{cor:VvSs}, which enables us to set up an equation   $ -\sqrt{|\mathcal L_3(u)|} W\big[\mathcal{F}_{3,1}(u),\mathcal{F}_{3,3}(u)\big]=c\mathcal{F}_{3,2}(u)+\frac{1}{5}\mathcal{F}_{3,3}(u)$ for a certain constant $c$. We observe that   $\lim_{u\to4^-} \sqrt{|\mathcal L_3(u)|} W\big[\mathcal{F}_{3,1}(u),\mathcal{F}_{3,3}(u)\big]=0$  (since  $ W\big[\mathcal{F}_{3,1}(u),\mathcal{F}_{3,3}(u)\big]$ is holomorphic in a neighborhood of $u=4$) and that $ \mathcal{F}_{3,2}(4)-3\mathcal{F}_{3,3}(4)=0$ follows from the sum rule $ \I\mathsf F_{2,1}^1(4)=0$,  a special case of Proposition~\ref{prop:u_sq}(a). Thus,  we must have
$c=-\frac{1}{15}$.
\end{proof}
\section{$ \mathbf W^\star$ algebra\label{sec:AdjW_alg}}The main purpose of this section is to construct analogs of Corollary \ref{cor:S3_comp} for matrices  of all sizes, and explore their consequences. In other words, for each $ m\in\mathbb Z_{>2}$, we will build a matrix $ \mathbf S_m\in\mathbb Q^{m\times m}$ that fits into  \eqref{eq:Sm_fit}---the rational entries of $ \mathbf S_m$ enable us to express the Wro\'nskian determinant  $ W[f_1(u),\dots,f_{m-1}(u)]$ as a $ \mathbb Q$-linear combination of $\{ |\mathcal L_m(u)|^{(2-m)/2}\mathcal F_{m,j}(u)|j\in\mathbb Z\cap[1,m]\}$, whenever $ \{f_n|n\in\mathbb Z\cap[1,m-1]\}$ is a subset of $ \{\mathcal F_{m,j}|j\in\mathbb Z\cap[1,m]\}$.

This master plan decomposes into several subtasks to be completed in the subsections to follow.

In \S\ref{subsec:cofactor}, we study  algebraically all the cofactors of the Wro\'nskian matrix $\mathbf  W[h_1(u),\dots,h_m(u)]$,
 where $h_1(u),\dots,h_m(u)$ are annihilated by a differential operator $\widetilde L_{m} =\sum_{j=0}^m\ell_{m,j}(u)D^j=(-1)^m\widetilde L_{m} ^*$, with  smooth coefficients $ \ell_{m,j}(u)$.\footnote{Being indifferent to detailed structures of the polynomials $ \ell_{m,j}(u)$, these algebraic mechanisms extend naturally to differential systems other than Vanhove's operators. It is our hope that such extensions will enable us to construct other types of quadratic relations among  periods of connections \cite{FresanSabbahYu2020a}, going beyond the  cases of Bessel moments treated in \cite{FresanSabbahYu2020b} and the current work.}  The algebraic properties of these cofactors  account for not only the Vanhove matrix  $ \mathbf  V_{m}
(u)$ in Theorem \ref{thm:W_alg}, but also the parity relation $ \mathbf S^{}_m=(-1)^{m+1}\mathbf S^{\mathrm T}_m$.

In \S\ref{subsec:threshold}, we scrutinize  the adjugate matrix for  $\mathbf  W[\mathcal F_{m,1}(u),\dots,\mathcal F_{m,m}(u)]$ analytically, by revisiting some themes in \S\ref{sec:W_alg}, namely Bessel differential equations and sum rules for Bessel moments. In particular, we will show that the local expansions of certain Wro\'nskian cofactors and certain linear combinations of  $\mathcal F_{m,1}(u),\dots,\mathcal F_{m,m}(u)$ match each other, when $u$ approaches a threshold value.

In \S\ref{subsec:offshell_quad}, we prove Theorem \ref{thm:W_alg} in its entirety, by carefully handling real-analytic continuations of Wro\'nskian cofactors  across threshold values.
In addition to proving the representation of $ \mathbf S_{m}$ given by  Theorem \ref{thm:W_alg}, we will also explore the recursive structures of the combinatorial sums therein. Such combinatorial analysis will prepare us for the inversion of  $ \mathbf S_{m}$  later in \S\ref{subsec:onshell_quad}.

Simply put, in \S\ref{subsec:cofactor} we settle the qualitative viability
of   \eqref{eq:Sm_fit} through differential equations, and augment it to $ \mathbf W_m^\star(u)=\frac{\Lambda_m}{|\mathcal L_m(u)|^{(m-2)/2}}\mathbf S_{m}^{}\mathbf{W}_{m}^{\mathrm T}(u)\mathbf V^{}_{m}
(u)$; in \S\ref{subsec:threshold} we perform quantitative analysis on
both sides of   \eqref{eq:Sm_fit} near thresholds, preparing us for a complete calculation of   $ \mathbf S_{m}$  in \S\ref{subsec:offshell_quad}, as well as an off-shell quadratic relation that descends from $ \mathbf W_m^\star(u)=\frac{\Lambda_m}{|\mathcal L_m(u)|^{m/2}}[\mathbf W_m^{}(u)]^{-1}$.

\subsection{Cofactors of Wro\'nskians\label{subsec:cofactor}}The adjugate $ \mathbf W^\star$ of a Wro\'nskian matrix $\mathbf W=\mathbf W[f_1(u),\dots,f_m(u)]$ is the transpose of its cofactor matrix $ \cof \mathbf W$. Each entry in the last column of  $ \mathbf W^\star=(\cof\mathbf W)^{\mathrm T}$ is expressible as a Wro\'nskian determinant $ W[f_{i_1}(u),\dots,f_{i_{m-1}}(u)]$, where $\{i_1,\dots,i_{m-1}\}\subset\mathbb Z\cap[1,m] $.


%
\begin{proposition}[Vanhove duality]\label{prop:Vanhove_dual} For any smooth differential operator $ \widetilde L_m=\mathcal L_m(u)D^m+\cdots$ (omitting lower order terms)  satisfying the parity relation $ \widetilde L_m^*=(-1)^m\widetilde L_m$, and $ \{f_n(u),u\in I|n\in\mathbb Z\cap[1,m-1]\}\subset C^\infty(I)\cap \ker \widetilde L_m $ for an open interval $I$,  we have \begin{align}
\widetilde L_{m} (|\mathcal L _{m}(u)|^{(m-2)/2}W[f_{1}(u),\dots,f_{m-1}(u)])=0\label{eq:Vanhove_dual_f}
\end{align} on the same open interval.

In particular, this is true when  $ \widetilde L_m$ is the $m$-th order Vanhove operator (defined in Proposition \ref{prop:Vanhove_Verrill}) with parity $ (-1)^m$ (proved in Proposition \ref{prop:parity_VL}), and $\mathcal L_m(u) $ is the polynomial specified by \eqref{eq:Lm_u_defn}.

\end{proposition}\begin{proof}Without loss of generality, we may assume that the functions $ f_1,\dots,f_{m-1}$  are linearly independent. We may further assume that there exists a function $ f_m$, such that $ \Span_{\mathbb C}\{f_j(u),u\in I|j\in\mathbb Z\cap[1,m]\}=C^\infty(I)\cap\ker\widetilde L_{m}$.

Using Ince's notation \cite[\S5.21, p.~120]{Ince1956ODE} for Wro\'nskian determinants \begin{align}
\varDelta_0=1;\quad\varDelta_r=W[f_{1}(u),\dots,f_{r}(u)]\text{ where }r\in\mathbb Z\cap[1,m],
\end{align}we have\begin{align}
(-1)^{m}\frac{\widetilde L_{m}f(u)}{\mathcal L_{m}(u)}={}&\frac{W[f(u),f_1(u),\dots,f_m(u)]}{\varDelta_{m}}
=(-1)^{m}\frac{\varDelta_{m}}{\varDelta_{m-1}}\frac{\D }{\D u}\frac{\varDelta_{m-1}^{2}}{\varDelta_{m}\varDelta_{m-2}}\frac{\D }{\D u}\cdots\frac{\D }{\D u}\frac{\varDelta_{1}^2}{\varDelta_2\varDelta_0}\frac{\D }{\D u}\frac{\varDelta_0f(u)}{\varDelta_{1}}.\end{align}

From the combinatorial constraint \eqref{eq:sublead} that directly descends from the parity
relation $ \widetilde L_m^*=(-1)^m\widetilde L_m$, we conclude that  $ \widetilde L_m=\mathcal L_m(u)D^m+\frac{m}{2}\frac{\D \mathcal L_m(u)}{\D u}D^{m-1}\cdots$   [cf.~\eqref{eq:VL_m}], up to sub-leading order, so $ \varDelta_{m} $ is proportional to $ |\mathcal L_m(u)|^{-m/2}$. In view of  the linear independence of  $ f_1,\dots,f_m$, we know that  $ \varDelta_{m} =C|\mathcal L_m(u)|^{-m/2},u\in I$  for a non-vanishing constant $C$. Appealing again to the parity relation $ \widetilde L_m^*=(-1)^m\widetilde L_m$, we subsequently have\begin{align}
\widetilde L_{m}f(u)=\frac{C\varDelta_0}{\varDelta_{1}}\frac{\D }{\D u}\frac{\varDelta_{1}^2}{\varDelta_2\varDelta_0}\frac{\D }{\D u}\cdots\frac{\D }{\D u}\frac{\varDelta_{m-1}^{2}}{\varDelta _m\varDelta_{m-2}}\frac{\D }{\D u}\frac{\mathcal L_m(u)f(u)}{\varDelta_{m-1}|\mathcal L_m(u)|^{m/2}},
\end{align}so \eqref{eq:Vanhove_dual_f} follows immediately.
\end{proof}\begin{remark}
The Vanhove duality relations for  Vanhove's operators $ \widetilde L_3$  and $ \widetilde L_4$ (see Table \ref{tab:Vanhove_Lm})  have already appeared in \cite[(5.34) and (5.67)]{Zhou2017BMdet}.\eor\end{remark}

In the next proposition, we show that  all the cofactors of $   \mathbf W_{m}$   are expressible via linear combinations of fundamental solutions to Vanhove's differential equation, together with their derivatives.
\begin{proposition}[Vanhove matrix $ \mathbf V_m
(u)$]\label{prop:Vv}Fix an integer $ m\in\mathbb Z_{>1}$. For a smooth differential operator $\widetilde L_{m} =\sum_{j=0}^m\ell_{m,j}(u)D^j=(-1)^m\widetilde L_{m} ^*$ where $ \ell_{m,j}\in C^\infty(I)$, consider    $f_1,\dots,f_{m-1}\in{}C^{\infty}(I)\cap\ker\widetilde L_{m} $.  Define\begin{align}
F_r(u)\colonequals {}&|\ell_{m,m}(u)|^{\frac{m-2}{2}}\det(D^{i-1}f_{j}(u))_{ i\in\mathbb{Z}\cap([1,m-1-r]\cup[m+1-r,m]), j\in\mathbb{Z}\cap[1,m-1]}
\end{align}for $ r\in\mathbb Z\cap[0,m-1]$,  where all the  row indices     are arranged in increasing order (even when they are not consecutive integers). We have \begin{align}\begin{split}&
((-1)^{m-1}F_{m-1}(u),\ldots,(-1)^{r}F_{r}(u),\ldots,F_{0}(u))\\={}&(-1)^{m+1}(D^{0}F_{0}(u),\ldots,D^{m-r}F_{0}(u),\ldots,D^{m-1}F_{0}(u))\mathbf V^{{\mathrm T}}_{m}
(u)\end{split}
\end{align}where\begin{align}
( \mathbf V^{\vphantom{\mathrm T}}_{m}
(u))_{a,b}\colonequals \sum_{n=a+b-1}^{m}(-1)^{a+n+m+1}{n-a\choose b-1}\frac{D^{n-a-b+1}\ell_{m,n}(u)}{\mathcal \ell_{m,m}(u)}\label{eq:V_mat_ab}
\end{align} for $a,b\in\mathbb Z\cap[1,m]$.

In particular,  when  $ \widetilde L_m$ is the $m$-th order Vanhove operator (defined in Proposition \ref{prop:Vanhove_Verrill}) with parity $ (-1)^m$ (proved in Proposition \ref{prop:parity_VL}), this explains the origin of the Vanhove matrix $  \mathbf V^{\vphantom{\mathrm T}}_{m}
(u)$ defined in \eqref{eq:Vmat_defn}.\end{proposition}\begin{proof} Set  $\mathcal L_m(u)\colonequals \ell_{m,m} (u)$. If $f_1,\dots,f_{m-1}\in{}C^{\infty}(I)\cap\ker\widetilde L_m $, then the first-order derivative of \begin{align}
F_0(u)\colonequals {}&|\mathcal L_m(u)|^{\frac{m-2}{2}}\det(D^{i-1}f_{j}(u))_{  i,j\in\mathbb Z \cap[1,m-1]}=|\mathcal L_m(u)|^{\frac{m-2}{2}}W[f_1(u),\dots,f_{m-1}(u)]\in C^{\infty}(I)\cap\ker\widetilde L_{m}\end{align}reads\begin{align}\begin{split}D^{1}F_0(u)={}&|\mathcal L_m(u)|^{\frac{m-2}{2}}\det(D^{i-1}f_{j}(u))_{ i\in\mathbb{Z}\cap([1,m-2]\cup\{m\}), j\in\mathbb{Z}\cap[1,m-1]}+\frac{m-2}{2}\frac{F_0(u)}{\mathcal L_m(u)}D^{1}\mathcal L_m(u)\\={}&F_1(u)+\frac{m-2}{2}\frac{F_0(u)}{\mathcal L_m(u)}D^{1}\mathcal L_m(u).\end{split}\label{eq:F1_rec}
\end{align}For higher order derivatives, we have the following recursion:\begin{align}\begin{split}\frac{F_r(u)}{|\mathcal L_m(u)|^{\frac{m-2}{2}}}\colonequals {}&\det(D^{i-1}f_{j}(u))_{ i\in\mathbb{Z}\cap([1,m-1-r]\cup[m+1-r,m]), j\in\mathbb{Z}\cap[1, m-1]}\\={}&D^{1}\det(D^{i-1}f_{j}(u))_{ i\in\mathbb{Z}\cap([1,m-r]\cup[m+2-r,m]), j\in\mathbb{Z}\cap[1, m-1]}\\{}&-\det(D^{i-1}f_{j}(u))_{ i\in\mathbb{Z}\cap([1, m-r]\cup[m+2-r,m-1]\cup\{m+1\}), j\in\mathbb{Z}\cap[1, m-1]}\end{split}
\end{align}according to fundamental properties of determinants. Upon invoking the differential equations $ \widetilde L_{m}f_j(u)=0,j\in\mathbb Z\cap[1,m-1]$, we can further convert this recursion into\begin{align}
F_r(u)={}&D^{1}F_{r-1}(u)+\frac{F_{r-1}(u)}{ \mathcal L_m(u)^{}}D^{1} \mathcal L_m(u)+\frac{(-1)^{r}\ell_{m,m-r}(u)}{ \mathcal L_m(u)}F_{0}(u).
\end{align}Such a recursion is actually retroactively compatible with the $r=1$ case in \eqref{eq:F1_rec}, since $ \ell_{m,m-1}(u)=\frac{m}{2}\frac{\D }{\D u} \mathcal L_m(u)$ is a consequence of the parity relation $\widetilde L_{m} =\sum_{j=0}^m\ell_{m,j}(u)D^j=(-1)^m\widetilde L_{m} ^*$. In other words, to build a linear combination\begin{align}
F_r(u)=\sum_{q}\varphi_{r,q}(u)D^{q}F_{0}(u)
\end{align}for $ r\in\mathbb Z\cap[0,m-1]$, it would suffice to set $\varphi_{r,q}(u)\equiv0 $ for $ q\in\mathbb Z\smallsetminus[0,r] $, and \begin{align}\left\{\begin{array}{l}\varphi_{0,0}(u)\equiv1,\\\varphi_{r,q}(u)=\varphi_{r-1,q-1}(u)+\dfrac{D^{1}[\mathcal L_m(u)\varphi_{r-1,q}(u)]}{ \mathcal L_m(u)}+\dfrac{(-1)^{r}\ell_{m,m-r}(u)}{\mathcal L_m(u)}\varphi_{0,q}(u).\end{array}\right.\end{align}It is then not hard to verify that \begin{align}
\varphi_{r,q}(u)=\sum_{p=m-r+q}^{m}(-1)^{p+1}{r+p-m\choose  q}\frac{D^{r+p-m-q}\ell_{m,p}(u)}{ \mathcal L_m(u)}
\end{align}for  $ q\in\mathbb Z\cap[0,r] $ is the unique solution to the aforementioned recursion. This explains the origin of \eqref{eq:V_mat_ab}.
 \end{proof}

Suppose that the Wro\'nskian matrix $ \mathbf W_m(u)=\mathbf W[f_1(u),\dots,f_m(u)]$ of our concern has determinant $ \Lambda_m|\mathcal L_m(u)|^{-m/2}$, where $ \Lambda_m\neq0$.  The statement of Proposition \ref{prop:Vanhove_dual} guarantees the existence of a constant matrix $ \mathbf S_{m}$ such that the last column of $ \mathbf W^\star=\mathbf W_m^\star(u)$ satisfies [cf.~\eqref{eq:Sm_fit}] \begin{align}
\frac{|\mathcal L_m(u)|^{\frac{m-2}{2}}}{\Lambda_m}((\mathbf W^\star)_{1,m},\ldots,(\mathbf W^\star)_{m,m})^{\mathrm T}=\mathbf S_{m}(f_{1}(u),\dots,f_{m}(u))^{\mathrm T}.
\end{align}The result in Proposition \ref{prop:Vv} then tells us that   $ \mathbf W_m^\star(u)=\frac{(-1)^{m+1}\Lambda_m}{|\mathcal L_m(u)|^{(m-2)/2}}\mathbf S_{m}^{}\mathbf{W}_{m}^{\mathrm T}(u)\mathbf V^{\mathrm T}_{m}
(u)$, where $|\det\mathbf V^{\mathrm T}_{m}
(u)|=1$, so the adjugate equation $ \mathbf W\mathbf W ^\star =\det (\mathbf W)\mathbf I$ for  $ \mathbf W= \mathbf W_m(u)$  and $ \det(\mathbf W)=\Lambda_m|\mathcal L_m(u)|^{-m/2}$ translates into
\begin{align}\mathbf S_{m}=(-1)^{m+1}\frac{[\mathbf W_m(u)]^{-1}[\mathbf V^{\mathrm T}_{m}
(u)]^{-1}[\mathbf{W}_{m}^{\mathrm T}(u)]^{-1}}{|\mathcal L_m(u)|}.\label{eq:Sigma_symm_prep}
\end{align}
\begin{corollary}[(Skew) symmetries of certain matrices]\label{cor:VvSs} Under the same assumptions as Proposition  \ref{prop:Vv}, we have the relation $  \mathbf V^{\vphantom{\mathrm T}}_{m}
(u)=(-1)^{m+1} \mathbf V^{{\mathrm T}}_{m}
(u)$,  which entails $\mathbf S_{m}^{\vphantom{\mathrm T}}=(-1)^{m+1}  \mathbf S^{{\mathrm T}}_{m} $.\end{corollary}\begin{proof}First we show that
$ ( \mathbf V^{\vphantom{\mathrm T}}_{2k-1}
(u))_{a,b}=( \mathbf V^{\vphantom{\mathrm T}}_{2k-1}
(u))_{b,a}$ for  $k\in\mathbb Z_{>0}$.

Along the main anti-diagonal, we have $( \mathbf V^{\vphantom{\mathrm T}}_{2k-1}
(u))_{a,2k-a}=(-1)^{a-1}=( \mathbf V^{\vphantom{\mathrm T}}_{2k-1}
(u))_{2k-a,a}$ for $ a\in\mathbb Z\cap[1,2k-1]$. As we march away from the main  anti-diagonal, and work with $ a+b\in\mathbb Z\cap(2,2k),a\in\mathbb Z\cap[1,2k-2],b\in\mathbb Z\cap[2,2k-2]$, we have the following recursion:\begin{align}\begin{split}
( \mathbf V^{\vphantom{\mathrm T}}_{2k-1}
(u))_{a,b}+( \mathbf V^{\vphantom{\mathrm T}}_{2k-1}
(u))_{a+1,b-1}={}&\sum_{n=a+b-1}^{2k-1}(-1)^{a+n}\left[{n-a\choose b-1}-{n-a-1\choose b-2}\right]\frac{D^{n-a-b+1}\ell_{2k-1,n}(u)}{^{}\mathcal L_{2k-1}(u)}\\={}&\sum_{n=a+b}^{2k-1}(-1)^{a+n}{n-a-1\choose b-1}\frac{D^{n-a-b+1}\ell_{2k-1,n}(u)}{^{}\mathcal L_{2k-1}(u)}\\={}&-\frac{D^{1}[\mathcal L_{2k-1}(u)( \mathbf V^{\vphantom{\mathrm T}}_{2k-1}
(u))_{a+1,b}]}{\mathcal L_{2k-1}(u)}.\end{split}\label{eq:anti_diag_neighbor}
\end{align}

In view of this recursion, in order to show that $ ( \mathbf V^{\vphantom{\mathrm T}}_{2k-1}
(u))_{a,2k-1-a}=( \mathbf V^{\vphantom{\mathrm T}}_{2k-1}
(u))_{2k-1-a,a}$ for all  $ a\in\mathbb Z\cap[1,2k-2]$, it would suffice to verify that $ ( \mathbf V^{\vphantom{\mathrm T}}_{2k-1}
(u))_{1,2k-2}=( \mathbf V^{\vphantom{\mathrm T}}_{2k-1}
(u))_{2k-2,1}$. Direct computations reveal that\begin{align}\begin{split}
( \mathbf V^{\vphantom{\mathrm T}}_{2k-1}
(u))_{1,2k-2}-( \mathbf V^{\vphantom{\mathrm T}}_{2k-1}
(u))_{2k-2,1}={}&\sum_{n=2k-2}^{2k-1}(-1)^{1+n}\left[{n-1\choose 2k-3}+{n-2k+2\choose0}\right]\frac{D^{n-2k+2}\ell_{2k-1,n}(u)}{^{}\mathcal L_{2k-1}(u)}\\={}&\frac{1}{\mathcal L_{2k-1}(u)}[(2k-1)D^{1}\ell_{2k-1,2k-1}(u)-2\ell_{2k-1,2k-2}(u)]=0,\end{split}
\end{align}according to \eqref{eq:sublead}.

The efforts in the last paragraph amount to a verification that  $ ( \mathbf V^{\vphantom{\mathrm T}}_{2k-1}
(u))_{a,b}=( \mathbf V^{\vphantom{\mathrm T}}_{2k-1}
(u))_{b,a}$ holds for $ a+b=2k-1$. To show that the same identity holds true for  smaller values of $ a+b$, we exploit \eqref{eq:anti_diag_neighbor} again, and build up on the ``boundary conditions''  $ ( \mathbf V^{\vphantom{\mathrm T}}_{2k-1}
(u))_{1,b}=( \mathbf V^{\vphantom{\mathrm T}}_{2k-1}
(u))_{b,1}$ for $ b\in\mathbb Z\cap[2,2k-3]$. By \eqref{eq:V_mat_ab} and elementary combinatorics, we have\begin{align}\begin{split}
( \mathbf V^{\vphantom{\mathrm T}}_{2k-1}
(u))_{1,b}-( \mathbf V^{\vphantom{\mathrm T}}_{2k-1}(u))_{b,1}={}&\sum_{n=b}^{2k-1}\left[(-1)^{1+n}{n-1\choose b-1}-(-1)^{b+n}{n-b\choose 0}\right]\frac{D^{n-b}\ell_{2k-1,n}(u)}{^{}\mathcal L_{2k-1}(u)}\\={}&\sum_{n=b}^{2k-1}\left[(-1)^{1+n}{n\choose b}+(-1)^{ n}{n-1\choose b}-(-1)^{b+n}\right]\frac{D^{n-b}\ell_{2k-1,n}(u)}{^{}\mathcal L_{2k-1}(u)}.\end{split}\label{eq:anti_diag_rec}
\end{align}Referring back to the combinatorial constraint in \eqref{eq:VV_poly_rec}, which says\begin{align}
-\ell_{2k-1,b}(u)={}&\sum_{n=b}^{2k-1}(-1)^{n}{n\choose b}D^{n-b}\ell_{2k-1,n}(u),
\end{align}we may further reduce \eqref{eq:anti_diag_rec} into\begin{align}\begin{split}&
( \mathbf V^{\vphantom{\mathrm T}}_{2k-1}
(u))_{1,b}-( \mathbf V^{\vphantom{\mathrm T}}_{2k-1}(u))_{b,1}\\={}&\frac{\ell_{2k-1,b}(u)-\ell_{2k-1,b}(u)}{\mathcal L_{2k-1}(u)}-\sum_{n=b+1}^{2k-1}\left[(-1)^{ 1+n}{n-1\choose b+1-1}-(-1)^{b+1+n}\right]\frac{D^{n-b-1+1}\ell_{2k-1,n}(u)}{^{}\mathcal L_{2k-1}(u)}\\={}&-\frac{D^{1}\{\mathcal L_{2k-1}(u)[( \mathbf V^{\vphantom{\mathrm T}}_{2k-1}
(u))_{1,b+1}-( \mathbf V^{\vphantom{\mathrm T}}_{2k-1}(u))_{b+1,1}]\}}{\mathcal L_{2k-1}(u)}.\end{split}\tag{\ref{eq:anti_diag_rec}$'$}
\end{align}Thus, it is clear that   $ ( \mathbf V^{\vphantom{\mathrm T}}_{2k-1}
(u))_{1,b}=( \mathbf V^{\vphantom{\mathrm T}}_{2k-1}
(u))_{b,1}$ indeed holds for $ b\in\mathbb Z\cap[2,2k-3]$, so  $  \mathbf V^{\vphantom{\mathrm T}}_{2k-1}
(u)= \mathbf V^{{\mathrm T}}_{2k-1}
(u)$  is true.

Since  $ [\mathbf W_{2k-1}^{\mathrm T}(u)]^{-1}$ is the transpose of $ [\mathbf W_{2k-1}^{\vphantom{\mathrm T}}(u)]^{-1}$, the relation \eqref{eq:Sigma_symm_prep} allows us to deduce  $\mathbf S_{2k-1}^{\vphantom{\mathrm T}}=\mathbf S^{{\mathrm T}}_{2k-1} $
from   $  \mathbf V^{\vphantom{\mathrm T}}_{2k-1}
(u)= \mathbf V^{{\mathrm T}}_{2k-1}
(u)$.

The aforementioned techniques also lead us to    $ ( \mathbf V^{\vphantom{\mathrm T}}_{2k}
(u))_{1,b}=-( \mathbf V^{\vphantom{\mathrm T}}_{2k}
(u))_{b,1}$  for $ b\in\mathbb Z\cap[2,2k]$, and consequently     $ ( \mathbf V^{\vphantom{\mathrm T}}_{2k}
(u))_{a,b}=-( \mathbf V^{\vphantom{\mathrm T}}_{2k}
(u))_{b,a}$  for $ a+b\in\mathbb Z\cap[3,2k+1]$. What remains to  be investigated is the case where $a=b=1 $. By direct computation, we have{\allowdisplaybreaks\begin{align}
( \mathbf V^{\vphantom{\mathrm T}}_{2k}
(u))_{1,1}={}&\sum_{n=1}^{2k}(-1)^{n}{n-1\choose 0}\frac{D^{n-1}\ell_{2k,n}(u)}{^{}\mathcal L_{2k}(u)}=\sum_{n=1}^{2k}(-1)^{n}\frac{D^{n-1}\ell_{2k,n}(u)}{^{}\mathcal L_{2k}(u)},\\( \mathbf V^{\vphantom{\mathrm T}}_{2k}
(u))_{2,1}={}&\sum_{n=2}^{2k}(-1)^{n+1}{n-2\choose 0}\frac{D^{n-2}\ell_{2k,n}(u)}{^{}\mathcal L_{2k}(u)}=\sum_{n=2}^{2k}(-1)^{n+1}\frac{D^{n-2}\ell_{2k,n}(u)}{^{}\mathcal L_{2k}(u)},\\( \mathbf V^{\vphantom{\mathrm T}}_{2k}
(u))_{1,2}={}&\sum_{n=2}^{2k}(-1)^{n}{n-1\choose 1}\frac{D^{n-2}\ell_{2k,n}(u)}{^{}\mathcal L_{2k}(u)}=\sum_{n=2}^{2k}(-1)^{n}(n-1)\frac{D^{n-2}\ell_{2k,n}(u)}{^{}\mathcal L_{2k}(u)},\end{align}}so another invocation of  \eqref{eq:VV_poly_rec} brings us\begin{align}\begin{split}2( \mathbf V^{\vphantom{\mathrm T}}_{2k}
(u))_{1,1}={}&2( \mathbf V^{\vphantom{\mathrm T}}_{2k}
(u))_{1,1}+\frac{D^1\{\mathcal L_{2k}(u)[( \mathbf V^{\vphantom{\mathrm T}}_{2k}
(u))_{2,1}+( \mathbf V^{\vphantom{\mathrm T}}_{2k}
(u))_{1,2}]\}}{\mathcal L_{2k}(u)}\\={}&-\frac{2\ell_{2k,1}(u)}{^{}\mathcal L_{2k}(u)}+\sum_{n=2}^{2k}(-1)^{n}{n\choose 1}\frac{D^{n-1}\ell_{2k,n}(u)}{^{}\mathcal L_{2k}(u)}\\={}&\frac{-2\ell_{2k,1}(u)+\ell_{2k,1}(u)+\ell_{2k,1}(u)}{^{}\mathcal L_{2k}(u)}=0,\end{split}\end{align}as expected.

 Now that we have   $  \mathbf V^{\vphantom{\mathrm T}}_{2k}
(u)=- \mathbf V^{{\mathrm T}}_{2k}
(u)$,  the skew symmetry  $ \mathbf S_{2k}^{\vphantom{\mathrm T}}=-\mathbf S_{2k}^{{\mathrm T}}$ follows from \eqref{eq:Sigma_symm_prep}.
\end{proof}\subsection{Threshold behavior of Wro\'nskian cofactors\label{subsec:threshold}}
Naturally, our next goal is to seek explicit formulae (in terms of fundamental solutions to Vanhove's differential equations) for the cofactor Wro\'nskian determinants \begin{align}\begin{split}(-1)^{r+m}(\mathbf W_{m}^{\star}(u))_{r,m}={}& W\big[\mathcal{F}_{m,1}(u),\dots,\reallywidehat{\mathcal{F}_{m,r}(u)},\dots,\mathcal{F}_{m,m}(u)\big]\\={}&\det(D^{i-1}\mathcal{F}_{m,j}(u))_{i\in\mathbb Z\cap[1, m-1],j\in\mathbb Z\cap([1,r-1]\cup[r+1,m])},\end{split}\end{align}
for all $  r\in\mathbb Z\cap[1,m]$. This will amount to the complete characterizations of the matrix $ \mathbf S_m $ that fits into the identity \eqref{eq:Sm_fit} for $ u\in(0,1)$.

Even though the (skew) symmetry relations in Corollary \ref{cor:VvSs} may nearly halve our workload on the elements of  $ \mathbf S_{m}$, we still need somewhat extensive preparations in Lemmata \ref{lm:loc} and \ref{lm:spec_det} below, before carrying out our computations  for $ k\in\mathbb Z_{>2}$ in Proposition \ref{prop:Sigma_sigma}.

Let $ C^\omega(a,b)$ be the totality of real-analytic functions\footnote{For each $ u_0\in(a,b)$, there exists a positive number $ \varepsilon$, such that
the Taylor series $ \sum_{n=0}^\infty \frac{(u-u_0)^ n}{n!}D^nf(u_0)$ of a real-analytic function $ f\in C^\omega(a,b)$ converges to $ f(u)$ for all $u\in (u_0-\varepsilon,u_0+\varepsilon)$. As a slight abuse of notations, we sometimes use the same symbol for a real-analytic function and its real-analytic continuation to a larger domain of definition.} on the interval $(a,b)$. The following patterns \begin{align}
\frac{|\mathcal L_{2k-1}(u)|^{k-\frac{3}{2}}}{|4-u|^{k-\frac{3}{2}}}W\big[\mathcal F_{2k-1,1}(u),\dots,\reallywidehat{\mathcal F_{2k-1,k}(u)},\dots,\mathcal F_{2k-1,2k-1}(u)\big]\in {}&C^{\omega}(0,16),\label{eq:pivot_mu}\\\frac{|\mathcal L_{2k}(u)|^{k-1}}{|1-u|^{k-1}}W\big[\mathcal F_{2k,1}(u),\dots,\reallywidehat{\mathcal F_{2k,k+1}(u)},\dots,\mathcal F_{2k,2k}(u)]\in{}&C^\omega(0,9)
\label{eq:pivot_nu}\end{align}motivate us to search  for certain linear combinations of Bessel moments with  suitable  orders of vanishing and analytic behavior as $u $ approaches a threshold, in the next lemma.
\begin{lemma}[Local behavior near thresholds]\label{lm:loc}In what follows, the exact value of a positive number $\varepsilon $ may vary from context to context.\begin{enumerate}[leftmargin=*,  label=\emph{(\alph*)},ref=(\alph*),
widest=a, align=left]
\item\label{itm:loc_mu} For each $ s\in\mathbb Z\cap[1,k]$, there exists a real-analytic function $  \mathsf f_{k,s}\in C^\omega((2s)^{2}-\varepsilon,(2s)^{2}+\varepsilon)$ such that we have  \begin{align}
\R\mathsf F_{k,s}^1(u)={}&\frac{(-4s)^{1-k}}{(2k-3)!!}\frac{(-1)^{\left\lfloor\frac{k-s}{2}\right\rfloor+(k-1)(k-s)}}{2}|u-(2s)^{2}|^{k-\frac{3}{2}}\left\{ 1+[u-(2s)^{2}] \mathsf f_{k,s}(u)\right\}, \label{eq:mu_loc1}
\intertext{for $ ((2s)^{2}-u)(-1)^{k-s}\in[0,\varepsilon)$, and}
\I\mathsf F_{k,s}^1(u)={}&\frac{(-4s)^{1-k}}{(2k-3)!!}\frac{(-1)^{\left\lfloor\frac{k-s-1}{2}\right\rfloor+(k-1)(k-s-1)}}{2}|u-(2s)^{2}|^{k-\frac{3}{2}}\left\{ 1+[u-(2s)^{2}] \mathsf f_{k,s}(u)\right\} ,\label{eq:mu_loc2}\end{align}for $((2s)^{2}-u)(-1)^{k-s-1}\in[0,\varepsilon) $. \item \label{itm:loc_nu}For each $ s\in\mathbb Z\cap[1,k+1]$, there exists a real-analytic function $  \mathsf g_{k,s}\in C^\omega((2s-1)^{2}-\varepsilon,(2s-1)^{2}+\varepsilon)$ such that we have\begin{align}\begin{split}
\mathfrak G_{k,s}(u)\colonequals {}&\frac{(-1)^{\left\lfloor \frac{k-s\vphantom{b}}{2}\right\rfloor+k-s}}{2^{2k}(k-1)!}\sqrt{\frac{\pi}{(2s-1)^{2k-1}}}[u-(2s-1)^{2}]^{k-1}\left\{ 1+[u-(2s-1)^{2}] \mathsf g_{k,s}(u)\right\}\\={}&\begin{cases}\R \mathsf G^{1}_{k,s}(u), & (-1)^{k-s}=+1, \\
\I \mathsf G^{1}_{k,s}(u), & (-1)^{k-s}=-1, \\
\end{cases}
\end{split}\end{align} and the  expression\begin{align}
\begin{split}&\mathfrak L_{k,s}(u)\\\colonequals {}&\begin{cases}\mathcal I^1_{2k,s} (u)+\frac{(-1)^{k-s-1}-i}{\pi(1-i)}\mathfrak G_{k,s}(u)\big[\log|u-(2s-1)^{2}|-\frac{i\pi}{2}\big], & u\in((2s-1)^{2}-\varepsilon,(2s-1)^{2}), \\
\mathcal K^1_{2k,s} (u)+\frac{(-1)^{k-s-1}-i}{\pi(1-i)}\mathfrak G_{k,s}(u)\big[\log|u-(2s-1)^{2}|+\frac{i\pi}{2}\big], & u\in((2s-1)^{2},(2s-1)^{2}+\varepsilon) \\
\end{cases}\label{eq:log_jump}\end{split}
\end{align}extends to a real-analytic function in $ C^\omega((2s-1)^{2}-\varepsilon,(2s-1)^{2}+\varepsilon)$. \end{enumerate}\end{lemma}\begin{proof}\begin{enumerate}[leftmargin=*,  label=(\alph*),ref=(\alph*),
widest=a, align=left]
\item If $ k-s$ is even, then $\R\mathsf F_{k,s}^1(u)$ [see \eqref{eq:GI_sum} and \eqref{eq:GK_sum} for explicit formulae] is well-defined and real-analytic for $ u\in((2s-2)^{2},(2s)^{2})$, and  the function $ (-1)^{\left\lfloor\frac{k-s}{2}\right\rfloor}(1+2k\delta_{s,k})\mathcal F_{2k-1,2\left\lfloor\frac{k-s}{2}\right\rfloor+1}(u)=(-1)^{\left\lfloor\frac{k-s}{2}\right\rfloor} (1+ 2k\delta_{s,k})\mathcal F_{2k-1,k-s+1}(u)$ is the only addend that fails to be  real-analytic in a neighborhood of $ u=(2s)^{2}$. Moreover, the function $\R\mathsf F_{k,s}^1(u)$ has vanishing derivatives up to order $ k-2=\max\big\{2\big(\left\lfloor\frac{k}{2}\right\rfloor-1\big),2\left\lfloor\frac{k-1}{2}\right\rfloor-1\big\}$, according to Proposition \ref{prop:u_sq}(a). We claim that the leading order behavior for $\R\mathsf F_{k,s}^1(u)$ is $ A[(2s)^{2}-u]^{k-\frac{3}{2}}[1+o(1)]$, where the constant $ A$ can be determined through l'H\^opital's rule. Concretely speaking, when  $k$ is odd,
we have \begin{align}\begin{split}
A={}&\lim_{u\to(2s)^2-0^+}\frac{\R\mathsf F_{k,s}^1(u)}{[(2s)^{2}-u]^{k-\frac{3}{2}}}=\lim_{u\to(2s)^2-0^+}\frac{D^{k-1}\R\mathsf F_{k,s}^1(u)}{D^{k-1}\{[(2s)^{2}-u]^{k-\frac{3}{2}}\}}\\={}&\frac{(1+2k\delta_{s,k})(-1)^{\left\lfloor\frac{k-s}{2}\right\rfloor}}{(-2)^{1-k}(2k-3)!!}\lim_{u\to(2s)^2-0^+}{\sqrt{(2s)^{2}-u}}D^{k-1} \mathcal F_{2k-1,k-s+1}(u),\end{split}
\end{align} where \begin{align}\begin{split}&
\lim_{u\to(2s)^2-0^+}{\sqrt{(2s)^{2}-u}}D^{k-1}\mathcal  F_{2k-1,k-s+1}(u)\\={}&\frac{1}{(4s)^{k-1}}\lim_{u\to(2s)^2-0^+}{\sqrt{(2s)^{2}-u}} \mathcal F_{2k-1,k-s+1}^{\frac{k+1}{2}}(u)\\={}&\frac{1}{(4s)^{k-1}}\lim_{u\to(2s)^2-0^+}\frac{{\sqrt{(2s)^{2}-u}}}{\pi^{s}(1+2k\delta_{s,k})}\int_0^\infty\frac{e^{\sqrt{u}t}}{\sqrt{2\pi \sqrt{u} t}}\left( \frac{e^{t}}{\sqrt{2\pi t}} \right)^{k-s}\left( \sqrt{\frac{\pi}{2t}}e^{-t} \right)^{k+s}t^{k}\D t\\={}&\frac{2^{2-3k}s^{1-k}}{1+2k\delta_{s,k}}\end{split}\label{eq:lHopitalHankel}
\end{align}follows from  the asymptotic behavior \cite[\S7.23]{Watson1944Bessel}\begin{align}I_0(t)=
\frac{e^{t}}{\sqrt{2\pi t}}\left[ 1+O\left( \frac{1}{t} \right) \right]\quad\text{and}\quad K_0(t)=\sqrt{\frac{\pi}{2t}}e^{-t}\left[ 1+O\left( \frac{1}{t} \right) \right],
\end{align}along with Bessel's differential equations $ (uD^2+D^1)I_0(t)=\frac{t^2}{4}I_0(t)$ and  $ (uD^2+D^1)K_0(t)=\frac{t^2}{4}K_0(t)$. This shows that  $ A=\frac{(-4s)^{1-k}}{(2k-3)!!}\frac{(-1)^{\left\lfloor\frac{k-s}{2}\right\rfloor}}{2}$ when $k$ is odd. Based on the asymptotic behavior \cite[\S7.23]{Watson1944Bessel}\begin{align}
I_1(t)=
\frac{e^{t}}{\sqrt{2\pi t}}\left[ 1+O\left( \frac{1}{t} \right) \right]\quad\text{and}\quad K_1(t)=\sqrt{\frac{\pi}{2t}}e^{-t}\left[ 1+O\left( \frac{1}{t} \right) \right],
\end{align}one can perform a similar analysis when $k $ is even.

If $k-s $ is odd, then $ \R\mathsf F^1_{k,s}(u)$ is well-defined and real-analytic for $ u\in((2s)^{2},(2s+2)^{2})$, and the function $ (-1)^{\left\lfloor\frac{k+s-1}{2}\right\rfloor}\mathcal F_{2k-1,2\left\lfloor\frac{k+s-1}{2}\right\rfloor+k}^{\ell}(u)=(-1)^{\left\lfloor\frac{k+s-1}{2}\right\rfloor} \mathcal F_{2k-1,3k+s-1}^{\ell}(u)$ is the only addend that fails to be  real-analytic in a neighborhood of $ u=(2s)^{2}$. One can repeat the remaining procedures in the last paragraph to find the leading order  behavior in such scenarios.

By  the Fuchs condition \cite[\S15.3]{Ince1956ODE}, the operator  $\widetilde L_{2k-1}$ has a regular singular point at $u=(2s)^2$ for each  $ s\in\mathbb Z\cap[1,k]$, so the left-hand sides of  \eqref{eq:mu_loc1}  and   \eqref{eq:mu_loc2} should be equal to convergent generalized power series. According to the Frobenius method \cite[\S16.1]{Ince1956ODE}, such series may (in principle) involve both integer powers   $ [u-(2s)^2]^{\mathbb Z}$ and half-integer powers $ |u-(2s)^2|^{\mathbb Z+\frac{1}{2}}$. Our next task is to show that  integer powers   $ [u-(2s)^2]^{\mathbb Z}$ actually will not appear in these scenarios.

Without loss of generality, we only elaborate on  the cases where $ u<(2s)^2$, $ k-s$ is even, and   $k$ is odd.  We start from the following  asymptotic expansions for $M\in\mathbb Z_{\geq0},t\to\infty $:{\allowdisplaybreaks\begin{align}\begin{split}
\mathscr F_{k,s,M}(u,t)t^{k+2M}\colonequals {}&\left\{ I_0(\sqrt{u}t)[-iK_0(t)]^{k+s}\frac{[\pi I_{0}(t)+iK_{0}(t)]^{k-s}+[\pi I_{0}(t)-iK_{0}(t)]^{k-s}}{2}\right.\\{}&+iK_0(\sqrt{u}t)[-iK_0(t)]^{k+s}\frac{[\pi I_{0}(t)+iK_{0}(t)]^{k-s}-[\pi I_{0}(t)-iK_{0}(t)]^{k-s}}{2}\\{}&\left.+iK_0(\sqrt{u}t)[-iK_0(t)]^{k-s}\frac{[\pi I_{0}(t)+iK_{0}(t)]^{k+s}-[\pi I_{0}(t)-iK_{0}(t)]^{k+s}}{2} \right\}\frac{t^{k+2M}}{\pi^{k+1}}\\{}&-\frac{e^{(\sqrt{u}-2s)t}}{\sqrt{t}}\sum_{n=0}^{2M}\frac{f_{k,s,n}(u)}{t^{n-2M}}=O\left( \frac{e^{(\sqrt{u}-2s)t}}{|t|^{3/2}} \right),\end{split}
\end{align}}where $ f_{k,s,n}(u)$ is holomorphic in a $ \mathbb C$-neighborhood of $ u=(2s)^2$. Integrating the left-hand side of the equation above over $t\in(0,\infty)$, while invoking the dominated convergence theorem, we obtain\begin{align}
\lim_{u\to(2s)^2-0^+}
\left[\R\mathsf F_{k,s}^{\frac{k}{2}+M}(u)-\sum_{n=0}^{2M}\frac{f_{k,s,n}(u)\Gamma \left(2 M-n+\frac{1}{2}\right)}{|\sqrt{u}-2s|^{2M-n+\frac{1}{2}}}\right]=\int_{0}^\infty\mathscr F_{k,s,M}((2s)^2,t)t^{k+2M}\D t.
\end{align}The right-hand side of the equation above  vanishes, because it is a constant multiple of [cf.~\eqref{eq:H_vert_int}]\begin{align}
&\R\int_{-i\infty}^{i\infty}\left\{H_0^{(2)}(2sz)[H_0^{(1)}(z)]^{2s}[H_0^{(1)}(z)H_0^{(2)}(z)]^{k-s}z^{k+2M}-\frac{1}{\sqrt{z}}\sum_{n=0}^{2M}\frac{h_{k,s,n}}{z^{n-2M}}\right\}\D z,
\end{align} where the integrand has $ O(|z|^{-3/2})$ asymptotics for $|z|\to\infty,|\arg z|\leq\frac{\pi}{2}$, and the contour closes to the right. So far, we know that the generalized power series expansion for  $ \R\mathsf F_{k,s}^{\frac{k}{2}+M}(u)$ contains no $ [\sqrt{u}-2s]^{\mathbb Z_{\leq 0}}$ terms, for each  $M\in\mathbb Z_{\geq0}$. Likewise, one can argue that the same property applies to   $ \R\acute{\mathsf F}_{k,s}^{\frac{k}{2}+M}(u),M\in\mathbb Z_{\geq0}$. This proves that the generalized power series expansion for  $ \mathsf F_{k,s}^{1}(u)$ contains no $ [\sqrt{u}-2s]^{\mathbb Z}$ terms, and equivalently, no  $ [u-(2s)^{2}]^{\mathbb Z}$ terms, as claimed.

\item  Without loss of generality, we will focus on the scenarios where  $ k-s$ is even, and   $k$ is odd. Under such assumptions, one can show (through direct examination of the integral representations for off-shell Bessel moments) that $ \R \mathsf G^1_{k,s}(u)$ is real-analytic for $ u\in((\max\{0,2s-3\})^2,(2s+1)^2)$. At the threshold $u=(2s-1)^2 $, this function has vanishing derivatives up to $(k-2)$nd order, according to Proposition \ref{prop:u_sq}(b).

As a variation on our proof \cite[Lemma 2]{Zhou2018ExpoDESY} of the integral formulae for  generalized Crandall numbers \eqref{eq:Crandall_num}, we consider \begin{align}
\left(\int_{-iT}^{-i\varepsilon}+\int_{C_\varepsilon}+\int_{i\varepsilon}^{iT}\right)\left\{\frac{H_0^{(2)}((2s-1)z)}{(2/\pi)^{k+1}}[H_0^{(1)}(z)]^{2s-1}[H_0^{(1)}(z)H_0^{(2)}(z)]^{k+1-s}z^{k}-\frac{(-1)^{\left\lfloor \frac{s\vphantom{b}}{2}\right\rfloor}}{\sqrt{2s-1}z}\right\}\D z\label{eq:nu_Crandall}
\end{align} where $C_\varepsilon $ is a semi-circular arc in the right half-plane, joining $-i\varepsilon$ to $i\varepsilon$.  For each
fixed $\varepsilon> 0$, we can close the contour to the right, as $ T\to\infty$. Noting that  $ \int_{C_\varepsilon}\frac{\D z}{z}=\pi i$, we may read off the imaginary part of the last displayed equation as   $ \R \mathsf G^{\frac{k+1}{2}}_{k,s}((2s-1)^{2})=\frac{(-1)^{\left\lfloor \frac{s\vphantom{b}}{2}\right\rfloor}}{2^{k+1}} \sqrt{\frac{\pi}{2s-1}}$. Falling back on  l'H\^opital's rule, we can deduce  $ \R \mathsf G^{1}_{k,s}(u)=\frac{(-1)^{\left\lfloor \frac{k-s\vphantom{b}}{2}\right\rfloor}}{2^{2k}(k-1)!}\sqrt{\frac{\pi}{(2s-1)^{2k-1}}}[u-(2s-1)^{2}]^{k-1}[1+o(1)]$ as $ u\to (2s-1)^2$. Here, the  $ o(1)$ term can be rewritten as $ [u-(2s-1)^{2}] \mathsf g_{k,s}(u)$ with  $  \mathsf g_{k,s}\in C^\omega((2s-1)^{2}-\varepsilon,(2s-1)^{2}+\varepsilon)$, because $ \R \mathsf G^{1}_{k,s}(u)$ is real-analytic  as $u$ approaches $(2s-1)^2 $.

One can represent the Taylor coefficients $ g_{k,s,n}=n!D^n \mathsf g_{k,s}((2s-1)^{2}),n\in\mathbb Z_{\geq0}$ through the off-shell moments  $ \R \mathsf G^{\frac{k+1}{2}+M}_{k,s}((2s-1)^{2})$ and   $ \R \acute{\mathsf G}^{\frac{k+1}{2}+M}_{k,s}((2s-1)^{2})$  for $ M\in\mathbb Z_{\geq0}$. Here, one can  evaluate  $\R \mathsf G^{\frac{k+1}{2}+M}_{k,s}((2s-1)^{2}) $ via a contour integral over\begin{align}
&\frac{H_0^{(2)}((2s-1)z)}{(2/\pi)^{k+1}}[H_0^{(1)}(z)]^{2s-1}[H_0^{(1)}(z)H_0^{(2)}(z)]^{k+1-s}z^{k+2M}-\sum_{n=1}^{2M+1}\frac{\eta_{k,s,n}}{z^{n-2M}}=O\left( \frac{1}{z^{2}} \right),
\end{align}akin to our treatment of \eqref{eq:nu_Crandall}. It then follows that     $ \R \mathsf G^{\frac{k+1}{2}+M}_{k,s}((2s-1)^{2})=\frac{\sqrt{\pi}}{2^{k+1}} (-1)^{M}\eta_{k,s,2M+1}$. There is a similar relation that expresses    $ \R \acute{\mathsf G}^{\frac{k+1}{2}+M}_{k,s}((2s-1)^{2})$  through coefficients of asymptotic expansions.

 Now, for $ M\in\mathbb Z_{\geq0}$, we consider\begin{align}\begin{split}
\mathscr H_{k,s,M}(u,z)z^{k+2M}\colonequals {}&\frac{H_0^{(2)}(\sqrt{u}z)}{(2/\pi)^{k+1}}[H_0^{(1)}(z)]^{2s-1}[H_0^{(1)}(z)H_0^{(2)}(z)]^{k+1-s}z^{k+2M}\\{}&-e^{-i[\sqrt{u}-(2s-1)]z}\left[\frac{\eta_{k,s,2M+1}(u)}{z+1}+\sum_{n=1}^{2M}\frac{\eta_{k,s,n}(u)}{z^{n-2M}}\right],
\end{split}\end{align} where the functions $\eta_{k,s,n}(u) $ are holomorphic in a $ \mathbb C$-neighborhood of $ u=(2s-1)^2$, and $ \mathscr H_{k,s,M}(u,z)=O(e^{-i[\sqrt{u}-(2s-1)]z}/z^{2-k-2M})$ for fixed $u$ and $ |z|\to\infty,-\frac{\pi}{2}\leq\arg z\leq\frac{\pi}{2}$. For $ u\in((2s-1)^2-\varepsilon,(2s-1)^2)$, we can invoke Jordan's lemma in the form of \begin{align}
\int_0^{i\infty}\mathscr H_{k,s,M}(u,z)z^{k+2M}\D z=\int_0^{\infty}\mathscr H_{k,s,M}(u,x)x^{k+2M}\D x.\label{eq:HksM_Wick}
\end{align}In view of \eqref{eq:HnInKn}, we rewrite the left-hand side of \eqref{eq:HksM_Wick} as\begin{align}\begin{split}{}&\left( \frac{2}{\pi i} \right)^{k+1}{(-1)^{M}} \int_{0}^\infty[i \pi  I_0(\sqrt{u}t)-K_0(\sqrt{u}t)][K_0(t)]{}^{k+s} [i \pi  I_0(t)-K_0(t)]{}^{k+1-s}t^{k+2M}\D t\\{}&-\eta_{k,s,2M+1}(u)\int_0^{i\infty} \frac{e^{-i[\sqrt{u}-(2s-1)]z}}{z+1}\D z
-\sum_{n=1}^{2M}\frac{\eta_{k,s,n}(u)(2M-n)!}{\{i[\sqrt{u}-(2s-1)]\}^{2M+1-n}},
\end{split}\end{align}where $ \eta_{k,s,2M+1}((2s-1)^2)=\frac{2^{k+1}}{\sqrt{\pi}} (-1)^{M}\R \mathsf G^{\frac{k+1}{2}+M}_{k,s}((2s-1)^{2})$ and \begin{align}
\int_0^{i\infty} \frac{e^{-i[\sqrt{u}-(2s-1)]z}}{z+1}\D z=-\log|2s-1-\sqrt{u}|-\gamma_{0} +\frac{i \pi }{2}+o(1)
\end{align} as $ u\to(2s-1)^2-0^+$. Meanwhile, the right-hand side of  \eqref{eq:HksM_Wick}  is continuous across the threshold $ u=(2s-1)^2$, and admits analytic continuation to $ u\in((2s-1)^2,(2s-1)^2+\varepsilon)$,   in the following form:\begin{align}\begin{split}&
\int_0^{\infty}\mathscr H_{k,s,M}(u,x)x^{k+2M}\D x=\int_0^{-i\infty}\mathscr H_{k,s,M}(u,z)z^{k+2M}\D z\\={}&\left( \frac{2}{\pi i} \right)^{k+1}{(-1)^{M}}\int_{0}^\infty K_0(\sqrt{u}t)[K_0(t)]{}^{k+1-s} [i \pi  I_0(t)+K_0(t)]{}^{k+s}t^{k+2M}\D t\\{}&-\eta_{k,s,2M+1}(u)\int_0^{-i\infty} \frac{e^{-i[\sqrt{u}-(2s-1)]z}}{z+1}\D z-\sum_{n=1}^{2M}\frac{\eta_{k,s,n}(u)(2M-n)!}{\{i[\sqrt{u}-(2s-1)]\}^{2M+1-n}},\end{split}
\end{align}where\begin{align}
\int_0^{-i\infty} \frac{e^{-i[\sqrt{u}-(2s-1)]z}}{z+1}\D z=-\log|2s-1-\sqrt{u}|-\gamma_{0} -\frac{i \pi }{2}+o(1)
\end{align} as $ u\to(2s-1)^2+0^+$. Therefore, with $ \mathsf h^M_{k,s}(u)\colonequals \sum_{n=1}^{2M}{\eta_{k,s,n}(u) (2M-n)!\{i[\sqrt{u}-(2s-1)]\}^{n-1}}\in C^\omega((2s-1)^{2}-\varepsilon,(2s-1)^{2}+\varepsilon)$ and the understanding that $ \mathsf h^0_{k,s}(u)\equiv0$, we have the following continuity condition for every $ M\in\mathbb Z_{\geq0}$:{\allowdisplaybreaks\begin{align}\begin{split}
&\lim_{u\to(2s-1)^2-0^+}\left\{\int_{0}^\infty[i \pi  I_0(\sqrt{u}t)-K_0(\sqrt{u}t)][K_0(t)]{}^{k+s} \right.[i \pi  I_0(t)-K_0(t)]{}^{k+1-s}t^{k+2M}\D t+\\{}&\left.+\left( \frac{\pi i}{2} \right)^{k+1}\left[\frac{2^{k+1}}{\sqrt{\pi}} \left( \log|2s-1-\sqrt{u}|-\frac{i \pi }{2} \right)\R \mathsf G^{\frac{k+1}{2}+M}_{k,s}((2s-1)^{2})\right.\left.-\frac{ \mathsf h^M_{k,s}(u)}{[\sqrt{u}-(2s-1)]^{2M}}\right]\right\}\\={}&\lim_{u\to(2s-1)^2+0^+}\left\{\int_{0}^\infty K_0(\sqrt{u}t)[K_0(t)]{}^{k+1-s} [i \pi  I_0(t)+K_0(t)]{}^{k+s}t^{k+2M}\D t+\right.\\{}&\left.+\left( \frac{\pi i}{2} \right)^{k+1}\left[\frac{2^{k+1}}{\sqrt{\pi}} \left( \log|2s-1-\sqrt{u}|+\frac{i \pi }{2} \right)\R \mathsf G^{\frac{k+1}{2}+M}_{k,s}((2s-1)^{2})\right.\left.-\frac{ \mathsf h^M_{k,s}(u)}{[\sqrt{u}-(2s-1)]^{2M}}\right]\right\}.\end{split}
\end{align}}There is a similar continuity condition relating derivatives of these off-shell Bessel moments (with respect to $u$) to  $ \R \acute{\mathsf G}^{\frac{k+1}{2}+M}_{k,s}((2s-1)^{2})$.

Combining the input from the last two paragraphs,  one sees that both  $\mathcal I^1_{2k,s} (u)-\frac{i}{\pi}\mathfrak G_{k,s}(u)\big[\log|u-(2s-1)^{2}|-\frac{i\pi}{2}\big],u\in((2s-1)^{2}-\varepsilon,(2s-1)^{2})$ and  $\mathcal K^1_{2k,s} (u)-\frac{i}{\pi}\mathfrak G_{k,s}(u)\big[\log|u-(2s-1)^{2}|+\frac{i\pi}{2}\big],u\in((2s-1)^{2},(2s-1)^{2}+\varepsilon)$ are equal to  their respective Taylor series in powers of   $[u-(2s-1)^2]^{\mathbb Z_{\geq0}}$, and all the corresponding Taylor coefficients agree.  \qedhere\end{enumerate}\end{proof}

The next lemma will provide us with special values
of the determinants occurring in \eqref{eq:pivot_mu} and \eqref{eq:pivot_nu}, among other things.\begin{lemma}[Some special determinants]\label{lm:spec_det}Fix an integer $m\in\mathbb Z_{>1} $. For every $ s\in\mathbb Z\cap\left[1,\left\lfloor \frac{m}{2} \right\rfloor+1\right]$, define  thresholds $ \theta_{m,s}\colonequals\big(2s-\tfrac{1+(-1)^{m}}{2}\big)^{2}$ and  $ \theta_{m,\left\lfloor m/2 \right\rfloor+2}=\infty$, along with  Wro\'nskian determinants \begin{align}W_{m,s}(u)\colonequals{}&\det(
D^{i-1}\mathcal F_{m,j}(u))_{i\in\mathbb Z\cap[1,m],j\in\mathbb Z\cap([1,\left\lfloor m/2 \right\rfloor+1-s]\cup[\left\lfloor m/2 \right\rfloor+2,m+s])},\\ w_{m,s}(u)\colonequals{}&\det(  D^{i-1}\mathcal F_{m,j}(u))_{i\in\mathbb Z\cap[1,m-1],j\in\mathbb Z\cap([1,\left\lfloor m/2 \right\rfloor+1-s]\cup[\left\lfloor m/2 \right\rfloor+2,m-1+s])},\label{eq:wms_defn}\end{align}
so that we have\begin{align}\begin{split}&
C^{\infty}(\theta_{m,s},\theta_{m,s+1})\cap\ker\widetilde L_{m}\\={}&\Span_{\mathbb C}\big\{\mathcal F_{m,j}(u),u\in(\theta_{m,s},\theta_{m,s+1})\big|j\in\mathbb Z\cap\left[1,\left\lfloor \tfrac{m}{2} \right\rfloor+1-s\right]\cup\left[\left\lfloor \tfrac{m}{2} \right\rfloor+2,m+s\right]\big\}\end{split}\label{eq:ker_odd}
\end{align}and  {\allowdisplaybreaks\begin{align}W_{m,s}(u)={}&{\frac{1+(m+1)\delta_{s,\left\lfloor m/2 \right\rfloor+1}}{(-1)^{\frac{s(s+1)}{2}-\left(\left\lfloor \frac{m}{2} \right\rfloor+1\right)s}}}\frac{\Lambda_{m}}{ | \mathcal L_{m}(u)|^{m/2}},\quad u\in(\theta_{m,s},\theta_{m,s+1}),
\label{eq:detOmega_ext}\\w_{m,s}(\theta_{m,s})={}&\frac{1+(m+1)\delta_{s,\left\lfloor m/2 \right\rfloor+1}}{(-1)^{\frac{s(s+1)}{2}-\left(\left\lfloor \frac{m}{2} \right\rfloor+1\right)s}}\frac{2^{m+\left\lfloor \frac{m-1}{2} \right\rfloor}\theta_{m,s}^{\frac{m-1}{4}}}{(m-2)!!}\frac{(-\sqrt{\pi})^{\tfrac{1+(-1)^{m}}{2}}\Lambda_{m}}{ | \mathcal L_{m}'(\theta_{m,s})|^{m/2}},
\label{eq:det_mu_del1}\end{align}}where    $ \Lambda_{m}$ is defined in \eqref{eq:Lambda_m_defn} and  $ \mathcal L_{m}(u)$ in \eqref{eq:Lm_u_defn}. The expression \begin{align}\begin{split}\mathcal L_{m}'(\theta_{m,s})={}&\theta_{m,s}^{\left\lfloor \frac{m+1}{2}\right\rfloor } \prod _{j\in\mathbb Z\cap([1,s-1]\cup[s+1,\left\lfloor m/2 \right\rfloor+1])}[\theta_{m,s}-(2j-\tfrac{1+(-1)^{m}}{2})^2]\\={}& (-1)^{\left\lfloor \frac{m}{2} \right\rfloor+1-s}2^{m} \theta_{m,s}^{\frac{m-1}{2}} (\left\lfloor \tfrac{m}{2} \right\rfloor+1-s)! (\left\lfloor \tfrac{m+1}{2} \right\rfloor+s)!\end{split}\end{align} evaluates the derivative $ \mathcal L_{m}'(u)=D^1\mathcal L_{m}^{\phantom'}(u)$ at $u=\theta_{m,s}$.

\end{lemma}\begin{proof}In view of the arguments in \cite[Lemma 4.2]{Zhou2017BMdet}, one can show that each  off-shell Bessel moment $\mathcal F_{m,j}(u)$ listed in \eqref{eq:ker_odd} indeed resides in the kernel space of  Vanhove's operator of  order $m$. What remains to be elucidated is the mechanism by which $ W_{m,s-1}(u),u\in(\theta_{m,s-1},\theta_{m,s})$ and $ W_{m,s}(u),u\in(\theta_{m,s},\theta_{m,s+1})$  are relayed to each other [see
\eqref{eq:detOmega_ext}], as well as the rationale behind
\eqref{eq:det_mu_del1}.

Without loss of generality, we assume that $m=2k-1$ is an odd number.

 By cofactor expansion, one can compute\begin{align}\begin{split}
&\lim_{u\to(2s)^{2}-0^+}\frac{W_{2k,s-1}(u)}{|(2s)^{2}-u|^{\frac{1}{2}-k}}=(-1)^{k-s}w_{2k-1,s}((2s)^{2})\lim_{u\to(2s)^{2}-0^+}\frac{D^{2k-2}\mathcal F_{2k-1,k+1-s}(u)}{|(2s)^{2}-u|^{\frac{1}{2}-k}},\\
&\lim_{u\to(2s)^{2}+0^+}\frac{W_{2k,s}(u)}{|(2s)^{2}-u|^{\frac{1}{2}-k}}=w_{2k-1,s}((2s)^{2})\lim_{u\to(2s)^{2}+0^+}\frac{D^{2k-2}\mathcal F_{2k-1,2k-1+s}(u)}{|(2s)^{2}-u|^{\frac{1}{2}-k}},\end{split}
\end{align}
which involve a common factor $ w_{2k-1,s}((2s)^{2})$. When $ s\in\mathbb Z\cap[1,k)$, by a procedure akin to the proof of \eqref{eq:lHopitalHankel},  one can show that \begin{align}\begin{split}&
\lim_{u\to(2s)^{2}-0^+}|(2s)^{2}-u|^{k-\frac{1}{2}}D^{2k-2}\mathcal F_{2k-1,k+1-s}(u)\\={}&\frac{1}{(4s)^{2k-2}}\lim_{u\to(2s)^{2}-0^+}|(2s)^{2}-u|^{k-\frac{1}{2}}\mathcal F^k_{2k-1,k+1-s}(u)\\={}&\lim_{u\to(2s)^{2}-0^+}\frac{|(2s)^{2}-u|^{k-\frac{1}{2}}}{\pi^{s}(4s)^{2k-2}}\int_0^\infty\frac{e^{\sqrt{u}t}}{\sqrt{2\pi \sqrt{u} t}}\left( \frac{e^{t}}{\sqrt{2\pi t}} \right)^{k-s}\left( \sqrt{\frac{\pi}{2t}}e^{-t} \right)^{k+s}t^{2k-1}\D t\\={}&\frac{ \Gamma \left(k-\frac{1}{2}\right)}{\sqrt{\pi }}2^{2-3k}s^{1-k}=(2k-3)!!2^{3-4k}s^{1-k}\end{split}\end{align}and that $ \lim_{u\to(2s)^{2}+0^+}|(2s)^{2}-u|^{k-\frac{1}{2}}D^{2k-2}\mathcal F_{2k-1,2k-1+s}(u)$ evaluates to exactly the same number. This shows that the asymptotic behavior of our Wro\'nskian determinants match at $ u\to(2s)^2\pm0^+$, up to a trackable sign. Therefore, we can verify \eqref{eq:detOmega_ext} for  $ s\in\mathbb Z\cap[1,k)$ by successively  applying such a matching procedure to \eqref{eq:detWm}. When $s=k$, one must exert extra care, in that \begin{align}
\lim_{u\to(2k)^{2}-0^+}|(2k)^{2}-u|^{k-\frac{1}{2}}D^{2k-2}\mathcal F_{2k-1,1}(u)={}&\lim_{u\to(2k)^{2}-0^+}\frac{|(2k)^{2}-u|^{k-\frac{1}{2}}}{\pi^{k}}\frac{D^{2k-2}\IvKM(1,2k;1|u)}{2k+1}
\end{align}carries a factor of $\frac1{2k+1} $.

As a by-product,  the common factor $w_{2k-1,s}((2s)^{2})$ during this matching procedure can  be identified with the right-hand side of \eqref{eq:det_mu_del1}.\end{proof}

\subsection{Quadratic relations for off-shell Bessel moments\label{subsec:offshell_quad}}
With the preparations in the last two lemmata, we can compute $ \mathbf S_{m}$ for all $ m\in\mathbb Z_{>0}$.
\begin{proposition}[Explicit formulae for $ \mathbf S_{m}$]\label{prop:Sigma_sigma}\begin{enumerate}[leftmargin=*,
label=\emph{(\alph*)},ref=(\alph*),
widest=a, align=left]
\item If $ a\in\mathbb Z\cap[1,k]$, then\begin{align}\begin{split}{}&\frac{(a-1)!(2k+1-a)!(\mathbf S_{2k-1})_{a,b}}{2^{2k-1}(-1)^{\left\lfloor\frac{a\vphantom{b}}{2}\right\rfloor-k}(1+2 k \delta _{a,1})}\\={}&\begin{cases}-\frac{1+(-1)^{a+1}}{2 }\binom{2 k-a}{k}(2k+1), & b=1, \\
\frac{1+(-1)^{a+b}}{2(-1)^{\left\lfloor \frac{b}{2}\right\rfloor }} {\sum\limits _{s=1}^{k-a+1} (-1)^s} \binom{2 k+1-a}{k+s} \binom{k-s}{b-1},  & b\in\mathbb Z\cap[2,k],\ \\
\frac{(-1)^{a-1}+(-1)^{b-k}}{2(-1)^{\left\lfloor \frac{b-k}{2}\right\rfloor }}{\sum\limits _{s=1}^{k-a+1} (-1)^s }\binom{2 k+1-a}{k+s}\left[\binom{k-s}{b-k+1}+(-1)^{b-k} \binom{k+s}{b-k+1}\right],  & b\in\mathbb Z\cap[k+1,2k-1],\ \\
\end{cases}\end{split}\label{eq:Sigma_upper_panel}\end{align}where the first row can be rewritten as\begin{align}
\frac{(k!)^{2}(\mathbf S_{2k-1})_{1,b}}{(-4)^{k-1}(2k+1)}={}&\begin{cases}2k+1, & b=1, \\
\frac{1+(-1)^{b+1}}{(-1)^{\left\lfloor \frac{b}{2}\right\rfloor }} \frac{b}{2k+1-b}\binom{k}{b} , & b\in\mathbb Z\cap[2,k],\\
\frac{1+(-1)^{b-k}}{(-1)^{\left\lfloor \frac{b-k}{2}\right\rfloor }}\binom{k}{b-k+1} , & b\in\mathbb Z\cap[k+1,2k-1],\ \\
\end{cases}\label{eq:Sigma_1st_row_alt}
\end{align}which entails\begin{align}
(\mathbf S_{2k-1})_{a,b}=\frac{(-4)^{k-1}}{(k!)^{2}}\frac{[1+(-1)^{a-k}][1+(-1)^{b-k}]}{(-1)^{\left\lfloor \frac{a-k}{2}\right\rfloor+\left\lfloor \frac{b-k}{2}\right\rfloor }}\binom{k}{a-k+1}\binom{k}{b-k+1}
\label{eq:Sigma_bottom_right}\end{align}for $a,b\in\mathbb Z\cap[k+1,2k-1] $. \item

\noindent If $ a\in\mathbb Z\cap[1,k]$, then\begin{align}\begin{split}&
\frac{(a-1)! (2 k+2-a)!(\mathbf S_{2k}
)_{a,b}}{2^{2k-1}(-1)^k[1+(2 k+1) \delta _{a,1}]}\\={}&\begin{cases}\frac{1+(-1)^{a+b+1}}{(-1)^{\left\lfloor \frac{a\vphantom{b}}{2}\right\rfloor +\left\lfloor \frac{b-1}{2}\right\rfloor -1}} [1+(2k+1)\delta_{b,1}]\sum\limits _{s=1}^{k+2-a} (-1)^s \binom{2 k+2-a}{k+s}\binom{k+1-s}{b-1},& b\in\mathbb Z\cap[1,k+1],
\\
\frac{1+(-1)^{a+b-k-1}}{(-1)^{\left\lfloor \frac{a-1\vphantom{b}}{2}\right\rfloor +\left\lfloor \frac{b-k}{2}\right\rfloor }}\sum\limits _{s=1}^{k+2-a} (-1)^s \binom{2 k+2-a}{k+s} \left[\binom{k+1-s}{b-k}+(-1)^{b-k-1} \binom{k+s}{b-k}\right], & b\in\mathbb Z\cap[k+2,2k], \\
\end{cases}\end{split}\label{eq:sigma_upper_panel}
\end{align}where the first row can be rewritten as \begin{align}
(\mathbf S_{2k}
)_{1,b}=\begin{cases}\frac{ 2^{2 k} b}{(k!)^2 }\frac{1+(-1)^b}{(-1)^{\left\lfloor \frac{b-1}{2}\right\rfloor +k}}\frac{\binom{k+1}{b}}{2 k+2-b}, & b\in\mathbb Z\cap[1,k+1] \\
0, & b\in\mathbb Z\cap[k+2,2k]. \\
\end{cases}\label{eq:sigma_1st_row_alt}
\end{align}
If $ a,b\in \mathbb Z\cap[k+2,2k]$, then $ (\mathbf S_{2k})_{a,b}=0$.

 \end{enumerate}\end{proposition}
\begin{proof}In the proof below, we will repeatedly use the   Wro\'nskian determinant $ w_{m,s}(u)$ introduced in \eqref{eq:wms_defn}.

\begin{enumerate}[leftmargin=*,  label=(\alph*),ref=(\alph*),
widest=a, align=left]
\item The inputs from Lemma \ref{lm:loc}\ref{itm:loc_mu} and Lemma \ref{lm:spec_det} ($m=2k-1$) allow us to set up the following identities:\begin{align}\begin{split}&\frac{4^{2-3k}(-s^2)^{1-k}}{(-1)^{ks-\frac{s(s+1)}{2}}}\frac{| \mathcal L_{2k-1}'((2s)^{2})|}{1+2k\delta_{s,k}}
\frac{|\mathcal L_{2k-1}(u)|^{k-\frac{3}{2}}}{\Lambda_{2k-1}}w_{2k-1,s}(u)\\
={}&\begin{cases}(-1)^{\left\lfloor\frac{k-s}{2}\right\rfloor+(k-1)(k-s)}\R\mathsf F_{k,s}^1(u), & (-1)^{k-s}=+1, \\
(-1)^{\left\lfloor\frac{k-s-1}{2}\right\rfloor+(k-1)(k-s-1)}\I\mathsf F_{k,s}^1(u), & (-1)^{k-s}=-1, \\
\end{cases}\end{split}\label{eq:cof_mu_l}
\end{align}when  $ u\in((2s-2)^2,(2s)^2)$, and \begin{align}\begin{split}&
\frac{4^{2-3k}(-s^2)^{1-k}}{(-1)^{ks-\frac{s(s+1)}{2}}}\frac{| \mathcal L_{2k-1}'((2s)^{2})|}{1+2k\delta_{s,k}}\frac{|\mathcal L_{2k-1}(u)|^{k-\frac{3}{2}}}{\Lambda_{2k-1}}w_{2k-1,s}(u)
\\={}&\begin{cases}(-1)^{\left\lfloor\frac{k-s}{2}\right\rfloor+(k-1)(k-s)}\R\mathsf F_{k,s}^1(u), & (-1)^{k-s}=-1, \\
(-1)^{\left\lfloor\frac{k-s-1}{2}\right\rfloor+(k-1)(k-s-1)}\I\mathsf F_{k,s}^1(u), & (-1)^{k-s}=+1, \\
\end{cases}\end{split}\label{eq:cof_mu_g}\end{align}when $ u\in((2s)^2,(2s+2)^2)$.  Here, the left-hand sides of \eqref{eq:cof_mu_l} and \eqref{eq:cof_mu_g} behave like a constant multiple of $1+[u-(2s)^2]\mathsf f_{k,s}(u) $ as $ u\to(2s)^2$, where $\mathsf f_{k,s}(u)$ is the real-analytic function mentioned in Lemma \ref{lm:loc}\ref{itm:loc_mu}. Thus, one can establish equalities in  \eqref{eq:cof_mu_l} and \eqref{eq:cof_mu_g} after comparing leading order asymptotic behavior.

Without loss of generality, we momentarily consider the cases where $k$ is odd.
We begin with the $k$th row of $ \mathbf S_{2k-1}$ [cf.\ \eqref{eq:Sm_fit}], which  concerns \begin{align}
W\big[\mathcal F_{2k-1,1}(u),\dots,\reallywidehat{\mathcal F_{2k-1,k}(u)},\dots,\mathcal F_{2k-1,2k-1}(u)\big]=w_{2k-1,1}(u).
\end{align}Setting  $s=1 $ in  \eqref{eq:cof_mu_l}, while simplifying $| \mathcal L_{2k-1}'(4)|=2^{4 k-3} (k+1) (k-1)! k! $ and \linebreak $(-1)^{\left\lfloor\frac{k-s}{2}\right\rfloor+(k-1)(k-s)+(k-1)+ks-\frac{s(s+1)}{2}}|_{s=1}=(-1)^{\left\lfloor\frac{k-1}{2}\right\rfloor+k-1} $, we may spell out{\allowdisplaybreaks\begin{align}
\begin{split}&\frac{|\mathcal L_{2k-1}(u)|^{k-\frac{3}{2}}}{\Lambda_{2k-1}}(-1)^kw_{2k-1,1}(u)=\frac{2^{2k-1}(-1)^{\left\lfloor\frac{k-1}{2}\right\rfloor-1}}{(k-1)!(k_{}+1)!}\R\mathsf F_{k,1}^1(u)\\={}&\frac{2^{2k-1}(-1)^{\left\lfloor\frac{k-1}{2}\right\rfloor-1}}{(k-1)!(k_{}+1)!}\left\{(2k+1)\mathcal F_{2k-1,1}^{\ell}(u)+\sum_{j=1}^{\left\lfloor\frac{k-1}{2}\right\rfloor}(-1)^{j}{k-1\choose 2j} \mathcal F_{2k-1,2j+1}^{\ell}(u)\right.\\{}&+\left.\sum_{j=1}^{\left\lfloor\frac{k}{2}\right\rfloor}(-1)^{j}\left[{k-1\choose 2j+1}+{k+1\choose 2j+1}\right]\mathcal F_{2k-1,2j+k}^{\ell}(u)\right\},\quad u\in(0,16),\end{split}
\end{align}}which explains why \begin{align}\frac{(k-1)!(k_{}+1)!(\mathbf S_{2k-1})_{k,b}}{2^{2k-1}(-1)^{\left\lfloor\frac{k-1}{2}\right\rfloor-1}}={}&\begin{cases}2k+1, & b=1, \\
\frac{1-(-1)^{b}}{2(-1)^{\left\lfloor\frac b2\right\rfloor}}{k-1\choose b-1} ,  & b\in\mathbb Z\cap[2,k], \\
\frac{1-(-1)^{b}}{2(-1)^{\left\lfloor\frac {b-k}2\right\rfloor}}\left[{k-1\choose b-k+1}+{k+1\choose b-k+1}\right] , & b\in\mathbb Z\cap[k+1,2k-1], \\
\end{cases}\end{align}when $ k$ is odd. Then we proceed to the $ (k-1)$-st row. We observe that \begin{align}
\begin{split}&\frac{(-1)^{k-1}W\big[\mathcal F_{2k-1,1}(u),\dots,\reallywidehat{\mathcal F_{2k-1,k-1}(u)},\dots,\mathcal F_{2k-1,2k-1}(u)\big]}{|\mathcal L_{2k-1}(u)|^{\frac{3}{2}-k}\Lambda_{2k-1}}\\={}&\frac{W\big[\mathcal F_{2k-1,1}(u),\dots,\mathcal F_{2k-1,k-2}(u),\mathcal F_{2k-1,k}(u),\dots,\mathcal F_{2k-1,2k-1}(u)\big]}{|\mathcal L_{2k-1}(u)|^{\frac{3}{2}-k}\Lambda_{2k-1}}\\={}&\frac{W\big[\mathcal F_{2k-1,1}(u),\dots,\mathcal F_{2k-1,k-2}(u),\R\mathsf F_{k,1}^1(u),\mathcal F_{2k-1,k+1}(u),\ldots,\mathcal F_{2k-1,2k-1}(u)\big]}{(-1)^{\left\lfloor \frac{k-1}{2}\right\rfloor}|\mathcal L_{2k-1}(u)|^{\frac{3}{2}-k}\Lambda_{2k-1}}\end{split}
\end{align}can be developed into a Taylor series centered at $ u_0=4$ for $u\in(4-\varepsilon,4)$, by virtue of \eqref{eq:mu_loc1}. Such a function admits real-analytic continuation to {\allowdisplaybreaks\begin{align}
-\frac{W\big[\mathcal F_{2k-1,1}(u),\dots,\mathcal F_{2k-1,k-2}(u),\I\mathsf F_{k,1}^1(u),\mathcal F_{2k-1,k+1}(u),\dots,\mathcal F_{2k-1,2k-1}(u)\big]}{(-1)^{\left\lfloor \frac{k-1}{2}\right\rfloor}|\mathcal L_{2k-1}(u)|^{\frac{3}{2}-k}\Lambda_{2k-1}}\label{eq:W_mu_rac1}
\end{align}}for $ u\in(4,16)$, by a comparison between \eqref{eq:mu_loc1} and \eqref{eq:mu_loc2}. Since \eqref{eq:GI_sum} and \eqref{eq:GK_sum} tell us that\begin{align}\begin{split}&\I\mathsf F_{k,1}^1(u)-\left\{ (-1)^{\left\lfloor \frac{k}{2}\right\rfloor-1} (k-1)\mathcal F_{2k-1,k-1}(u)+(-1)^{\left\lfloor \frac{k+1}{2}\right\rfloor}\mathcal F_{2k-1,2k}(u)\right\}\\
\in{}&\Span_{\mathbb R}\{\mathcal F_{2k-1,j}(u),u\in(4,16)|j\in\mathbb Z\cap([1,k-2]\cup[k+1,2k-1])\},\end{split}
\end{align}we can convert  \eqref{eq:W_mu_rac1} into [cf.~\eqref{eq:cof_mu_l} and \eqref{eq:cof_mu_g}] a $ \mathbb Q$-linear combination of $ |\mathcal L_{2k-1}(u)|^{k-\frac{3}{2}}w_{2k-1,1}(u)$ and $ |\mathcal L_{2k-1}(u)|^{k-\frac{3}{2}}w_{2k-1,2}(u)$, in the following manner: \begin{align}\begin{split}&
\frac{|\mathcal L_{2k-1}(u)|^{k-\frac{3}{2}}}{\Lambda_{2k-1}}\big\{W\big[\mathcal F_{2k-1,1}(u),\dots,\mathcal F_{2k-1,k-2}(u),(k-1)\mathcal F_{2k-1,k-1}(u),\mathcal F_{2k-1,k+1}(u),\dots,\mathcal F_{2k-1,2k-1}(u)\big]
\\{}&+W\big[\mathcal F_{2k-1,1}(u),\dots,\mathcal F_{2k-1,k-2}(u),\mathcal F_{2k-1,2k}(u),\mathcal F_{2k-1,k+1}(u),\dots,\mathcal F_{2k-1,2k-1}(u)\big]\big\}\\={}&\frac{2^{2k-1}(-1)^{\left\lfloor\frac{k-1}{2}\right\rfloor-1}}{(k-1)!(k_{}+1)!}\left[(k-1)\I\mathsf F_{k,1}^1(u)-\frac{k-1}{k+2}\I\mathsf F_{k,2}^1(u)\right]\\={}&\frac{2^{2k-1}(-1)^{\left\lfloor\frac{k-1}{2}\right\rfloor-1}}{(k-2)!(k_{}+2)!}\left\{ \sum_{j=0}^{\left\lfloor\frac{k}{2}\right\rfloor-1}(-1)^{j}\left[(k+2){k-1\choose 2j+1} -{k-2\choose 2j+1} \right]\mathcal F_{2k-1,2j+2}^{\ell}(u)\right.\\{}&\left.-\sum_{j=1}^{\left\lfloor\frac{k-1}{2}\right\rfloor}(-1)^{j}\left[(k+2){k-1\choose 2j}-(k+2){k+1\choose 2j}-{k-2\choose 2j}+{k+2\choose 2j}\right]\mathcal F_{2k-1,2j+k-1}^{\ell}(u) \right\}.\end{split}\label{eq:1st_canc}\end{align}On the right-hand side of this expression, the last summands of $ \I\mathsf F_{k,1}^1(u)$ and $ \frac{1}{k+2}\I\mathsf F_{k,2}^1(u)$ \big[which are $ (-1)^\frac{k+1}{2}\binom{k+1}{k+1}\mathcal F_{2k-1,2k}^{\ell}(u)$ and  $ (-1)^\frac{k+1}{2}\frac{1}{k+2}\binom{k+2}{k+1}\mathcal F_{2k-1,2k}^{\ell}(u)$, respectively\big] cancel each other. Thanks to this cancelation of terms undefined for $ u\in(0,4)$, the right-hand side of \eqref{eq:1st_canc} real-analytically continues to $ u\in(0,16)$, and is compatible with our statement about $(\mathbf S_{2k-1})_{k-1,b} $ in \eqref{eq:Sigma_upper_panel}.

Proceeding as  the last paragraph, we obtain the  following real-analytic continuations of  \eqref{eq:cof_mu_l}:\footnote{Similar to Proposition \ref{prop:Vv}, here we require that all the  column indices    be arranged in increasing order (even when they are not consecutive integers).}\begin{align}\begin{split}&
\frac{(a-1)! ( 2k+1-a)!}{  2^{2 k-1} (1+2 k \delta _{a,1})}\frac{|\mathcal L_{2k-1}(u)|^{k-\frac{3}{2}}}{\Lambda_{2k-1}}\frac{\det(D^{i-1}\mathcal F_{2k-1,j}(u))_{i\in\mathbb Z\cap[1,2k-2],j\in\mathbb Z\cap([1,a-1]\cup[a+1,k-t]\cup[k+1,2k-1+t])}}{(-1)^{k t-\frac{t (t-1)}{2} }(-1)^{\left\lfloor \frac{a\vphantom1}{2}\right\rfloor -k}(-1)^{a+1}}\\={}& \sum_{b=1}^{k-t}\left[(1+2 k \delta _{b,1})\frac{1+(-1)^{a+b}}{2(-1)^{\left\lfloor \frac{b}{2}\right\rfloor }} \sum _{s=t+1}^{k+1-a} (-1)^s \binom{2 k+1-a}{k+s} \binom{k-s}{b-1}\right] \mathcal F_{2k-1,b}(u)\\{}&+\sum_{b=k+1}^{2k-1+t}\left\{\frac{(-1)^{a-1}+(-1)^{b+k}}{2(-1)^{\left\lfloor \frac{b-k}{2}\right\rfloor }} \sum _{s=t+1}^{k+1-a} (-1)^s \binom{2 k+1-a}{k+s}\times\right.\\&\left.\times\left[ \binom{k-s}{b-k+1}-(-1)^{b-k+1} \binom{k+s}{b-k+1}\right]\vphantom{\sum^{k+1-a}_{s=t+1}}\right\} \mathcal F_{2k-1,b}(u),\end{split}\label{eq:cof_mu_l*}\tag{\ref{eq:cof_mu_l}*}
\end{align}where $ t\in\mathbb  Z\cap[0,k-1],a=k-1-t,u\in ((2t)^2,(2t+2)^2)$. [If we set $t\in\mathbb  Z\cap[0,k-1],a=k-t,u\in ((2t)^2,(2t+2)^2)$ in  \eqref{eq:cof_mu_l*}, we recover the original   \eqref{eq:cof_mu_l}.] Recursively performing real-analytic continuations with the aid of~\eqref{eq:cof_mu_g},  down-stepping $ \sqrt u$ by an amount of $ 2$ every time, we can further extend the validity of
 \eqref{eq:cof_mu_l*} to  $ t\in\mathbb  Z\cap[0,k-1],a\in\mathbb  Z\cap[1,k-t],u\in ((2t)^2,(2t+2)^2)$.   In particular, setting $t=0$ in
 \eqref{eq:cof_mu_l*}, one can verify the expression
$(\mathbf S_{2k-1})_{a,b} $ in \eqref{eq:Sigma_upper_panel}, for each $ a\in\mathbb Z\cap[1,k]$.

One can readily generalize the procedures in the last two paragraphs [which essentially revolve around real-analytic continuations of $ w_{2k-1,s}(u)$]
to even numbers $k$, thus establishing   \eqref{eq:Sigma_upper_panel}  in its entirety.

Plugging two combinatorial identities\begin{align}\sum_{s=1}^{k} (-1)^s \binom{2 k}{k+s} \binom{k-s}{n}={}&-\frac{2k}{2k-n}\binom{2k-1}{k}\binom{k-1}{n},\\\sum_{s=1}^{k} (-1)^s \binom{2 k}{k+s} \binom{k+s}{n}={}&-\frac{2k}{2k-n}\binom{2k-1}{k}\binom{k}{n}\end{align}into     \eqref{eq:Sigma_upper_panel}  for the expression $ (\mathbf S_{2k-1})_{1,b}$, we obtain \eqref{eq:Sigma_1st_row_alt}.

Let $ k\in\mathbb Z_{\geq2}$ and $ r\in\mathbb Z\cap[k+1,2k-1]$. Through elementary manipulations of rows and columns in a determinant (bearing in mind that  $ |\det\pmb{\boldsymbol \beta}_{2k-2}(u)|=2^{(k-1) (2 k-3)}|u|^{(k-1)^2}$), one can identify\begin{align}\begin{split}&
|\mathcal L_{2k-1}(u)|^{k-\frac{3}{2}}\left\{(-1)^{r+1}W\big[\mathcal F_{2k-1,1}(u),\dots,\reallywidehat{\mathcal F_{2k-1,r}(u)},\dots,\mathcal F_{2k-1,2k-1}(u)\big]\vphantom{\binom{k}{k}\frac{1+(-1)^{r-k}}{(-1)^{\left\lfloor \frac{r-k}{2}\right\rfloor }}}\right.\\{}&\left.-\frac{1+(-1)^{r-k}}{(-1)^{\left\lfloor \frac{r-k}{2}\right\rfloor }}\binom{k}{r-k+1} W\big[\reallywidehat{\mathcal F_{2k-1,1}(u)},\dots,{\mathcal F_{2k-1,r}(u),}\dots,\mathcal F_{2k-1,2k-1}(u)\big]\right\}\end{split}
\end{align}for $u\in(0,4)$ with a constant multiple of \begin{align}
\left(\prod_{j=1}^k[(2j)^2-u]\right)^{k-\frac{3}{2}}\det\begin{pmatrix*}[r]\pmb{\boldsymbol \iota}_{k-1}(u) & \acute{\pmb{\boldsymbol \iota}}_{k-1}(u)\\
\pmb{\boldsymbol \kappa}_{k-1,r}(u) & \acute{\pmb{\boldsymbol \kappa}}_{k-1,r}(u) \\
\end{pmatrix*},\label{eq:hol_det}
\end{align} where the blocks are designed as follows:\begin{enumerate}[leftmargin=*,  label=(\arabic*),ref=(\arabic*),
widest=a, align=left]
\item
We have  $ \pmb{\boldsymbol \iota}_{k-1}(u)=(\mathcal F^b_{2k-1,a+1}(u))_{1\leq a,b\leq k-1}$ and   $ \acute{\pmb{\boldsymbol \iota}}_{k-1}(u)=(\sqrt{u}\acute{\mathcal F}^b_{2k-1,a+1}(u))_{1\leq a,b\leq k-1}$;\item The first row of $ \pmb{\boldsymbol \kappa}_{k-1,r}(u)$ [resp.~$ \acute{\pmb{\boldsymbol \kappa}}_{k-1,r}(u)$] is given by
$ (\pmb{\boldsymbol \kappa}_{k-1,r}(u))_{1,b}=\pi \mathcal F^b_{2k-1,1}(u)+\frac{2k}{2k+1} \mathcal F^b_{2k-1,2}(u)\log\frac{\sqrt{u}}{2}$ \Big[resp.
$ (\acute{\pmb{\boldsymbol \kappa}}_{k-1,r}(u))_{1,b}=\sqrt{u}\big[\pi \acute{\mathcal F}^b_{2k-1,1}(u)+\frac{2k}{2k+1} \acute{\mathcal F}^b_{2k-1,2}(u)\log\frac{\sqrt{u}}{2}\big ]$\Big];\item The next $k-3$ rows in $ \pmb{\boldsymbol \kappa}_{k-1,r}(u)$ [resp.~$ \acute{\pmb{\boldsymbol \kappa}}_{k-1,r}(u)$] form a subset of
  $ \big\{\big(\pi\mathcal F^b_{2k-1,a}(u)+\mathcal F^b_{2k-1,a-k+2}(u)\linebreak\log\frac{\sqrt{u}}2\big)_{1\leq b\leq k-1}\big|a\in\mathbb Z\cap[k+1,2k-2]\big\}$ \Big[resp.~$ \big\{\big(\sqrt{u}\big[\pi\acute{\mathcal F}^b_{2k-1,a}(u)+\acute{\mathcal F}^b_{2k-1,a-k+2}(u)\log\frac{\sqrt{u}}2\big]\big)_{1\leq b\leq k-1}\big|a\in\mathbb Z\cap[k+1,2k-2]\big\}$\Big];\item The last row of
$ \pmb{\boldsymbol \kappa}_{k-1,r}(u)$  is specified by [cf.~\eqref{eq:Okl} and \eqref{eq:Ekl}]\begin{align}
(\pmb{\boldsymbol \kappa}_{k-1,r}(u))_{k-1,b}={}&\begin{cases}u^{(k-2)/2}O_k^b(u), & k\in1+2\mathbb Z_{>0}, \\
u^{(k-2)/2}E_k^b(u), & k\in2\mathbb Z_{>0},r\neq 2k-1, \\
u^{(k-2)/2}\left[ \pi\mathcal{F}^b_{2k-1,2k-2}(u)+\mathcal{F}^b_{2k-1,k}(u)\log\frac{\sqrt{u}}{2} \right], & k\in2\mathbb Z_{>0},r= 2k-1, \\
\end{cases}\label{eq:kappa_last_row}\end{align}{while the last row of $ \acute{\pmb{\boldsymbol \kappa}}_{k-1,r}(u)$ is specified by [cf.~\eqref{eq:O'kl} and \eqref{eq:E'kl}]}\begin{align}(\acute{\pmb{\boldsymbol \kappa}}_{k-1,r}(u))_{k-1,b}={}&\begin{cases}u^{(k-1)/2}\acute O_k^b(u), & k\in1+2\mathbb Z_{>0}, \\
u^{(k-1)/2}\acute E_k^b(u), & k\in2\mathbb Z_{>0},r\neq 2k-1, \\
u^{(k-1)/2}\left[ \pi\acute{\mathcal F}^b_{2k-1,2k-2}(u)+\acute{\mathcal F}^b_{2k-1,k}(u)\log\frac{\sqrt{u}}{2} \right], & k\in2\mathbb Z_{>0},r= 2k-1. \\
\end{cases}\tag{\ref{eq:kappa_last_row}$'$}
\end{align}\end{enumerate}[For $k=2$, the instructions in  items (3) and (4) must be ignored, whereupon  we are reduced to the case already treated in Corollary \ref{cor:S3_comp}.]  According to the arguments in Proposition \ref{prop:u0_sum}(a), the expression in \eqref{eq:hol_det} defines a holomorphic function in a non-void $ \mathbb C$-neighborhood of $u=0$, hence a member of $ \Span_{\mathbb C}\{\mathcal{F}_{2k-1,a}(u)|a\in\mathbb Z\cap[2,k]\}$. This completes the proof of \eqref{eq:Sigma_bottom_right}.

 \item By \eqref{eq:log_jump}, the expression \begin{align}[u-(2s-1)^{2}]^{k-1}\left\{D^n\mathcal
I^1_{2k,s} (u)+\frac{(-1)^{k-s-1}-i}{\pi(1-i)}[D^n\mathfrak G_{k,s}(u)]\left[\log|u-(2s-1)^{2}|-\frac{i\pi}{2}\right]-D^n\mathfrak L_{k,s}(u)\right\}
\end{align}for $ n\in\mathbb Z\cap[0,2k-1]$ is equal to a convergent Taylor series in powers of $[u-(2s-1)^{2}]^{\mathbb Z_{\geq0}} $  for $ u\in((2s-1)^2-\varepsilon,(2s-1)^2)$, where $\varepsilon>0 $. Accordingly, for each $ s\in\mathbb Z\cap[1,k]$, there exists a non-vanishing constant $ c_{k,s}$ that allows us to identify the following formula \begin{align}\begin{split}&
[\mathcal L_{2k}(u)]^{k-1}\left\{W[\mathcal F_{2k,1}(u),\ldots,\mathcal F_{2k,k-s}(u),\mathcal F_{2k,k+2-s}(u),\right.\\{}&\left.\mathcal F_{2k,k+2}(u),\ldots,\mathcal F_{2k,2k-1+s}(u)]+c_{k,s}w_{2k,s}(u)\log|u-(2s-1)^2|\right\}\end{split}
\end{align} with another convergent Taylor series in powers of $[u-(2s-1)^{2}]^{\mathbb Z_{\geq0}} $  for $ u\in((2s-1)^2-\varepsilon,(2s-1)^2)$. In the light of this, the function  $ [\mathcal L_{2k}(u)]^{k-1}w_{2k,s}(u)$ must be a constant multiple of $ \mathfrak G_{k,s}(u)$, and the constant in question can be determined through Lemma \ref{lm:spec_det} ($ m=2k$).

In fact, the arguments in the last paragraph lead us to \begin{align}
\frac{|\mathcal L_{2k}(u)|^{k-1}}{\Lambda_{2k}}w_{2k,s}(u)
=-\frac{2^{4 k} (2 s-1)^{2 k-1}}{(-1)^{\left\lfloor \frac{k-s}{2}\right\rfloor +\frac{(2k-s)(s+1)}{2}}} \frac{\mathfrak G_{k,s}(u)}{\left|\mathcal L'_{k}((2s-1)^2)\right| }\end{align}for $ u\in((\max\{0,2s-3\})^2,(2s+1)^2),s\in\mathbb Z\cap[1,k]$. Setting $s=1$ in the formula above, we recover the expression for $ (\mathbf S_{2k})_{k+1,b},b\in\mathbb Z\cap[1,2k]$, as claimed in \eqref{eq:sigma_upper_panel}. Setting $s=2$ in the last displayed formula, and performing real-analytic continuation of $ \mathfrak L_{k,1}(u)$ [which allows us to trade $ \mathcal F_{2k,2k+1}(u),u\in(1,9)$ for $ \mathcal F_{2k,k+1}(u),u\in(0,1)$, up to some trailing terms contributed by  $ \mathfrak G_{k,1}(u)$], we can verify the formula for  $ (\mathbf S_{2k})_{k,b},b\in\mathbb Z\cap[1,2k]$, as claimed in \eqref{eq:sigma_upper_panel}. Subsequently, after finitely many steps of real-analytic continuations, we can check the combinatorial expression for  $   (\mathbf S_{2k})_{a,b},a\in \mathbb Z\cap[2,k],b\in\mathbb Z\cap[1,2k]$.

We can directly compute $ (\mathbf S_{2k})_{1,b},b\in \mathbb Z\cap[2,k]$ from our knowledge of $ -(\mathbf S_{2k})_{b,1},b\in \mathbb Z\cap[2,k]$.

We can deduce  $ (\mathbf S_{2k})_{a,b}=0$  for  $ a,b\in \mathbb Z\cap(\{1\}\cup[k+2,2k])$
 from  Proposition \ref{prop:u0_sum}(b).\qedhere\end{enumerate}\end{proof}

At this point,  we have verified all the statements in Theorem  \ref{thm:W_alg}, effectively proving the following matrix inversion for $ u\in(0,1)$:\begin{align}[
\mathbf W^{\vphantom{\mathrm T}}_{m}(u)]^{-1}={}&|\mathcal L_m(u)|\mathbf S^{\vphantom{\mathrm T}}_{m}\mathbf W_{m}^{\mathrm T}(u)\mathbf V^{\vphantom{\mathrm T}}_{m}
(u),
\end{align}which in turn, can be recast into the following alternative formulation of quadratic relations for off-shell Bessel moments:\begin{align}
\mathbf W_{m}^{\mathrm T}(u)\mathbf V^{\vphantom{\mathrm T}}_{m}
(u)
\mathbf W^{\vphantom{\mathrm T}}_{m}(u)={}&\frac{\mathbf S^{-1}_{m}}{|\mathcal L_{m}(u)|} .\label{eq:omegavomega}
\end{align}

To prepare for an inductive computation of  $ \mathbf S^{-1}_{m}$   in Proposition \ref{prop:inv_S}, we wrap up this section with a combinatorial analysis of   $ \mathbf S_{m}=\left(\begin{smallmatrix}\mathbf S_m^A & \mathbf S_m^B \\
\mathbf S_m^C & \mathbf S_m^D \\
\end{smallmatrix}\right)$ in the proposition below.
To ease computation, we comb $\mathbf S_{m}$ with  a reshuffling-rescaling matrix  $\pmb{ \boldsymbol\xi}_{m}\in\mathbb Q^{m\times m}$ whose elements are\begin{align}
(\pmb{ \boldsymbol\xi}_{m})_{a,b}:=\begin{cases}\delta_{a,b-1}, & a\in\mathbb Z\cap\left[1,\left\lfloor \frac{m}{2} \right\rfloor\right],\\\frac{m+2}{m+1}\delta_{1,b},&a=\left\lfloor \frac{m}{2} \right\rfloor+1, \\
\delta_{a,b}, & a\in\mathbb Z\cap\left[\left\lfloor \frac{m}{2} \right\rfloor+2,m\right].\\
\end{cases}\label{eq:Rmat}
\end{align}As a result, the block partition of $ \smash{\underset{^\circ}{\mathbf S}}_m^{}\colonequals(\pmb{ \boldsymbol\xi}_{m}^{-1})^{\mathrm T}\mathbf S_m^{}\pmb{ \boldsymbol\xi}_{m}^{-1}=\left(\begin{smallmatrix}\smash{\underset{^\circ}{\mathbf S}}_m^A & \smash{\underset{^\circ}{\mathbf S}}_m^B \\[2pt]
\smash{\underset{^\circ}{\mathbf S}} _m^C &\smash{\underset{^\circ}{\mathbf S}}_m^D \\[1pt]
\end{smallmatrix}\right)$ has the following structure:\begin{align}
\begin{cases}(\smash{\underset{^\circ}{\mathbf S}}_m^A)_{a,b}=(\mathbf S_m^A)_{a+1,b+1}, & a,b\in\mathbb Z\cap\left[1,\left\lfloor \frac{m}{2} \right\rfloor\right], \\
(\smash{\underset{^\circ}{\mathbf S}}_m^B)_{a,b}=(\mathbf S_m^B)_{a+1,b-1}, & a\in\mathbb Z\cap\left[1,\left\lfloor \frac{m}{2} \right\rfloor\right],b\in\mathbb Z\cap\left[1,\left\lfloor \frac{m+1}{2} \right\rfloor\right], \\
(\smash{\underset{^\circ}{\mathbf S}}_m^D)_{a,b}=(\mathbf S_m^D)_{a-1,b-1}, & a,b\in\mathbb Z\cap\left[1,\left\lfloor \frac{m+1}{2} \right\rfloor\right], \\
\end{cases}
\end{align}with parity $ (\smash{\underset{^\circ}{\mathbf S}}_m^B)_{a,b}=(-1)^{m+1}(\smash{\underset{^\circ}{\mathbf S}}_m^C)_{b,a}$, where extensions of \eqref{eq:SAmat_defn}--\eqref{eq:SDmat_defn} lead us to\begin{align}
(\smash{\underset{^\circ}{\mathbf S}}_m^A)_{a,b}=\frac{1+(-1)^{a+b+m+1}}{2^{1-m}a!(m+1-a)!}\frac{{\sum\limits _{s=1}^{\left\lfloor \frac{m}{2} \right\rfloor+1-a} (-1)^s} \binom{m+1-a}{\left\lfloor \frac {m+1}2\right\rfloor+s} \binom{\left\lfloor \frac m2\right\rfloor+1-s}{b}}{(-1)^{\left\lfloor\frac{a+1\vphantom{b}}{2}\right\rfloor+\left\lfloor \frac{b+1}{2}-\frac{1+(-1)^{m}}{4}\right\rfloor -\left\lfloor \frac m2\right\rfloor-1}},\tag{\ref{eq:SAmat_defn}$^\circ$}
\label{eq:S0Amat_defn}\end{align}
\begin{align}(\smash{\underset{^\circ}{\mathbf S}}_m^B)_{a,b}=(-1)^{m+1}(\smash{\underset{^\circ}{\mathbf S}}_m^C)_{b,a}={}&\frac{1+(-1)^{a+b+m}}{2^{1-m}a!(m+1-a)!}\frac{{\sum\limits _{s=1}^{\left\lfloor \frac m2\right\rfloor+1-a} (-1)^s }\binom{m+1-a}{\left\lfloor \frac {m+1}2\right\rfloor+s}\left[\binom{\left\lfloor \frac m2\right\rfloor+1-s}{b}-(-1)^{b} \binom{\left\lfloor \frac {m+1}2\right\rfloor+s}{b}\right]}{(-1)^{am+\left\lfloor\frac{a+1\vphantom{b}}{2}-\frac{1+(-1)^{m}}{4}\right\rfloor+\left\lfloor \frac{\smash{b-1}\vphantom1}{2}+\frac{1+(-1)^{m}}{4}\right\rfloor -\left\lfloor \frac {m+1}2\right\rfloor}},\tag{\ref{eq:SBmat_defn}$^\circ$}\label{eq:S0Bmat_defn}\end{align}
and \begin{align}(\smash{\underset{^\circ}{\mathbf S}}_m^D)_{a,b}={}&\frac{1+(-1)^{m+1}}{2}\frac{(-4)^{\left\lfloor \frac {m+1}2\right\rfloor-1}}{\left(\left\lfloor \frac {m+1}2\right\rfloor!\right)^{2}}\frac{[1-(-1)^{a}][1-(-1)^{b}]}{(-1)^{\left\lfloor \frac{\smash{a-1}\vphantom1}{2}\right\rfloor+\left\lfloor \frac{\smash{b-1}\vphantom1}{2}\right\rfloor }}\binom{\left\lfloor \frac {m+1}2\right\rfloor}{a}\binom{\left\lfloor \frac {m+1}2\right\rfloor}{b}\tag{\ref{eq:SDmat_defn}$^\circ$}\label{eq:S0Dmat_defn}\end{align} for $ a,b\in\mathbb Z_{\geq0}$. (It is still understood that  empty sums vanish and certain combinatorial expressions are zero.)

\begin{proposition}[Recursions for   $ \smash{\underset{^\circ}{\mathbf S}}_m^{}$]\label{prop:comb_S}\begin{enumerate}[leftmargin=*,  label=\emph{(\alph*)},ref=(\alph*),
widest=a, align=left]\item For  $ k\in\mathbb Z_{>1 }$ and $a,b\in\mathbb  Z\cap[1,k]$, we have \begin{align}
a(2k+1-b)(\smash{\underset{^\circ}{\mathbf S}}_{2k}^A)_{a,b}={}&2(2k+1)(\smash{\underset{^\circ}{\mathbf S}}_{2k-1}^A)_{a-1,b}-4(\smash{\underset{^\circ}{\mathbf S}}_{2k-2}^A)_{a-1,b-1},\label{eq:SArec}
\\
a(2k+1-b)(\smash{\underset{^\circ}{\mathbf S}}_{2k}^B)_{a,b}={}&2(2k+1)(\smash{\underset{^\circ}{\mathbf S}}_{2k-1}^B)_{a-1,b}-4(\smash{\underset{^\circ}{\mathbf S}}_{2k-2}^B)_{a-1,b-1}.\label{eq:SBrec}
\end{align}\item For $ k\in\mathbb Z_{>1}$, we have \begin{align}
\pmb{\boldsymbol \varphi}^{-1}_{2k-1}\pmb{\boldsymbol \alpha}^{-1}_{k}\mathbf S_{2k-1}^{\vphantom{\mathrm T}}(\pmb{\boldsymbol \alpha}^{-1}_{k})^{\mathrm T}( \pmb{\boldsymbol \varphi}^{-1}_{2k-1})^{\mathrm T}=\begin{pmatrix}\frac{2k+1}{2}\smash{\underset{^\circ}{\mathbf S}}_{2k}^B
&  \\
 & -\frac{2}{2k+1} \smash{\underset{^\circ}{\mathbf S}}_{2k-2}^B\\
\end{pmatrix}\label{eq:Sigma_block_diag}
\end{align}in block diagonal form, along with a recursion\begin{align}
\frac{4(\smash{\underset{^\circ}{\mathbf S}}_{2k}^B)_{a,b}}{(2k+2-a)(2k+2-b)}=(\smash{\underset{^\circ}{\mathbf S}}_{2k+2}^B )_{a+1,b+1}-{}&\frac{1+(-1)^{a}}{ (-1)^{\left\lfloor \frac{a\vphantom1}{2}\right\rfloor }} \frac{\binom{k+1}{a+1} (\smash{\underset{^\circ}{\mathbf S}}_{2k+2}^B)_{1,b+1}}{2k+2-a}\label{eq:invBettiM_rec}
\end{align} for $ a,b\in\mathbb Z\cap[1,k]$.\end{enumerate} \end{proposition}\begin{proof}\begin{enumerate}[leftmargin=*,  label=(\alph*),ref=(\alph*),
widest=a, align=left]\item Define \begin{align}\mathscr S_{m,a,b,s}^A\colonequals{}&\frac{\binom{m+1-a}{\left\lfloor \frac {m+1}2\right\rfloor+s} \binom{\left\lfloor \frac m2\right\rfloor+1-s}{b}}{a!(m+1-a)!}=\mathscr S_{m,b,a,s}^A \label{eq:SAelem_defn},\\\mathscr S_{m,a,b,s}^B\colonequals{}&\mathscr S_{m,a,b,s}^A- \frac{{ }(-1)^{b} \binom{\left\lfloor \frac {m+1}2\right\rfloor+s}{b}}{a!(m+1-a)!}\label{eq:SBelem_defn}\end{align}for all the integers $ m,a,b,s$ that make these expressions meaningful.

By direct computation, we obtain\begin{align}a(2k+1-b)
\mathscr S_{2k,a,b,s}^{A}={}&(2k+1)
\mathscr S_{2k-1,a-1,b,s}^A+\mathscr S_{2k-2,a-1,b-1,s}^A,
\end{align}for integers $k,a,b,s$ that make each summand well-defined. Multiplying through $ (-1)^s$ and summing over $s\in\mathbb Z_{>1}$, we arrive at the recursion \eqref{eq:SArec} that involves\begin{align}
(\smash{\underset{^\circ}{\mathbf S}}_m^A)_{a,b}=\frac{1+(-1)^{a+b+m+1}}{2^{1-m}}\frac{{\sum\limits _{s=1}^{\infty} (-1)^s} \mathscr S_{m,a,b,s}^A{}}{(-1)^{\left\lfloor\frac{a+1\vphantom{b}}{2}\right\rfloor+\left\lfloor \frac{b+1}{2}-\frac{1+(-1)^{m}}{4}\right\rfloor -\left\lfloor \frac m2\right\rfloor-1}}.
\end{align}
Here, the infinite series over $s$ truncates after finitely many terms, in accordance with the previous definition of $ (\smash{\underset{^\circ}{\mathbf S}}_m^A)_{a,b}$ in \eqref{eq:S0Amat_defn}.

Likewise, one can sum over \begin{align}
a(2k+1-b)
\mathscr S_{2k,a,b,s}^{B}={}&(2k+1)
\mathscr S_{2k-1,a-1,b,s}^B+\mathscr S_{2k-2,a-1,b-1,s}^B
\end{align}to arrive at \eqref{eq:SBrec}. \item We  partition the matrix $\pmb{\boldsymbol \alpha}^{-1}_{k}\mathbf S_{2k-1}^{\vphantom{\mathrm T}}(\pmb{\boldsymbol \alpha}^{-1}_{k})^{\mathrm T}=\pmb{\boldsymbol \alpha}^{-1}_{k}\Big(\begin{smallmatrix}\mathbf S^{A}_{2k-1} & \mathbf S^{B}_{2k-1}  \\
\mathbf S^{C}_{2k-1} & \mathbf S^{D}_{2k-1} \\
\end{smallmatrix}\Big)^{\vphantom{\mathrm T}}(\pmb{\boldsymbol \alpha}^{-1}_{k})^{\mathrm T}=\Big(\begin{smallmatrix}\mathbf T^{A}_{2k-1} & \mathbf T^{B}_{2k-1}  \\
\mathbf  T^{C}_{2k-1} & \mathbf  T^{D}_{2k-1} \\
\end{smallmatrix}\Big)$  into four blocks. Here, the matrix $ \mathbf T^{A}_{2k-1} $ also forms the top-left $k\times k$ block of  $\pmb{\boldsymbol \varphi}^{-1}_{2k-1}\pmb{\boldsymbol \alpha}^{-1}_{k}\mathbf S_{2k-1}^{\vphantom{\mathrm T}}(\pmb{\boldsymbol \alpha}^{-1}_{k})^{\mathrm T}( \pmb{\boldsymbol \varphi}^{-1}_{2k-1})^{\mathrm T}$.

By definition of $ \pmb{\boldsymbol \alpha}_{k}$ [see \eqref{eq:Amat_defn}], we have $ \mathbf T^{D}_{2k-1}=\mathbf S^{D}_{2k-1}$.

Thanks to \eqref{eq:Sigma_1st_row_alt}, we can evaluate the first row and the first column in the top-left $ k\times k$ block $ \mathbf T^{A}_{2k-1} $  in closed form. This allows us to verify the relations  $ (\mathbf T^{A}_{2k-1})_{1,1}\colonequals(\mathbf S^{A}_{2k-1})_{1,1}=\frac{2k+1}{2}(\smash{\underset{^\circ}{\mathbf S}}_{2k}^{B})_{1,1}$ and $(\mathbf T^{A}_{2k-1})_{1,b}\equiv(\mathbf T^{A}_{2k-1})_{b,1}\colonequals(\mathbf S^{A}_{2k-1})_{1,b}+(\mathbf S^{B}_{2k-1})_{1,b-1}=\frac{2k+1}{2}(\smash{\underset{^\circ}{\mathbf S}}_{2k}^{B})_{1,b},b\in\mathbb{Z}\cap[2,k] $  explicitly.
In general, we may construct a sum rule $(\mathbf S^{A}_{2k-1})_{a,b}+(\mathbf S^{B}_{2k-1})_{a,b-1} =\frac{a}{2}(\smash{\underset{^\circ}{\mathbf S}}_{2k}^{B})_{a,b}$ for $a,b\in\mathbb{Z}\cap[2,k]$, using the combinatorial identity  $ \binom{k-s}{b-1}+\binom{k-s}{b}=\binom{k+1-s}{b}$.

By definition, for  $a,b\in\mathbb{Z}\cap[2,k]$, we have $ (\mathbf T^{A}_{2k-1})_{a,b}\colonequals(\mathbf S^{A}_{2k-1})_{a,b}+(\mathbf S^{B}_{2k-1})_{a,b-1}+(\mathbf S^{C}_{2k-1})_{a-1,b}+(\mathbf S^{D}_{2k-1})_{a-1,b-1}$ and $ (\mathbf T^{C}_{2k-1})_{a-1,b}\colonequals(\mathbf S^{C}_{2k-1})_{a-1,b}+(\mathbf S^{D}_{2k-1})_{a-1,b-1}$, which together imply $ (\mathbf T^{A}_{2k-1})_{a,b}=\frac{a}{2}(\smash{\underset{^\circ}{\mathbf S}}_{2k}^{B})_{a,b}+(\mathbf T^{C}_{2k-1})_{a-1,b}$. In order to  identify the last expression with $\frac{2k+1}{2}(\smash{\underset{^\circ}{\mathbf S}}_{2k}^{B})_{a,b}$ [as claimed in the right-hand side of \eqref{eq:Sigma_block_diag}], we will prove the relation  $ (\mathbf T^{C}_{2k-1})_{a-1,b}=\frac{2k+1-a}{2}(\smash{\underset{^\circ}{\mathbf S}}_{2k}^{B})_{a,b}$ in the next paragraph.

We aim for an identity that is slightly stronger than what has been just claimed, namely\begin{align}(\smash{\underset{^\circ}{\mathbf S}}_{2k-1}^{B})_{a-1,b}+(\smash{\underset{^\circ}{\mathbf S}}_{2k-1}^{D})_{a,b}=
\frac{2k+1-b}{2}(\smash{\underset{^\circ}{\mathbf S}}_{2k}^{B})_{a,b}\label{eq:BDB_rec}
\end{align}
for all the possible values of  $ k\in\mathbb Z_{>0}$ and $a,b\in\mathbb Z$ that make both sides well-defined.  For $k=1$, we can check \eqref{eq:BDB_rec} by direct computation. The same is true for $ a,b\in\mathbb Z_{\leq0}$ and arbitrary $ k\in\mathbb Z_{>0}$. For larger integer values of $k$, we will verify  \eqref{eq:BDB_rec} by recursion.  Summing over\begin{align}
b(2k+2-a)
\mathscr S_{2k+1,a,b,s}^A={}&(2k+2)
\mathscr S_{2k,a,b-1,s+1}^A+\mathscr S_{2k-1,a-1,b-1,s}^A,
\end{align}while taking care of contributions from the boundary value (for $s=1$) and a combinatorial identity\begin{align}
{\sum\limits _{s=1}^{\infty} (-1)^s}( \mathscr S_{m,a,b,s}^B- \mathscr S_{m,a,b,s}^A)=\frac{\left\lfloor \frac{m+1}{2}\right\rfloor +1-b}{m+1-a-b}\frac{  \binom{m+1-a}{\left\lfloor \frac{m+1}{2}\right\rfloor +1} \binom{\left\lfloor \frac{m+1}{2}\right\rfloor +1}{b}}{(-1)^ba! (m+1-a)!},
\end{align} we reach a result\begin{align}\begin{split}&
b (2 k + 2 -a)(\smash{\underset{^\circ}{\mathbf S}}_{2k+1}^{B})_{a,b}-4 (k + 1)(\smash{\underset{^\circ}{\mathbf S}}_{2k}^{B})_{a,b-1}+4(\smash{\underset{^\circ}{\mathbf S}}_{2k-1}^{B})_{a-1,b-1}\\={}&(a+1)b(\smash{\underset{^\circ}{\mathbf S}}_{2k+1}^{D})_{a+1,b}-4(\smash{\underset{^\circ}{\mathbf S}}_{2k-1}^{D})_{a,b-1}.\end{split}
\end{align}We then combine the equation above with  \eqref{eq:SBrec} into a recursion\begin{align}\begin{split}
&(a+1)b\left[(\smash{\underset{^\circ}{\mathbf S}}_{2k+1}^{B})_{a,b}+(\smash{\underset{^\circ}{\mathbf S}}_{2k+1}^{D})_{a+1,b}-\frac{2k+3-b}{2}(\smash{\underset{^\circ}{\mathbf S}}_{2k+2}^{B})_{a+1,b}\right]\\={}&4\left[(\smash{\underset{^\circ}{\mathbf S}}_{2k-1}^{B})_{a-1,b-1} +(\smash{\underset{^\circ}{\mathbf S}}_{2k-1}^{D})_{a,b-1}-\frac{2k+2-b}{2}(\smash{\underset{^\circ}{\mathbf S}}_{2k}^{B})_{a,b-1}\right],
\end{split}\end{align} which proves   \eqref{eq:BDB_rec}.

The transpose of   \eqref{eq:BDB_rec}  brings us  $\frac{2k+1-a}{2}(\smash{\underset{^\circ}{\mathbf S}}_{2k}^{B})_{a,b}=(\mathbf T^{C}_{2k-1})_{a-1,b}=\frac{2k+1-a}{2k+1}(\mathbf T^{A}_{2k-1})_{a,b}$ for $ a,b\in\mathbb Z\cap[2,k]$, while    $ (\mathbf T^{C}_{2k-1})_{a-1,1}=\frac{2k+1-a}{2k+1}(\mathbf T^{A}_{2k-1})_{a,1}$   also holds for $ a\in\mathbb Z\cap[2,k]$. Thus,  the top-left, top-right and  bottom-left blocks  of  $\pmb{\boldsymbol \varphi}^{-1}_{2k-1}\pmb{\boldsymbol \alpha}^{-1}_{k}\mathbf S_{2k-1}^{\vphantom{\mathrm T}}(\pmb{\boldsymbol \alpha}^{-1}_{k})^{\mathrm T}( \pmb{\boldsymbol \varphi}^{-1}_{2k-1})^{\mathrm T}$ assume the forms described in  the right-hand side of \eqref{eq:Sigma_block_diag}.

Eliminating $(\smash{\underset{^\circ}{\mathbf S}}_{2k-1}^{B})_{a-1,b}$ from \eqref{eq:SBrec} and \eqref{eq:BDB_rec}, while noting that $ (\smash{\underset{^\circ}{\mathbf S}}_{2k-1}^{D})_{a,b}=(-1)^{\left\lfloor \frac{a-1}{2}\right\rfloor }\frac{1-(-1)^a}{2}\linebreak\frac{2k+1-b}{2k+1}\binom{k}{a}(\smash{\underset{^\circ}{\mathbf S}}_{2k}^{B})_{1,b}$, we arrive at \eqref{eq:invBettiM_rec}.

With the aid of  \eqref{eq:invBettiM_rec}, one can verify that the bottom-right $(k-1)\times(k-1)$ block of  $\pmb{\boldsymbol \varphi}^{-1}_{2k-1}\pmb{\boldsymbol \alpha}^{-1}_{k}\mathbf S_{2k-1}^{\vphantom{\mathrm T}}\linebreak(\pmb{\boldsymbol \alpha}^{-1}_{k})^{\mathrm T}( \pmb{\boldsymbol \varphi}^{-1}_{2k-1})^{\mathrm T}$ is indeed equal to $ -\frac{2}{2k+1} \smash{\underset{^\circ}{\mathbf S}}_{2k-2}^B$.
 \end{enumerate} \end{proof}
\section{Broadhurst--Roberts quadratic relations\label{sec:BR_quad}}In \S\ref{subsec:onshell_quad}, we characterize $ \mathbf S^{-1}_{m}$ in terms of Bernoulli numbers, thereby connecting sum rules of Feynman diagrams to the Betti matrix $ \mathbf B_{m-2}$  in the on-shell limit. In the meantime, we will also  represent the de Rham matrix $ \mathbf D_{m-2}$ for Broadhurst--Roberts quadratic relations in the form declared in Theorem \ref{thm:BRquad}.

In \S\ref{subsec:app_Feynman}, we explore the Broadhurst--Roberts varieties $ V_{m+2}^{\mathrm{BR}}\colonequals \Big\{\mathbf X_{\left\lfloor \frac{m+1}{2} \right\rfloor}^{\vphantom{\mathrm{T}}}\in\mathbb C^{\left\lfloor \frac{m+1}{2} \right\rfloor\times \left\lfloor \frac{m+1}{2} \right\rfloor}\Big|\mathbf X_{\left\lfloor \frac{m+1}{2} \right\rfloor}^{\vphantom{\mathrm{T}}}\mathbf D_{m}^{\vphantom{\mathrm{T}}}\mathbf X_{\left\lfloor \frac{m+1}{2} \right\rfloor}^{{\mathrm{T}}}-\mathbf B_{m}^{\vphantom{\mathrm{T}}}=0\Big \}$, replacing the Broadhurst--Roberts period matrices $\mathbf P_m$ with formal variables. We demonstrate algebraically that  $ V_{2k+1}^{\mathrm{BR}}$ imposes a constraint on certain minor determinants of $ \mathbf X_k$, in the form of a reflection formula \eqref{eq:det_refl}. 

The purpose of \S\ref{subsec:deRham_alt} is twofold. First, we display the diversity of de Rham representations in the Wro\'nskian framework, along with formal proofs. Second, we conclude with an open-ended discussion on other representations of the de Rham matrix $ \mathbf  D_m$ outside the Wro\'nskian framework.  \subsection{Quadratic relations  for on-shell Bessel moments\label{subsec:onshell_quad}}
It is now clear from Proposition \ref{prop:lim_W}  that \begin{align}
|\mathcal L_{2k-1}(1)|\pmb{\boldsymbol \alpha}^{{\mathrm T}}_{k}\mathbf W_{2k-1}^{\mathrm T}(1)\mathbf V^{\vphantom{\mathrm T}}_{2k-1}
(1)
\mathbf W^{\vphantom{\mathrm T}}_{2k-1}(1)\pmb{\boldsymbol \alpha}^{\vphantom{\mathrm T}}_{k}=\pmb{\boldsymbol \alpha}^{{\mathrm T}}_{k}\mathbf S^{-1}_{2k-1}\pmb{\boldsymbol \alpha}^{\vphantom{\mathrm T}}_{k}
\end{align}can be rewritten as \begin{align}\frac{\pmb{\boldsymbol \alpha}^{{\mathrm T}}_{k}\mathbf S^{-1}_{2k-1}\pmb{\boldsymbol \alpha}^{\vphantom{\mathrm T}}_{k}}{|\mathcal L_{2k-1}(1)|}=\begin{pmatrix*}[r]\mathbf P_{2k-1} & \acute{\mathbf{P}}_{2k-1}   \\[5pt]
 & -\mathbf{P}_{2k-3}  \\
\end{pmatrix*}\left[(\pmb{\boldsymbol \beta}_{2k-1}^{-1}\vphantom{\mathrm T})^{\mathrm T}\mathbf V^{\vphantom{\mathrm T}}_{2k-1}
(1)\pmb{\boldsymbol \beta}_{2k-1}^{-1}\right]\begin{pmatrix}\mathbf P_{2k-1}^{\mathrm T} &    \\[5pt]
\acute{\mathbf{P}}_{2k-1}^{\mathrm T} & -\mathbf{P}_{2k-3}^{\mathrm T}  \\
\end{pmatrix},
\label{eq:Mk_block_tridiag}\end{align}{whose bottom-right $ (k-1)\times(k-1)$ block is partly responsible for Theorem \ref{thm:BRquad} [$ m=2k-1 $ in \eqref{eq:Bn_defn} and \eqref{eq:P0_defn}]. Proposition \ref{prop:block_diag_W} then brings us a further refinement:}\begin{align}
\frac{\pmb{\boldsymbol \varphi}^{{\mathrm T}}_{2k-1}\pmb{\boldsymbol \alpha}^{{\mathrm T}}_{k}\mathbf S^{-1}_{2k-1}\pmb{\boldsymbol \alpha}^{\vphantom{\mathrm T}}_{k}\pmb{\boldsymbol \varphi}^{\vphantom{\mathrm T}}_{2k-1}}{|\mathcal L_{2k-1}(1)|}=
\begin{pmatrix}\mathbf P_{2k-1} &
 \\[5pt] & -\mathbf P_{2k-3}  \\
\end{pmatrix}\left[(\pmb{\boldsymbol \vartheta}^{-1}_{2k-1})^{\mathrm T}(\pmb{\boldsymbol \beta}_{2k-1}^{-1}\vphantom{\mathrm T})^{\mathrm T}\mathbf V^{\vphantom{\mathrm T}}_{2k-1}
(1)\pmb{\boldsymbol \beta}_{2k-1}^{-1}\pmb{\boldsymbol \vartheta}^{-1}_{2k-1}\right]\begin{pmatrix}\mathbf P_{2k-1}^{\mathrm T} & \\[5pt]
 & -\mathbf P_{2k-1}^{\mathrm T} \\
\end{pmatrix}.\label{eq:Mk_block_diag}\tag{\ref{eq:Mk_block_tridiag}$^*$}
\end{align}Here, the bottom right  $ (k-1)\times(k-1)$ block in
$ \pmb{\boldsymbol \alpha}^{{\mathrm T}}_{k}\mathbf S^{-1}_{2k-1}\pmb{\boldsymbol \alpha}^{\vphantom{\mathrm T}}_{k}$ [resp.\ $ (\pmb{\boldsymbol \beta}_{2k-1}^{-1}\vphantom{\mathrm T})^{\mathrm T}\mathbf V^{\vphantom{\mathrm T}}_{2k-1}
(1)\pmb{\boldsymbol \beta}_{2k-1}$]
 is unaffected by left multiplication of  $ \pmb{\boldsymbol \varphi}^{{\mathrm T}}_{2k-1}$ [resp.\ $(\pmb{\boldsymbol \vartheta}^{-1}_{2k-1})^{\mathrm T}$] and right multiplication of $ \pmb{\boldsymbol \varphi}^{\vphantom{\mathrm T}}_{2k-1}$ (resp.~$ \pmb{\boldsymbol \vartheta}^{-1}_{2k-1}$).

To obtain  similar block (tri)diagonalization for Theorem \ref{thm:BRquad} $( m=2k)$, we need to expend a little more effort, in the lemma below.
\begin{lemma}[Margin behavior]\label{lm:margin_v}For each given $k\in\mathbb Z_{>0} $, the $ 2k$-th row and the $2k$-th column of \begin{align} \lim_{u\to1^-}|\mathcal L_{2k}(u)|(\pmb{\boldsymbol \beta}_{2k}^{-1}\vphantom{\mathrm T})^{\mathrm T} \mathbf V^{\vphantom{\mathrm T}}_{2k}
(u)\pmb{\boldsymbol \beta}_{2k}^{-1}\end{align} are filled with zeros. \end{lemma}\begin{proof}We first note  that \begin{align}\begin{split}(D^0I_{0}(\sqrt{u}t),\ldots,D^{2k-1}I_{0}(\sqrt{u}t))^{\mathrm T}={}&[
\pmb{\boldsymbol \beta}_{2k}(u)]^{-1}(I_{0}(\sqrt{u}t),-I_{0}(\sqrt{u}t)t^{2}\ldots,(-1)^{k-1}I_{0}(\sqrt{u}t)t^{2k-2},\\{}&\sqrt{u}I_{1}(\sqrt{u}t)t,-\sqrt{u}I_{1}(\sqrt{u}t)t^{3}\ldots,(-1)^{k-1}\sqrt{u}I_{1}(\sqrt{u}t)t^{2k-1})^{\mathrm T}\end{split}\label{eq:Bessel_mat_inv}
\end{align}entails a characterization of the last column in $ \pmb{\boldsymbol \beta}_{2k}^{-1}$:\begin{align}
(\pmb{\boldsymbol \beta}_{2k}^{-1})_{a,2k}=\frac{(-1)^{k-1}}{2^{2k-1}}\delta_{a,2k}.
\end{align}

Then we observe that $( \mathbf V^{\vphantom{\mathrm T}}_{2k}
(u)\pmb{\boldsymbol \beta}_{2k}^{-1})_{a,2k}=\sum_{t=1}^{2k}( \mathbf V^{\vphantom{\mathrm T}}_{2k}
(u))_{a,t}\frac{(-1)^{k-1}}{2^{2k-1}}\delta_{t,2k}=\frac{(-1)^{k-1}}{2^{2k-1}}( \mathbf V^{\vphantom{\mathrm T}}_{2k}
(u))_{a,2k}=\frac{(-1)^{k-1}}{2^{2k-1}}\delta_{a,1}$ and $ ((\pmb{\boldsymbol \beta}_{2k}^{-1}\vphantom{\mathrm T})^{\mathrm T} \mathbf V^{\vphantom{\mathrm T}}_{2k}
(u)\pmb{\boldsymbol \beta}_{2k})_{a,2k}=\frac{(-1)^{k-1}}{2^{2k-1}}\delta_{a,1}$ follow immediately. By skew symmetry, we have $ ((\pmb{\boldsymbol \beta}_{2k}^{-1}\vphantom{\mathrm T})^{\mathrm T} \mathbf V^{\vphantom{\mathrm T}}_{2k}
(u)\pmb{\boldsymbol \beta}_{2k})_{2k,b}=\frac{(-1)^{k}}{2^{2k-1}}\delta_{1,b}$, so the right and bottom margins of the matrix $\lim_{u\to1^-}|\mathcal L_{2k}(u)|(\pmb{\boldsymbol \beta}_{2k}^{-1}\vphantom{\mathrm T})^{\mathrm T} \mathbf V^{\vphantom{\mathrm T}}_{2k}
(u)\pmb{\boldsymbol \beta}_{2k} $ are indeed filled with zeros.\end{proof}According to the result from the last lemma, we have (cf.\ Proposition \ref{prop:lim_W}) \begin{align}\begin{split}&\pmb{\boldsymbol \alpha}^{{\mathrm T}}_{k}\pmb{\boldsymbol \psi}_{k}^{{\mathrm T}}\mathbf S^{-1}_{2k}\pmb{\boldsymbol \psi}_{k}^{\vphantom{\mathrm T}}\pmb{\boldsymbol \alpha}^{\vphantom{\mathrm T}}_{k}\\={}&
\lim_{u\to1^-} |\mathcal L_{2k}(u)|\pmb{\boldsymbol \alpha}^{{\mathrm T}}_{k}\pmb{\boldsymbol \psi}_{k}^{{\mathrm T}}\mathbf W^{{\mathrm T}}_{2k}(u)\mathbf V^{\vphantom{\mathrm T}}_{2k}
(u)\mathbf W^{\vphantom{\mathrm T}}_{2k}(u)\pmb{\boldsymbol \psi}_{k}^{\vphantom{\mathrm T}}\pmb{\boldsymbol \alpha}^{\vphantom{\mathrm T}}_{k}\\={}&\begin{pmatrix*}[r]\mathbf P_{2k} & \acute{\mathbf{P}}_{2k}   \\[5pt]
 & -\mathbf{P}_{2k-2}  \\
\end{pmatrix*}\left[\lim_{u\to1^-}|\mathcal L_{2k}(u)|\pmb{\boldsymbol \varrho}_{k}^{\phantom{\mathrm T}}(\pmb{\boldsymbol \beta}_{2k}^{-1}\vphantom{\mathrm T})^{\mathrm T} \mathbf V^{\vphantom{\mathrm T}}_{2k}
(u)\pmb{\boldsymbol \beta}_{2k}^{-1}\pmb{\boldsymbol \varrho}_{k}^{\mathrm T}\right]\begin{pmatrix}\mathbf P_{2k}^{\mathrm T} &    \\[5pt]
\acute{\mathbf{P}}_{2k}^{\mathrm T} & -\mathbf{P}_{2k-2}^{\mathrm T}  \\
\end{pmatrix},\end{split}\label{eq:Nk_block_tridiag}
\end{align}{which is partly responsible for Theorem \ref{thm:BRquad} [$m=2k$ in \eqref{eq:Bn_defn} and \eqref{eq:P0_defn}]. Its block diagonalized counterpart reads (cf.\ Proposition \ref{prop:block_diag_W})}{\allowdisplaybreaks\begin{align}\begin{split}&\pmb{\boldsymbol\varphi}_{2k}^{\mathrm T}\pmb{\boldsymbol \alpha}^{{\mathrm T}}_{k}\pmb{\boldsymbol \psi}_{k}^{{\mathrm T}}\mathbf S^{-1}_{2k}\pmb{\boldsymbol \psi}_{k}^{\vphantom{\mathrm T}}\pmb{\boldsymbol \alpha}^{\vphantom{\mathrm T}}_{k}\pmb{\boldsymbol\varphi}_{2k}^{\vphantom{\mathrm T}}\\=
{}&\begin{pmatrix}\mathbf P_{2k} &
 \\[5pt] & -\mathbf{P}_{2k-2}  \\
\end{pmatrix}(\pmb{\boldsymbol\vartheta}_{2k}^{-1})^{\mathrm T}\left[\lim_{u\to1^-}|\mathcal L_{2k}(u)|\pmb{\boldsymbol \varrho}_{k}^{\phantom{\mathrm T}}(\pmb{\boldsymbol \beta}_{2k}^{-1}\vphantom{\mathrm T})^{\mathrm T} \mathbf V^{\vphantom{\mathrm T}}_{2k}
(u)\pmb{\boldsymbol \beta}_{2k}^{-1}\pmb{\boldsymbol \varrho}_{k}^{\mathrm T}\right]\pmb{\boldsymbol\vartheta}_{2k}^{-1}\begin{pmatrix}\mathbf P_{2k}^{\mathrm T} & \\[5pt]
 & -\mathbf{P}_{2k-2}^{\mathrm T}  \\
\end{pmatrix}.\end{split}\label{eq:Nk_block_diag}\tag{\ref{eq:Nk_block_tridiag}$^*$}
\end{align}}

To complete the proof of  Theorem \ref{thm:BRquad}, we need to express the entries of $ \mathbf S^{-1}_{m}$ via Bernoulli numbers, in the next proposition. \begin{proposition}[Bernoulli representation of $ \mathbf S^{-1}_{m}$]\label{prop:inv_S}Let $\mathsf B_n $ be the Bernoulli numbers generated by \eqref{eq:Bn_defn}. For all $ m\in\mathbb Z_{>0}$, the inverse of $ \mathbf S_m=\left(\begin{smallmatrix}\mathbf S_m^A & \mathbf S_m^B \\
\mathbf S_m^C & \mathbf S_m^D \\
\end{smallmatrix}\right)$ is given by $ \pmb{\mathscr B}_{m}=\left(\begin{smallmatrix}\pmb{\mathscr B}_{m}^A & \pmb{\mathscr B}_{m}^B \\
\pmb{\mathscr B}_{m}^C & \pmb{\mathscr B}_{m}^D \\
\end{smallmatrix}\right)$ for\footnote{It is understood  that $ (\mathbf S_1)_{1,1}=9,(\pmb{\mathscr B}_{1})_{1,1}=\frac{1}{9}$, and $ \mathbf S_2^{}=\mathbf S_2^{A}=(\pmb{\mathscr B}_{2}^{})^{-1}=(\pmb{\mathscr B}_{2}^A)^{-1}=\left(\begin{smallmatrix}0&-8\\8&\phantom{-}0\end{smallmatrix}\right)$.
} \begin{align}
(\pmb{\mathscr B}_{m}^A )_{a,b}\colonequals{}&\frac{(m+1)! }{m+2}\frac{\delta _{\min\{a,b\},1}\mathsf B_{m+3-a-b} }{ 2^{m-1} (-1)^{\left\lfloor \frac{m+2-a-b}{2} \right\rfloor +\frac{1+a-(-1)^{m}(1-a)}{2} +\left\lfloor \frac{m+1}{2}\right\rfloor }}\label{eq:BAmat_defn}
\end{align}where $ a,b\in\mathbb Z\cap\left[1,\left\lfloor\frac{m}2\right\rfloor+1\right]$, \begin{align}
(\pmb{\mathscr B}_{m}^B )_{a,b'}=(-1)^{m+1}(\pmb{\mathscr B}_{m}^C )_{b',a}\colonequals{}&\frac{(m+2-a)! (m-b')!}{(m + 2 - a - b')!}\frac{  \left(1-\frac{a+b' }{m+2}\delta _{a,1}\right)\mathsf B_{m+2-a-b'}}{ 2^{m-1} (-1)^{\left\lfloor \frac{m+1-a-b'}{2} \right\rfloor +b'+\left\lfloor \frac{m}{2}\right\rfloor }}\label{eq:BBmat_defn}
\end{align} where $ a\in\mathbb Z\cap\left[1,\left\lfloor\frac{m}2\right\rfloor+1\right],b'\in\mathbb Z\cap\left[ 1, \left\lfloor \frac{ m-1}2\right\rfloor\right]$, and \begin{align}
(\pmb{\mathscr B}_{m}^D )_{a',b'}\colonequals\frac{(m-a')! (m-b')!}{(m + 1- a'- b')!}\frac{  \left(m-a'-b'\right)\mathsf B_{m+1-a'-b'}}{ 2^{m-1} (-1)^{\left\lfloor \frac{m-a'-b'}{2} \right\rfloor +a'+\left\lfloor \frac{m-1}{2}\right\rfloor }}\label{eq:BDmat_defn}
\end{align} where $a', b'\in\mathbb Z\cap\left[ 1, \left\lfloor \frac {m-1}2\right\rfloor\right]$.

Consequently, both \eqref{eq:Mk_block_tridiag} and \eqref{eq:Mk_block_diag} [resp.~\eqref{eq:Nk_block_tridiag} and \eqref{eq:Nk_block_diag}] account for the representations of the de Rham matrix $ \mathbf D_{2k-3}$ [resp.~$ \mathbf D_{2k-2}$] per \eqref{eq:deRhamD}  and the Betti matrix  $ \mathbf B_{2k-3}$ [resp.~$ \mathbf B_{2k-2}$] per \eqref{eq:BettiB}.\end{proposition}\begin{proof}By direct computation, one can check that \begin{align}
\pmb{\boldsymbol \varphi}^{{\mathrm T}}_{2k-1}\pmb{\boldsymbol \alpha}^{{\mathrm T}}_{k}\pmb{\mathscr B}_{2k-1}^{\vphantom{\mathrm T}}\pmb{\boldsymbol \alpha}^{\vphantom{\mathrm T}}_{k}\pmb{\boldsymbol \varphi}^{\vphantom{\mathrm T}}_{2k-1}={}&\begin{pmatrix}\frac{2^{4}(-1)^{k-1}}{2k+1}\mathbf B_{2k-1} &  \\
 & 2^{2}(2k+1)(-1)^{k-1} \mathbf B_{2k-3}\\
\end{pmatrix},\label{eq:Sigma_inv_block_diag}
\end{align} and \begin{align}
\pmb{\boldsymbol\varphi}_{2k}^{\mathrm T}\pmb{\boldsymbol \alpha}^{{\mathrm T}}_{k}\pmb{\boldsymbol \psi}_{k}^{{\mathrm T}}\pmb{\mathscr B}_{2k}^{\vphantom{\mathrm T}}\pmb{\boldsymbol \psi}_{k}^{\vphantom{\mathrm T}}\pmb{\boldsymbol \alpha}^{\vphantom{\mathrm T}}_{k}\pmb{\boldsymbol\varphi}_{2k}={}&\begin{pmatrix}\frac{2^{4}(-1)^{k-1}}{2k+2}\mathbf B_{2k} &  \\
 & 2^{2}(2k+2)(-1)^{k-1}\mathbf B_{2k-2}\\
\end{pmatrix},
\end{align} where the entries of $ \mathbf  B_{m-2}\in\mathbb Q^{\left\lfloor \frac{m-1}{2} \right\rfloor\times \left\lfloor \frac{m-1}{2} \right\rfloor}$ are given by the last line of \eqref{eq:BettiB}. (This is essentially the original form of Betti matrix conjectured by Broadhurst--Roberts \cite[(5.4)]{BroadhurstRoberts2018}, except that all its rows and columns are placed in reverse order.) Therefore,  so long as one can verify that $ \mathbf S_m^{-1}=\pmb{\mathscr B}_{m}^{}$ for $m\in\mathbb Z_{>1}$, one can deduce all the Broadhurst--Roberts quadratic relations $ \mathbf P^{}_m\mathbf D^{}_m\mathbf P^{\mathrm T}_m=\mathbf B_m^{}$  from the bottom-right blocks of  \eqref{eq:Mk_block_diag} and \eqref{eq:Nk_block_diag}.

Recall the reshuffling-rescaling matrix $ \pmb{ \boldsymbol\xi}_{m}$ from \eqref{eq:Rmat}.  In the block partition of $ \smash{\underset{^\circ}{\pmb{\mathscr B}}}_m^{}\colonequals\pmb{ \boldsymbol\xi}_{m}^{}\pmb{\mathscr B}_m^{}\pmb{ \boldsymbol\xi}_{m}^{\mathrm T}=\left(\begin{smallmatrix} & \smash{\underset{^\circ}{\pmb{\mathscr B}}}_m^B \\[2pt]
\smash{\underset{^\circ}{\pmb{\mathscr B}}} _m^C &\smash{\underset{^\circ}{\pmb{\mathscr B}}}_m^D \\[1pt]
\end{smallmatrix}\right)$, the top-left $ \left\lfloor \frac{m}{2} \right\rfloor\times\left\lfloor \frac{m}{2} \right\rfloor$ block is filled with zeros, while the remaining blocks are  $ (\smash{\underset{^\circ}{\pmb{\mathscr B}}}_m^B)_{a,b}=(-1)^{m+1}(\smash{\underset{^\circ}{\pmb{\mathscr B}}}_m^C)_{b,a}=({\pmb{\mathscr B}}_m^B)_{a+1,b-1}$ for $ a\in\mathbb Z\cap\left[1,\left\lfloor \frac{m}{2} \right\rfloor\right], b\in\mathbb Z\cap\left[1,\left\lfloor \frac{m+1}{2} \right\rfloor\right]$ and  \begin{align}
(\smash{\underset{^\circ}{\pmb{\mathscr B}}}_m^D)_{a,b}={}&({\pmb{\mathscr B}}_m^D)_{a-1,b-1}\left( 1+\frac{\delta_{a,1}+\delta_{b,1}}{m+2-a-b} \right)\label{eq:B0Dmat_defn}
\end{align}for $ a,b\in\mathbb Z\cap\left[1,\left\lfloor \frac{m+1}{2} \right\rfloor\right]$. Here, {expressions like $(\pmb{\mathscr B}_m^B)_{a+1,0}$,   $(\pmb{\mathscr B}_m^D)_{0,b-1}$ and $(\pmb{\mathscr B}_m^D)_{a-1,0}$  are defined by the right-hand sides of \eqref{eq:BBmat_defn} and \eqref{eq:BDmat_defn}.}

In what follows, we will demonstrate inductively that \begin{align}
\smash{\underset{^\circ}{\mathbf S}}_m\smash{\underset{^\circ}{\pmb{\mathscr B}}}_{m}=\begin{pmatrix}\smash{\underset{^\circ}{\mathbf S}}_m^B \smash{\underset{^\circ}{\pmb{\mathscr B}}}_{m}^C & \smash{\underset{^\circ}{\mathbf S}}_m^A \smash{\underset{^\circ}{\pmb{\mathscr B}}}_{m}^B+\smash{\underset{^\circ}{\mathbf S}}_m^{B}\smash{\underset{^\circ}{\pmb{\mathscr B}}}_{m}^D\\[2pt]
\smash{\underset{^\circ}{\mathbf S}}_m^D \smash{\underset{^\circ}{\pmb{\mathscr B}}}_{m}^C& \smash{\underset{^\circ}{\mathbf S}}_m^C \smash{\underset{^\circ}{\pmb{\mathscr B}}}_{m}^B+\smash{\underset{^\circ}{\mathbf S}}_m^{D}\smash{\underset{^\circ}{\pmb{\mathscr B}}}_{m}^D \\
\end{pmatrix}\end{align}is the $ m\times m$ identity matrix $ \mathbf I_m$, for all positive integers $m$.
Suppose that we have verified the aforementioned claim up to a certain odd number $m=2k-1$. We will show  that  the same is true for $ \smash{\underset{^\circ}{\mathbf S}}_{2k}\smash{\underset{^\circ}{\pmb{\mathscr B}}}_{2k}$ and $ \smash{\underset{^\circ}{\mathbf S}}_{2k+1}\smash{\underset{^\circ}{\pmb{\mathscr B}}}_{2k+1}$.\vspace{2pt}

\noindent\fbox{Inductive evaluation of $ \smash{\underset{^\circ}{\mathbf S}}_{2k}\smash{\underset{^\circ}{\pmb{\mathscr B}}}_{2k}$}\nopagebreak\vspace{1pt}

We note that  $ \smash{\underset{^\circ}{\mathbf S}}_{2k}^D $ is filled with zeros. So must be  the case for $ \smash{\underset{^\circ}{\mathbf S}}_{2k}^D \smash{\underset{^\circ}{\pmb{\mathscr B}}}_{2k}^C$, which forms the bottom-left block of  $ \smash{\underset{^\circ}{\mathbf S}}_{2k}\smash{\underset{^\circ}{\pmb{\mathscr B}}}_{2k}$.

Paraphrasing our induction hypothesis  $ \mathbf S_{2k-1}^{-1}=\pmb{\mathscr B}_{2k-1}^{}$  with \eqref{eq:Sigma_block_diag} and \eqref{eq:Sigma_inv_block_diag}, we have $ (\smash{\underset{^\circ}{\mathbf S}}_{2k}^B)^{-1}=2^{3}(-1)^{k-1}\mathbf B_{2k-1}$. Since $ \smash{\underset{^\circ}{\pmb{\mathscr B}}}_{2k}^C=2^{3}(-1)^{k-1}\mathbf B_{2k-1}$, we can check  that the top-left block satisfies
$ \smash{\underset{^\circ}{\mathbf S}}_{2k}^B\smash{\underset{^\circ}{\pmb{\mathscr B}}}_{2k}^C=\mathbf I_k$ and the bottom-right block has the same behavior:   $ \smash{\underset{^\circ}{\mathbf S}}_{2k}^C \smash{\underset{^\circ}{\pmb{\mathscr B}}}_{2k}^B+\smash{\underset{^\circ}{\mathbf S}}_{2k}^{D}\smash{\underset{^\circ}{\pmb{\mathscr B}}}_{2k}^D =(\smash{\underset{^\circ}{\pmb{\mathscr B}}}_{2k}^C\smash{\underset{^\circ}{\mathbf S}}_{2k}^B)^{\mathrm T}=\mathbf I_k$.

In the next three paragraphs, we show that the top-right block $\smash{\underset{^\circ}{\mathbf S}}_{2k}^A \smash{\underset{^\circ}{\pmb{\mathscr B}}}_{2k}^B+\smash{\underset{^\circ}{\mathbf S}}_{2k}^{B}\smash{\underset{^\circ}{\pmb{\mathscr B}}}_{2k}^D$ has  vanishing entries. This task can be simplified by a similarity transformation \begin{align}\pmb{ \boldsymbol\xi}_{2k}^{\mathrm T} \smash{\underset{^\circ}{\mathbf S}}_{2k}^{}\smash{\underset{^\circ}{\pmb{\mathscr B}}}_{2k}^{}(\pmb{ \boldsymbol\xi}_{2k}^{\mathrm T})^{-1}=
\mathbf S_{2k}\pmb{\mathscr B}_{2k}=\begin{pmatrix*}[r]\mathbf S_{2k}^A\pmb{\mathscr B}_{2k}^A+\mathbf S_{2k}^B \pmb{\mathscr B}_{2k}^C & \mathbf S_{2k}^A \pmb{\mathscr B}_{2k}^B+\mathbf S_{2k}^{B}\pmb{\mathscr B}_{2k}^D\\
& \mathbf I_{k} \\
\end{pmatrix*}.\end{align} Here, the only possible exceptions to  $ (\mathbf S_{2k}\pmb{\mathscr B}_{2k})_{p,q}=\delta_{p,q}$ occur at positions $ p\in\mathbb Z\cap[2,k+1],q\in\mathbb Z\cap(\{1\}\cup[k+2,2k])$,  where we have   $ (\mathbf S_{2k}^A\pmb{\mathscr B}_{2k}^A+\mathbf S_{2k}^B \pmb{\mathscr B}_{2k}^C )_{a+1,1}=\frac{2k+1}{2k+2}(\smash{\underset{^\circ}{\mathbf S}}_{2k}^A \smash{\underset{^\circ}{\pmb{\mathscr B}}}_{2k}^B+\smash{\underset{^\circ}{\mathbf S}}_{2k}^{B}\smash{\underset{^\circ}{\pmb{\mathscr B}}}_{2k}^D)_{a,1}$ for $ a\in\mathbb Z\cap[1,k]$ and   $ (\mathbf S_{2k}^A \pmb{\mathscr B}_{2k}^B+\mathbf S_{2k}^{B}\pmb{\mathscr B}_{2k}^D)_{a+1,b-1}=(\smash{\underset{^\circ}{\mathbf S}}_{2k}^A \smash{\underset{^\circ}{\pmb{\mathscr B}}}_{2k}^B+\smash{\underset{^\circ}{\mathbf S}}_{2k}^{B}\smash{\underset{^\circ}{\pmb{\mathscr B}}}_{2k}^D)_{a,b}$ for $ a\in\mathbb Z\cap[1,k],b\in\mathbb Z\cap[2,k]$.

 First, we show that  $ (\mathbf S_{2k}^A \pmb{\mathscr B}_{2k}^B+\mathbf S_{2k}^{B}\pmb{\mathscr B}_{2k}^D)_{a+1,b-1}=0$ for $ a\in\mathbb Z\cap[1,k]$ and $b\in\mathbb Z\cap[3,k]$. Contracting the recursions $ (\smash{{\pmb{\mathscr B}}}_{2k}^B)_{n+1,b-1}=\frac{2k+1-n}{2} \frac{1+2k\delta_{n,0}}{1+(2k+1)\delta_{n,0}}(\smash{{\pmb{\mathscr B}}}_{2k-1}^B)_{n+1,b-2}$ for $ n\in\mathbb Z\cap[0,k],b\in\mathbb Z\cap[3,k]$ and  $ (\smash{{\pmb{\mathscr B}}}_{2k}^B)_{n+1,b-1}=\frac{[2k+(1-b)(1-\delta_{n,1})](2k+1-n)}{4}(\smash{{\pmb{\mathscr B}}}_{2k-2}^B)_{n,b-2}$ for $ n\in\mathbb Z\cap[1,k],b\in\mathbb Z\cap[3,k]$  with a variation on  \eqref{eq:SArec}, namely\begin{align}
\frac{a(2k+1-n)({\mathbf S}_{2k}^A)_{a+1,n+1}}{1+(2k+1)\delta_{n,0}}={}&\frac{2(2k+1)(\smash{{\mathbf S}}_{2k-1}^A)_{a,n+1}}{(1+2k\delta_{a,1})(1+2k\delta_{n,0})}-\frac{4(\smash{{\mathbf S}}_{2k-2}^A)_{a,n}}{[1+(2k-1)\delta_{a,1}][1+(2k-1)\delta_{n,1}]}
\tag{\ref{eq:SArec}$'$}\end{align} for $ a\in\mathbb Z\cap[1,k],n\in\mathbb Z\cap[0,k]$,  we  compute\begin{align}\begin{split}
(\mathbf S_{2k}^A \pmb{\mathscr B}_{2k}^B)_{a+1,b-1}={}&\frac{2k+1}{a}\frac{(\smash{{\mathbf S}}_{2k-1}^A \smash{{\pmb{\mathscr B}}}_{2k-1}^B)_{a,b-2}}{1+2k\delta_{a,1}}-\frac{2k+1-b}{a}\frac{(\smash{{\mathbf S}}_{2k-2}^A \smash{{\pmb{\mathscr B}}}_{2k-2}^B)_{a,b-2}}{1+(2k-1)\delta_{a,1}}\\{}&+\frac{2k-b}{a}\frac{(\smash{{\mathbf S}}_{2k-2}^A )_{a,1}(\smash{{\pmb{\mathscr B}}}_{2k-2}^B)_{1,b-2}}{1+(2k-1)\delta_{a,1}}\end{split}\label{eq:SA_BB_rec}
\end{align}for $ a\in\mathbb Z\cap[1,k],b\in\mathbb Z\cap[3,k]$, where we have exploited the facts that  $(\smash{{\mathbf S}}_{2k-1}^A)_{a,k+1}=(\smash{{\mathbf S}}_{2k-2}^A)_{a,0}=0$ [by extensions of  \eqref{eq:SAmat_defn}]. Likewise, we can argue that \begin{align}\begin{split}
(\mathbf S_{2k}^B \pmb{\mathscr B}_{2k}^D)_{a+1,b-1}={}&\frac{2k+1}{a}\frac{(\smash{{\mathbf S}}_{2k-1}^B \smash{{\pmb{\mathscr B}}}_{2k-1}^D)_{a,b-2}}{1+2k\delta_{a,1}}-\frac{2k+1-b}{a}\frac{(\smash{{\mathbf S}}_{2k-2}^B \smash{{\pmb{\mathscr B}}}_{2k-2}^D)_{a,b-2}}{1+(2k-1)\delta_{a,1}}\\{}&-\frac{2k+1-b}{a}\frac{(\smash{{\mathbf S}}_{2k-2}^B )_{a,0}(\smash{{\pmb{\mathscr B}}}_{2k-2}^D)_{0,b-2}}{1+(2k-1)\delta_{a,1}},\end{split}\label{eq:SB_BD_rec}
\end{align}for $ a\in\mathbb Z\cap[1,k],b\in\mathbb Z\cap[3,k]$. Adding up  \eqref{eq:SA_BB_rec} and \eqref{eq:SB_BD_rec}, while falling back on our induction hypothesis that $\smash{{\mathbf S}}_{2k-1}^{-1}= \smash{{\pmb{\mathscr B}}}_{2k-1} ^{}$ and $\smash{{\mathbf S}}_{2k-2}^{-1}= \smash{{\pmb{\mathscr B}}}_{2k-2} ^{}$, we arrive at     $ (\mathbf S_{2k}^A \pmb{\mathscr B}_{2k}^B+\mathbf S_{2k}^{B}\pmb{\mathscr B}_{2k}^D)_{a+1,b-1}=(\smash{\underset{^\circ}{\mathbf S}}_{2k}^A \smash{\underset{^\circ}{\pmb{\mathscr B}}}_{2k}^B+\smash{\underset{^\circ}{\mathbf S}}_{2k}^{B}\smash{\underset{^\circ}{\pmb{\mathscr B}}}_{2k}^D)_{a,b}=\frac{2k+1}{a(1+2k\delta_{a,1})}(\mathbf S_{2k}^A \pmb{\mathscr B}_{2k}^B+\mathbf S_{2k}^{B}\pmb{\mathscr B}_{2k}^D)_{a,b-2}-\frac{2k+1-b}{a[1+(2k-1)\delta_{a,1}]}(\mathbf S_{2k-2}^A \pmb{\mathscr B}_{2k-2}^B+\mathbf S_{2k-2}^{B}\pmb{\mathscr B}_{2k-2}^D)_{a,b-2}=0$ for $ a\in\mathbb Z\cap[1,k],b\in\mathbb Z\cap[3,k]$.

Second, we prove that $(\mathbf S_{2k}^A \pmb{\mathscr B}_{2k}^B+\mathbf S_{2k}^{B}\pmb{\mathscr B}_{2k}^D)_{2,1}=0 $.
Adding up  analogs of \eqref{eq:SA_BB_rec} and \eqref{eq:SB_BD_rec} for $ a=1,b=2$, we obtain\begin{align}\begin{split}
(\mathbf S_{2k}^A \pmb{\mathscr B}_{2k}^B+\mathbf S_{2k}^{B}\pmb{\mathscr B}_{2k}^D)_{2,1}={}&\sum_{n=1}^k[(\mathbf S_{2k-1}^A)_{1,n}( \pmb{\mathscr B}_{2k-1}^B)_{n,0}+(\mathbf S_{2k-1}^{B})_{1,n}(\pmb{\mathscr B}_{2k-1}^D)_{n,0}]\\{}&-\frac{2k-1}{2k}\sum_{n=2}^k[(\mathbf S_{2k-2}^A)_{1,n}( \pmb{\mathscr B}_{2k-2}^B)_{n,0}+(\mathbf S_{2k-2}^{B})_{1,n}(\pmb{\mathscr B}_{2k-2}^D)_{n,0}].\end{split}\label{eq:two_sums}
\end{align}Through the explicit formula in  \eqref{eq:Sigma_1st_row_alt}, one can check  that $ (\mathbf S_{2k-1}^{B})_{1,k}=0$ and  $ (\mathbf S_{2k-1}^A)_{1,n}( \pmb{\mathscr B}_{2k-1}^B)_{n,0}+(\mathbf S_{2k-1}^{B})_{1,n-1}(\pmb{\mathscr B}_{2k-1}^D)_{n-1,0}(1-\delta_{n,1})=(\smash{\underset{^\circ}{\mathbf S}}_{2k}^{B})_{1,n}(\smash{\underset{^\circ}{\pmb{\mathscr B}}}_{2k}^C)_{n,1}$ for $ n\in\mathbb Z\cap[1,k]$,  so the first displayed
sum in \eqref{eq:two_sums} evaluates to $(\smash{\underset{^\circ}{\mathbf S}}_{2k}^{B}\smash{\underset{^\circ}{\pmb{\mathscr B}}}_{2k}^C)_{1,1}=1$; referring back to  \eqref{eq:sigma_1st_row_alt}, one sees that $ (\mathbf S_{2k-2}^{B})_{1,n}=0$ and $ (\mathbf S_{2k-2}^A)_{1,n}( \pmb{\mathscr B}_{2k-2}^B)_{n,0}=\frac{2k}{2k-1}(\smash{\underset{^\circ}{\mathbf S}}_{2k-2}^{B})_{1,n-1}(\smash{\underset{^\circ}{\pmb{\mathscr B}}}_{2k-1}^C)_{n-1,1}$ for $ n\in\mathbb Z\cap[2,k]$, so the second displayed
sum in \eqref{eq:two_sums} evaluates to $\frac{2k}{2k-1}\linebreak(\smash{\underset{^\circ}{\mathbf S}}_{2k-2}^{B}\smash{\underset{^\circ}{\pmb{\mathscr B}}}_{2k-2}^C)_{1,1}=\frac{2k}{2k-1}$.
 Thus, the claimed identity is a consequence of our induction hypothesis.

Third, we demonstrate that   $ (\smash{\underset{^\circ}{\mathbf S}}_{2k}^A \smash{\underset{^\circ}{\pmb{\mathscr B}}}_{2k}^B+\smash{\underset{^\circ}{\mathbf S}}_{2k}^{B}\smash{\underset{^\circ}{\pmb{\mathscr B}}}_{2k}^D)_{a,b}=0$ for $ a\in\mathbb Z\cap[1,k],b\in\{1,2\}$. If the integers $a$ and $b$ have the same parity, then $(\smash{\underset{^\circ}{\mathbf S}}_{2k}^A )_{a,n}(\smash{\underset{^\circ}{\pmb{\mathscr B}}}_{2k}^B)_{n,b}=(\smash{\underset{^\circ}{\mathbf S}}_{2k}^B )_{a,n}(\smash{\underset{^\circ}{\pmb{\mathscr B}}}_{2k}^D)_{n,b} =0$ for all $ n\in\mathbb Z\cap[1,k]$. Thus, we have $(\mathbf S_{2k}^A\pmb{\mathscr B}_{2k}^A+\mathbf S_{2k}^B \pmb{\mathscr B}_{2k}^C )_{2,1}=\frac{2k+1}{2k+2}(\smash{\underset{^\circ}{\mathbf S}}_{2k}^A \smash{\underset{^\circ}{\pmb{\mathscr B}}}_{2k}^B+\smash{\underset{^\circ}{\mathbf S}}_{2k}^{B}\smash{\underset{^\circ}{\pmb{\mathscr B}}}_{2k}^D)_{1,1}=0  $. Now we have a matrix partition $ \mathbf S_{2k}\pmb{\mathscr B}_{2k}=\left( \begin{smallmatrix*}[r]\mathbf I_2&&\\\mathbf U&\mathbf I_{k-1}&\mathbf V\\&&\mathbf I_{k-1}\end{smallmatrix*} \right)$ where all the  possibly non-vanishing entries of $\mathbf U$ and $\mathbf V$ are confined to their first columns. Such a matrix partition is invertible, and $ \mathbf S_{2k}^{-1}=\pmb{\mathscr B}_{2k}^{} \left( \begin{smallmatrix*}[r]\mathbf I_2&&\\-\mathbf U&\mathbf I_{k-1}&-\mathbf V\\&&\mathbf I_{k-1}\end{smallmatrix*} \right)$ is skew-symmetric. With a matrix partition $\pmb{\mathscr B}_{2k}=\left( \begin{smallmatrix*}[r]\pmb{\mathscr B}_{2k}^{A'}&-(\pmb{\mathscr B}_{2k}^{A''})^{\mathrm T}&-(\pmb{\mathscr B}_{2k}^{C'})^{\mathrm T}\\\pmb{\mathscr B}_{2k}^{A''}&&-(\pmb{\mathscr B}_{k}^\flat)^{\mathrm T}\\\pmb{\mathscr B}_{2k}^{C'}&\pmb{\mathscr B}_{k}^\flat&\pmb{\mathscr B}_{2k}^D\end{smallmatrix*} \right) $ where $\pmb{\mathscr B}_{k}^\flat\in\mathbb Q^{(k-1)\times(k-1)} $ is taken from the last $(k-1)$ columns of $ \pmb{\mathscr B}_{2k}^{C}$, we will spell out the skew symmetry of  $ \mathbf S_{2k}^{-1}=\left( \begin{smallmatrix*}[r]\pmb{\mathscr B}_{2k}^{A'}+(\pmb{\mathscr B}_{2k}^{A''})^{\mathrm T}\mathbf U&-(\pmb{\mathscr B}_{2k}^{A''})^{\mathrm T}&-(\pmb{\mathscr B}_{2k}^{C'})^{\mathrm T}+(\pmb{\mathscr B}_{2k}^{A''})^{\mathrm T}\mathbf V\\\pmb{\mathscr B}_{2k}^{A''}&&-(\pmb{\mathscr B}_{k}^\flat)^{\mathrm T}\\\pmb{\mathscr B}_{2k}^{C'}-\pmb{\mathscr B}_{k}^\flat\mathbf U&\pmb{\mathscr B}_{k}^\flat&\pmb{\mathscr B}_{2k}^D-\pmb{\mathscr B}_{k}^\flat\mathbf V\end{smallmatrix*} \right)$. Here, the matrix   $\pmb{\mathscr B}_{k}^\flat\mathbf V$ is equal to $ \pmb{\mathscr B}_{k}^\flat\mathbf v$ followed by the  $(k-2)$  identically vanishing columns to the right, where $\mathbf v$ is the first column of $ \mathbf V$. In order to guarantee the skew symmetry of  $\pmb{\mathscr B}_{k}^\flat\mathbf V$, one has to have $ \pmb{\mathscr B}_{k}^\flat\mathbf v=\mathbf 0\in\mathbb R^{k-1}$. We claim that $ \det  \pmb{\mathscr B} ^\flat_{k}\neq0$, which can be seen from the relation between \begin{align}\label{eq:det_Bflat}
|\det  \pmb{\mathscr B} ^\flat_{k}  |={}&\frac{1}{2^{(k-1)(2k-1)}}\left(\prod_{n=k+1}^{2k-1}n!\right)^2\left\vert \det\left( \frac{\mathsf B_{a+b}}{(a+b)!} \right)_{1\leq a,b\leq k-1} \right\vert
\intertext{and}\label{eq:det_Bflat_c}
|\det\smash{\underset{^\circ}{\pmb{\mathscr B}}}_{2k-2}^{C}|={}&\frac{1}{2^{(k-1)(2k-3)}}\left(\prod_{n=k}^{2k-2}n!\right)^2\left\vert \det\left( \frac{\mathsf B_{a+b}}{(a+b)!} \right)_{1\leq a,b\leq k-1} \right|,
\end{align}the latter of which is equal to $ 1/|\det\smash{\underset{^\circ}{\mathbf S}}_{2k-2}|>0$, by  induction. Therefore, we conclude that  $ \mathbf v=(\pmb{\mathscr B}_{k}^\flat)^{-1}\mathbf 0=\mathbf 0\in\mathbb R^{k-1}$.  Now that $\mathbf V$ is filled with zeros, the skew symmetry concerning the bottom-left and the top-right blocks of  $ \mathbf S_{2k}^{-1}$ further brings us a  matrix $ \pmb{\mathscr B}_{k}^\flat\mathbf U$ filled with two columns and $(k-1)$ rows of zeros, hence an identically vanishing $\mathbf U$.

Thus far, we have verified $ \mathbf S_m^{-1}= \pmb{\mathscr B}_{m}^{}$ up to an even number $m=2k$.

\noindent\fbox{Inductive evaluation of $\smash{\underset{^\circ}{\mathbf S}}_{2k+1}\smash{\underset{^\circ}{\pmb{\mathscr B}}}_{2k+1}$}\nopagebreak\vspace{1pt}

In view of  \eqref{eq:Sigma_block_diag} and \eqref{eq:Sigma_inv_block_diag}, we only need to check that
$ \smash{\underset{^\circ}{\mathbf S}}_{2k+2}^B\smash{\underset{^\circ}{\pmb{\mathscr B}}}_{2k+2}^C=\mathbf I_{k+1}^{}$.

Noting that $ (\smash{\underset{^\circ}{\mathbf S}}_{2k}^B)_{a,0}=0$ by definition in \eqref{eq:S0Bmat_defn}, and that $ (\smash{\underset{^\circ}{\pmb{\mathscr B}}}_{2k}^C)_{n,b}=\frac{4}{(2k+2-n)(2k+2-b)}(\smash{\underset{^\circ}{\pmb{\mathscr B}}}_{2k+2}^C)_{n+1,b+1}$, we can deduce from \eqref{eq:invBettiM_rec} the following result for   $ a,b\in\mathbb Z\cap[1,k]$:
\begin{align}\begin{split} \delta_{a,b}={}&\frac{2k+2-b}{2k+2-a}\sum_{n=1}^k(\smash{\underset{^\circ}{\mathbf S}}_{2k}^B)_{a,n}(\smash{\underset{^\circ}{\pmb{\mathscr B}}}_{2k}^C)_{n,b}\\={}&(\smash{\underset{^\circ}{\mathbf S}}_{2k+2}^B \smash{\underset{^\circ}{\pmb{\mathscr B}}}_{2k+2}^{C})_{a+1,b+1}-\frac{1+(-1)^{a}}{ (-1)^{\left\lfloor \frac{a\vphantom1}{2}\right\rfloor }} \frac{\binom{k+1}{a+1} (\smash{\underset{^\circ}{\mathbf S}}_{2k+2}^B\smash{\underset{^\circ}{\pmb{\mathscr B}}}_{2k+2}^{C})_{1,b+1}}{2k+2-a},\end{split}\end{align}
upon invoking our induction hypothesis that
$ \smash{\underset{^\circ}{\mathbf S}}_{2k}^B\smash{\underset{^\circ}{\pmb{\mathscr B}}}_{2k}^C$ is  the $k\times k$ identity matrix.  The last displayed identity alleviates our workload to a demonstration that $ (\smash{\underset{^\circ}{\mathbf S}}_{2k+2}^B \smash{\underset{^\circ}{\pmb{\mathscr B}}}_{2k+2}^{C})_{1,b+1}=(\smash{\underset{^\circ}{\mathbf S}}_{2k+2}^B \smash{\underset{^\circ}{\pmb{\mathscr B}}}_{2k+2}^{C})_{b+1,1}=\delta_{b,0}$ for all $ b\in\mathbb Z\cap[0,k]$. This can be achieved in three steps, in the paragraphs to follow.

First, we show that $(\smash{\underset{^\circ}{\mathbf S}}_{2k+2}^B\smash{\underset{^\circ}{\pmb{\mathscr B}}}_{2k+2}^{C})_{1,b+1}=0$ for $ b\in\mathbb Z\cap[1,k]$. If $b$ is an odd number in  $\mathbb Z\cap[1,k]$, then $ (\smash{\underset{^\circ}{\mathbf S}}_{2k+2}^B)_{1,n}(\smash{\underset{^\circ}{\pmb{\mathscr B}}}_{2k+2}^{C})_{n,b+1}=0$ for all $ n\in\mathbb Z\cap[1,k+1]$, so we have $ (\smash{\underset{^\circ}{\mathbf S}}_{2k+2}^B\smash{\underset{^\circ}{\pmb{\mathscr B}}}_{2k+2}^{C})_{1,b+1}=0$. If $b$ is an even number in  $\mathbb Z\cap[2,k]$, then we have a recursion\begin{align}\begin{split}
\frac{k (k+1)(\smash{\underset{^\circ}{\mathbf S}}_{2k+2}^B)_{1,n}(\smash{\underset{^\circ}{\pmb{\mathscr B}}}_{2k+2}^{C})_{n,b+1}}{(2 k+3) (2 k+2-b)}={}&\frac{(b-2)(\smash{\underset{^\circ}{\mathbf S}}_{2k}^{B})_{1,n}(\smash{\underset{^\circ}{\pmb{\mathscr B}}}_{2k}^C)_{n,b-1}}{2 k+2-b}\\{}&+(1-\delta_{n,1})(\smash{\underset{^\circ}{\mathbf S}}_{2k}^A )_{1,n-1}(\smash{\underset{^\circ}{\pmb{\mathscr B}}}_{2k}^B)_{n-1,b}+(\smash{\underset{^\circ}{\mathbf S}}_{2k}^{B})_{1,n}(\smash{\underset{^\circ}{\pmb{\mathscr B}}}_{2k}^D)_{n,b},\end{split}\label{eq:SumBetti_even_rec}
\end{align}whose sum over $ n\in\mathbb Z\cap[1,k+1] $ results in  $(\smash{\underset{^\circ}{\mathbf S}}_{2k+2}^B\smash{\underset{^\circ}{\pmb{\mathscr B}}}_{2k+2}^{C})_{1,b+1}=0$. This is because $ (\smash{\underset{^\circ}{\mathbf S}}_{2k}^{B})_{1,k+1}=0$ by definition, and  our induction hypothesis
on   $ \smash{\underset{^\circ}{\mathbf S}}_{2k}\smash{\underset{^\circ}{\pmb{\mathscr B}}}_{2k}$ leaves us an identity matrix $ \smash{\underset{^\circ}{\mathbf S}}_{2k}^{B}\smash{\underset{^\circ}{\pmb{\mathscr B}}}_{2k}^C$ along with a $ k\times k$ block  $ \smash{\underset{^\circ}{\mathbf S}}_{2k}^A \smash{\underset{^\circ}{\pmb{\mathscr B}}}_{2k}^B+\smash{\underset{^\circ}{\mathbf S}}_{2k}^B \smash{\underset{^\circ}{\pmb{\mathscr B}}}_{2k}^D$ that is  filled with zeros.

Second, we verify the identity $(\smash{\underset{^\circ}{\mathbf S}}_{2k+2}^B \smash{\underset{^\circ}{\pmb{\mathscr B}}}_{2k+2}^{C})_{1,1}=1$ by directly computing\begin{align}
(\smash{\underset{^\circ}{\mathbf S}}_{2k+2}^B \smash{\underset{^\circ}{\pmb{\mathscr B}}}_{2k+2}^{C})_{1,1}={}&\sum_{n=0}^k(\smash{\underset{^\circ}{\mathbf S}}_{2k+2}^B)_{1,n+1}(\smash{\underset{^\circ}{\pmb{\mathscr B}}}_{2k+2}^C)_{n+1,1}=\frac{2(-1)^{k}  (2 k+3)! }{[(k+1)!]^2}\sum _{n=0}^{k+1} \frac{ \binom{k+1}{k+1-n}\mathsf B_{2 k+3-n}}{2 k+3-n},
\end{align}and enlisting the help from  a special case of the Saalsch\"utz--Nielsen--Gelfand reciprocity relation \cite{Saalschuetz1892,Saalschuetz1893,Nielsen1923,Gelfand1968} revisited by Agoh--Dilcher \cite{AgohDilcher2008}:\begin{align}
(-1)^{\beta+1}\sum_{j=0}^{\alpha} \frac{ \binom{\alpha}{j} \mathsf B_{\beta+1+j}}{ {\beta+1+j}}+(-1)^{\alpha+1}\sum_{j=0}^{\beta} \frac{ \binom{\beta}{j} \mathsf B_{\alpha+1+j}}{ {\alpha+1+j}}=\frac{\alpha!\beta!}{(\alpha+\beta+1)!},
\end{align}where $ \alpha=\beta=k+1,j=k+1-n$.

Third, we prove that $(\smash{\underset{^\circ}{\mathbf S}}_{2k+2}^B\smash{\underset{^\circ}{\pmb{\mathscr B}}}_{2k+2}^{C})_{b+1,1}=0$ for $ b\in\mathbb Z\cap[1,k]$.
In other words, we will show that in the block partition $ \smash{\underset{^\circ}{\mathbf S}}_{2k+2}^B\smash{\underset{^\circ}{\pmb{\mathscr B}}}_{2k+2}^{C}=\left( \begin{smallmatrix}1&\\ \mathbf x&\mathbf I_{k} \end{smallmatrix}\right)$, the column vector $ \mathbf x\in\mathbb R^k$ is in fact a zero vector. Taking determinants on both sides of the block partition, we know that the symmetric matrix    $  \smash{\underset{^\circ}{\mathbf S}}_{2k+2}^B$ is  invertible, and its inverse $ ( \smash{\underset{^\circ}{\mathbf S}}_{2k+2}^B)^{-1}=\smash{\underset{^\circ}{\pmb{\mathscr B}}}_{2k+2}^{C}\left( \begin{smallmatrix*}[r]1&\\ -\mathbf x&\mathbf I_{k} \end{smallmatrix*}\right)$ is also symmetric. Taking another block partition   $ \smash{\underset{^\circ}{\pmb{\mathscr B}}}_{2k+2}^{C}= \Big(\begin{smallmatrix*}[r]{b}_{k+1}&\mathbf b_{k+1}^{\mathrm T}\\ \mathbf b_{k+1}&\pmb{\mathscr B} ^\flat_{k+1}\end{smallmatrix*}\Big)$ where ${b}_{k+1}= ( \smash{\underset{^\circ}{\pmb{\mathscr B}}}_{2k+2}^{C})_{1,1}$, we may exploit the symmetry of $ ( \smash{\underset{^\circ}{\mathbf S}}_{2k+2}^B)^{-1}=\Big(\begin{smallmatrix*}[r] {b}_{k+1}-\mathbf b_{k+1}^{\mathrm T}\mathbf x&\mathbf b_{k+1}^{\mathrm T}\\\mathbf b_{k+1}^{\vphantom1}-\pmb{\mathscr B} ^\flat _{k+1}\mathbf x&\pmb{\mathscr B} ^\flat_{k+1}\end{smallmatrix*}\Big)$ to arrive at the following relation: $  \pmb{\mathscr B} ^\flat_{k+1} \mathbf x=\mathbf 0\in\mathbb R^k$. In view of \eqref{eq:det_Bflat} and \eqref{eq:det_Bflat_c}, we can deduce the invertibility of $ \pmb{\mathscr B} ^\flat_{k+1}$ from the non-singular matrix  $  \smash{\underset{^\circ}{\pmb{\mathscr B}}}_{2k}^{C}=( \smash{\underset{^\circ}{\mathbf S}}_{2k}^B)^{-1}$, so $ \mathbf x=( \pmb{\mathscr B} ^\flat _{k+1})^{-1}\mathbf 0=\mathbf 0\in\mathbb R^k$.

This completes the induction.
\end{proof}

As a by-product of the proof above, we can quickly recover part of  \cite[Proposition 4.9(1)]{FresanSabbahYu2020b} without using any combinatorial identities concerning Bernoulli numbers.
\begin{corollary}[Determinants of Betti matrices]\label{cor:detBetti}The following formula holds true:\begin{align}
\det\mathbf B_{2k-1}=\frac{(-1)^{k-1}(2k-1)!}{2^{5k-1}}\Lambda_{2k-1}.
\end{align}When $ k$ is odd, we have $\det\mathbf B_{2k}=0$.\end{corollary}\begin{proof}Taking determinants on both sides of \eqref{eq:omegavomega}, we arrive at $ \Lambda^2_{2k-1}=1/\det\mathbf S_{2k-1}$. Taking determinants on both sides of \eqref{eq:Sigma_inv_block_diag}, we obtain $ \det\mathbf B_{2k-1}\det\mathbf B_{2k-3}\det\mathbf S_{2k-1}=(-1)^{k-1} (2 k+1) 4^{1-3 k}$. Hence, we can build the formula for $ \det \mathbf B_{2k-1}$ inductively
on $ \det \mathbf B_1=\frac{1}{2^4\cdot 3}$ and $ \det\mathbf B_{2k-1}\det\mathbf B_{2k-3}=(-1)^{k-1} (2 k+1) 4^{1-3 k}\Lambda^2_{2k-1}$.

When  $ k$ is odd, the odd-numbered rows of $\mathbf B_{2k} $ have non-zero elements in the even-numbered columns. Since there are $\left\lfloor\frac{k}{2}\right\rfloor $ such columns and  $\left\lfloor\frac{k+1}{2}\right\rfloor =\left\lfloor\frac{k}{2}\right\rfloor+1$ such rows, all the rows in question must be linearly dependent. Hence  $\det\mathbf B_{2k}=0$.  \end{proof}

For $ k\in\mathbb Z_{>0}$, define minors $ \mathbf B_{2k-1}^{\mathrm o}\in\mathbb Q^{\left\lfloor \frac{k+1}{2} \right\rfloor\times\left\lfloor \frac{k+1}{2} \right\rfloor }$ and  $\mathbf B_{2k-1}^{{\mathrm e}}\in\mathbb Q^{\left\lfloor \frac{k}{2} \right\rfloor\times\left\lfloor \frac{k}{2} \right\rfloor }$ as follows:\begin{align}
(\mathbf B_{2k-1}^{\mathrm o})_{a,b}=(\mathbf B_{2k-1})_{2a-1,2b-1},\quad (\mathbf B_{2k-1}^{\mathrm e})_{a,b}=(\mathbf B_{2k-1})_{2a,2b}.\label{eq:BkoBke}
\end{align} By convention, we set $ \mathbf B_1^{\mathrm e}=\varnothing$ and $\det\mathbf B_1^{\mathrm e}=1$. Clearly, we have $ \det\mathbf B_{2k-1}^{\vphantom{\mathrm{o}}}=\det\mathbf B_{2k-1}^{\mathrm{o}}\det\mathbf B_{2k-1}^\mathrm{e}$ for all positive integers $k$. In the next corollary, we  evaluate $\det\mathbf B_{2k-1}^{\mathrm e}$, a result that  will be used later in the proof of Proposition \ref{prop:det_refl}.

\begin{corollary}[Determinants of Betti minors]\label{cor:detBettiMeven}We have\begin{align}\det\mathbf B_{2k-1}^{\mathrm e}=
\frac{[(2k+1)!!]^{\frac{3-(-1)^k}{2}}}{(-2)^{\left\lfloor \frac{k}{2}\right\rfloor}(k-1)!!(k!!)^{1-(-1)^{k}}}\left[ \frac{(2k+1)!!}{2^{k+1}} \right]^{2\left\lfloor \frac{k}{2} \right\rfloor}\prod^k_{j=1}\frac{(2j)^{k-j}}{(2j+1)^{j+1}}\label{eq:BettiMe}
\end{align}for each positive integer $k$.\end{corollary}\begin{proof} By rescaling and rearranging columns and rows  of determinants, we have the following results for each $ m\in\mathbb Z_{>0}$:{\allowdisplaybreaks\begin{align}\det \mathbf B_{4m-3}^{\mathrm o}={}&\frac{1}{2^{4 m^{2}}}\left[ \prod_{n=m}^{2m-1}(2n)! \right]^{2}\det\left( \frac{\mathsf B_{2(a+b-1)}}{[{2(a+b-1)]!}} \right)_{1\leq a,b\leq m},\\\det \mathbf  B_{4m-3}^{\mathrm e}={}&\frac{1}{2^{4 m(m-1)}}\left[ \prod_{n=m}^{2(m-1)}(2n+1)! \right]^{2}\det\left( \frac{\mathsf B_{2(a+b)}}{[{2(a+b)]!}} \right)_{1\leq a,b\leq m-1},\label{eq:BettiMeven_kodd}\\
\det \mathbf  B_{4m-1}^{\mathrm o}={}&\frac{(-1)^{m}}{2^{2 m (2 m+1)}}\left[ \prod_{n=m+1}^{2m}(2n)! \right]^{2}\det\left( \frac{\mathsf B_{2(a+b)}}{[{2(a+b)]!}} \right)_{1\leq a,b\leq m},\label{eq:BettiModd_keven}\\\det \mathbf B_{4m-1}^{\mathrm e}={}&\frac{(-1)^{m}}{2^{2 m (2 m+1)}}\left[ \prod_{n=m+1}^{2m}(2n-1)! \right]^{2}\det\left( \frac{\mathsf B_{2(a+b-1)}}{[{2(a+b-1)]!}} \right)_{1\leq a,b\leq m}.
\end{align}}Thus, we have a recursion\begin{align}\begin{split}
-\frac{(4 m-1) (4 m-2)}{2^5}\frac{\Lambda_{4m-1}}{\Lambda_{4m-3}}={}&\frac{\det \mathbf  B_{4m-1}}{\det \mathbf B_{4m-3}}=\frac{\det \mathbf B_{4m-1}^{\mathrm o}\det \mathbf B_{4m-1}^{\mathrm e}}{\det \mathbf B_{4m-3}^{\mathrm o}\det \mathbf B_{4m-3}^{\mathrm e}}\\={}&\frac{1}{2^{8m}}\left[\frac{(4m)!(4m-1)!}{(2m)!}\right]^2\frac{\det\left( \frac{\mathsf B_{2(a+b)}}{[{2(a+b)]!}} \right)_{1\leq a,b\leq m}}{\det\left( \frac{\mathsf B_{2(a+b)}}{[{2(a+b)]!}} \right)_{1\leq a,b\leq m-1}},\end{split}
\end{align}namely\begin{align}-\frac{1}{4m+1}\left[ \frac{(2m)!}{(4m)!} \right]^{2}={}&\frac{\det\left( \frac{\mathsf B_{2(a+b)}}{[{2(a+b)]!}} \right)_{1\leq a,b\leq m}}{\det\left( \frac{\mathsf B_{2(a+b)}}{[{2(a+b)]!}} \right)_{1\leq a,b\leq m-1}},\end{align}from which we can solve\begin{align}
\det\left( \frac{\mathsf B_{2(a+b)}}{[{2(a+b)]!}} \right)_{1\leq a,b\leq m}=(-1)^m \prod _{n=1}^m \frac{[(2 n)!]^2}{ (4 n)!(4 n+1)!}.
\end{align}
Substituting back into \eqref{eq:BettiMeven_kodd} and \eqref{eq:BettiModd_keven}, we may verify\begin{align}
\det \mathbf B_{4m-3}^{\mathrm e}={}&\frac{(-1)^{m+1} 2^{(5-4 m) m}}{(2 m-1)\text{!!} (4 m-2)\text{!!}}\prod _{n=1}^{2 m-1} (2 n-1)!,\label{eq:BettiMe_o}\\\det \mathbf B_{4m-1}^{\mathrm o}={}&\frac{2^{-m (4 m+3)}}{(2 m)\text{!!} (4 m+1)\text{!!}}\prod _{n=1}^{2 m} (2 n)!,
\end{align}where the last equation also implies\begin{align}
\det \mathbf B_{4m-1}^{\mathrm e}=\frac{\det \mathbf B_{4m-1}}{\det \mathbf B_{4m-1}^{\mathrm o}}=\frac{(-1)^m}{2^{m (4 m+1)}} {}&\frac{ (2 m)\text{!!}}{(4 m)\text{!!}}\prod_{n=1}^{2 m} (2 n-1)!.\label{eq:BettiMe_e}
\end{align}Through elementary manipulations of the factors, we may convert \eqref{eq:BettiMe_o} and \eqref{eq:BettiMe_e} into our statement in \eqref{eq:BettiMe}.  \end{proof}
\subsection{Arithmetic applications to Feynman diagrams\label{subsec:app_Feynman}}
The Broadhurst--Roberts quadratic relations inject new insights into the arithmetic properties of on-shell Bessel moments, an important class of Feynman diagrams that respect classical equations of motion.

First, we examine period matrices  $ \mathbf P_{2k-1}$ for the $(2k+1)$-Bessel problems, where $  k\in\mathbb Z_{>1}$.
\begin{tiny}
\begin{table}\caption{Some Betti matrices and  de Rham matrices in Broadhurst--Roberts quadratic relations  \label{tab:Betti_deRham}}\begin{scriptsize}\begin{tabular}{c|l|l}\hline\hline$m$&$ \mathbf B_m$&$\mathbf  D_{m}$\\\hline3&$ \begin{pmatrix*}[r] \frac{1}{2^{4}\cdot5}&\\&-\frac{3}{2^{6}} \\
\end{pmatrix*}$&$ \begin{pmatrix*}[r] \frac{13}{2^{3}}& \frac{3^{2}\cdot5^{2}}{2^{6}}\\ \frac{3^{2}\cdot5^{2}}{2^{6}}& \\
\end{pmatrix*}$\\4&$ \begin{pmatrix*}[r]& \frac{1}{2^{5}}\\-\frac{1}{2^{5}}& \\
\end{pmatrix*}$&$ \begin{pmatrix*}[r]&-2\cdot3^{2}\\ 2\cdot3^{2}& \\
\end{pmatrix*}$\\5&$ \begin{pmatrix*}[r] \frac{3\cdot5}{2^{5}\cdot7}&& \frac{3}{2^{5}}\\&-\frac{5}{2^{6}}&\\ \frac{3}{2^{5}}&& \frac{3}{2^{4}} \\
\end{pmatrix*}$&$ \begin{pmatrix*}[r] \frac{3\cdot17}{2^{4}}& \frac{3\cdot863}{2^{5}}& \frac{3^{2}\cdot5^{2}\cdot7^{2}}{2^{8}}\\ \frac{3\cdot863}{2^{5}}& \frac{3^{2}\cdot5^{2}\cdot7^{2}}{2^{8}}&\\ \frac{3^{2}\cdot5^{2}\cdot7^{2}}{2^{8}}&& \\
\end{pmatrix*}$\\6&$ \begin{pmatrix*}[r]& \frac{3\cdot5}{2^{6}}&\\-\frac{3\cdot5}{2^{6}}&&-\frac{3\cdot5}{2^{6}}\\& \frac{3\cdot5}{2^{6}}& \\
\end{pmatrix*}$&$ \begin{pmatrix*}[r]&-2^{5}\cdot3^{2}&-2^{6}\cdot3^{2}\\ 2^{5}\cdot3^{2}&&\\ 2^{6}\cdot3^{2}&& \\
\end{pmatrix*}$\\7&$ \begin{pmatrix*}[r] \frac{3\cdot7}{2^{4}}&& \frac{3\cdot5}{2^{4}}&\\&-\frac{3\cdot5\cdot7}{2^{7}}&&-\frac{3\cdot5\cdot7}{2^{7}}\\ \frac{3\cdot5}{2^{4}}&& \frac{3^{2}\cdot5}{2^{6}}&\\&-\frac{3\cdot5\cdot7}{2^{7}}&&-\frac{3\cdot5^{2}}{2^{6}} \\
\end{pmatrix*}$&$ \begin{pmatrix*}[r] \frac{3\cdot7}{2^{2}}& \frac{3\cdot20173}{2^{6}}& \frac{3^{3}\cdot5\cdot97^{2}}{2^{8}}& \frac{3^{6}\cdot5^{2}\cdot7^{2}}{2^{10}}\\ \frac{3\cdot20173}{2^{6}}& \frac{3^{2}\cdot11833}{2^{7}}& \frac{3^{6}\cdot5^{2}\cdot7^{2}}{2^{10}}&\\ \frac{3^{3}\cdot5\cdot97^{2}}{2^{8}}& \frac{3^{6}\cdot5^{2}\cdot7^{2}}{2^{10}}&&\\ \frac{3^{6}\cdot5^{2}\cdot7^{2}}{2^{10}}&&& \\
\end{pmatrix*}$\\8&$ \begin{pmatrix*}[r]& \frac{3^{3}\cdot7}{2^{5}}&& \frac{3^{3}\cdot5}{2^{5}}\\-\frac{3^{3}\cdot7}{2^{5}}&&-\frac{3\cdot5\cdot7}{2^{5}}&\\& \frac{3\cdot5\cdot7}{2^{5}}&& \frac{3^{2}\cdot5\cdot7}{2^{7}}\\-\frac{3^{3}\cdot5}{2^{5}}&&-\frac{3^{2}\cdot5\cdot7}{2^{7}}& \\
\end{pmatrix*}$&$\begin{pmatrix*}[r]&-\frac{3^{5}\cdot47}{2^{2}}&-3\cdot59\cdot191&-2^{5}\cdot3^{3}\cdot5^{2}\\ \frac{3^{5}\cdot47}{2^{2}}&&-2^{5}\cdot3^{2}\cdot5^{2}&\\ 3\cdot59\cdot191& 2^{5}\cdot3^{2}\cdot5^{2}&&\\ 2^{5}\cdot3^{3}\cdot5^{2}&&& \\
\end{pmatrix*}$\\[20pt]\hline\hline\end{tabular}\end{scriptsize}\end{table}
\end{tiny}

From the work of Bailey--Borwein--Broadhurst--Glasser \cite{BBBG2008},    Bloch--Kerr--Vanhove \cite{BlochKerrVanhove2015}, Samart \cite{Samart2016}, and the present author \cite{Zhou2017WEF,Zhou2017BMdet,Zhou2019BMdimQ}, we know that \begin{align}
\mathbf P_3\colonequals {}&\begin{pmatrix*}[r]\frac{1}{\pi^{2}}\IKM(1,4;1) & -\frac{1}{\pi^{2}}\IKM(1,4;3) \\
\frac{1}{\pi}\IKM(2,3;1) & -\frac{1}{\pi}\IKM(2,3;3) \\
\end{pmatrix*}
=\begin{pmatrix*}[r] C & -\left( \frac{2}{15} \right)^2\left( 13C-\frac{1}{10C} \right)  \\
\frac{\sqrt{15}}{2}C & -\frac{\sqrt{15}}{2}\left( \frac{2}{15} \right)^2\big( 13C+\frac{1}{10C} \big)
\end{pmatrix*},
\label{eq:M2_eval_C}
\end{align}where\begin{align}
C\colonequals \frac{\Gamma \left(\frac{1}{15}\right) \Gamma \left(\frac{2}{15}\right) \Gamma \left(\frac{4}{15}\right) \Gamma \left(\frac{8}{15}\right)}{240 \sqrt{5}\pi^{2}}\label{eq:defn_Bologna}
\end{align} is the ``Bologna constant'' attributed to Broadhurst \cite{Broadhurst2007,BBBG2008} and Laporta \cite{Laporta2008}. Previously, each individual entry of   $ \mathbf P_3=((-1)^{b-1}\mathcal F_{3,a}^b(1))_{1\leq a,b\leq 2}$   was computed with sophisticated arguments. Now, a quadratic relation  $ \mathbf P^{\vphantom{\mathrm  T}}_3\mathbf D_3^{\vphantom{\mathrm  T}}\mathbf P_3^{\mathrm  T}=\mathbf B_3$ (see Table \ref{tab:Betti_deRham} for non-vanishing entries of $ \mathbf  B_3$ and $ \mathbf  D_3$) produces three independent quadratic equations involving Bessel moments, which immediately allow us to express all the matrix elements [cf.~\eqref{eq:M2_eval_C}] in terms of the top-left entry $ \mathcal F_{3,1}^1(1)=C$.

As for the 7-Bessel problem, numerical evidence has
motivated the following conjecture by Broadhurst and Mellit  (see \cite[(6.8)]{BroadhurstMellit2016} or \cite[(129)]{Broadhurst2016}):\begin{align}
\IKM(2,5;1)\overset{?}{=}\frac{5\pi^2}{24}\zeta_{7,1}(2).\label{eq:IKM251}
\end{align}
 Here, we have an $ L$-function associated with  the 7th symmetric power moments of Kloosterman sums \begin{align}
\zeta_{7,1}(s)\colonequals \prod_pL_p(\mathbb G_{m,\mathbb F_p},\Sym^7\Kl_2,s)\equiv\prod_p \frac{1}{Z_7(p,p^{-s})},
\end{align}whose local factors at primes $p$ are written as $ L_p(\mathbb G_{m,\mathbb F_p},\Sym^7\Kl_2,s)$ in the Fu--Wan notation \cite[\S0]{FuWan2008MathAnn} and $ \frac{1}{Z_7(p,p^{-s})}$ in the Broadhurst notation \cite[\S2]{Broadhurst2016}. Meanwhile, according to the Fres\'an--Sabbah--Yu statement \cite[\S7.c]{FresanSabbahYu2020b} of Deligne's conjecture for the motive $ \mathrm M_7$ whose period realization is isomorphic to $(\mathrm H^1_{\mathrm{dR,mid}}(\mathbb G_m,\Sym^7\Kl_2),\mathrm H_1^{\mathrm{mid}}(\mathbb G_m,\Sym^7\Kl_2),\mathsf P_7^{\mathrm{mid}})$, one has\begin{align}\det \begin{pmatrix}\IKM(1,6;1) & \IKM(1,6;3) \\
\IKM(3,4;1) & \IKM(3,4;3) \\
\end{pmatrix}\overset{?}{=}\frac{\pi^4}{14\sqrt{105}}\zeta_{7,1}(2).\label{eq:IKM251'}
\end{align}In transcribing the conjectural identity above, we have already exploited a relation between  $ \zeta_{7,1}(2)$ and $ \zeta_{7,1}(3)$, which results from a functional equation conjectured by Broadhurst--Mellit (see \cite[(6.7)]{BroadhurstMellit2016} or \cite[(128)]{Broadhurst2016}) and verified by Fres\'an--Sabbah--Yu \cite[Theorem 1.2]{FresanSabbahYu2018}:\begin{align}
\Lambda_7(s)\colonequals \left( \frac{105}{\pi^3} \right)^{s/2}\Gamma\left( \frac{s-1}{2} \right)\Gamma\left( \frac{s}{2} \right)\Gamma\left( \frac{s+1}{2} \right)  \zeta_{7,1}(s)=\Lambda_7(5-s).
\end{align}
The conjectures \eqref{eq:IKM251} and \eqref{eq:IKM251'} are actually equivalent to each other, since we have \begin{align}
\mathcal{F}_{5,2}^1(1)_{}=-\frac{1}{2^2}\sqrt{\frac{5^3 \cdot7^3}{3}}\det\begin{pmatrix}\mathcal{F}_{5,1}^1(1) &\mathcal{F}_{5,1}^2(1) \\
\mathcal{F}_{5,3}^1(1) & \mathcal{F}_{5,3}^2(1) \\
\end{pmatrix}\label{eq:Deligne7}
\end{align}by reading off the element in the second row and the first column from both sides of $\mathbf P^{\vphantom{\mathrm  T}}_5 =\mathbf B_5^{\vphantom{\mathrm T}}\frac{\cof \mathbf P^{\vphantom{\mathrm  T}}_5}{\det\mathbf P^{\vphantom{\mathrm  T}}_5}\mathbf D_5^{-1}$, which is a reformulation of the quadratic relation $ \mathbf P^{\vphantom{\mathrm  T}}_5\mathbf D_5^{\vphantom{\mathrm  T}}\mathbf P_5^{\mathrm  T}=\mathbf B_5^{\vphantom{\mathrm T}}$ using the determinant $ \det \mathbf P^{\vphantom{\mathrm  T}}_5=-\frac{2^4}{\sqrt{3^3 5^5 7^7}}$ and  the cofactor matrix $\cof \mathbf P^{\vphantom{\mathrm  T}}_5=(\det\mathbf P^{\vphantom{\mathrm  T}}_5)(\mathbf P_5^{\mathrm  T})^{-1}$. The sum rule in \eqref{eq:Deligne7} is in fact a continuation of a pattern attested in $ \mathbf P_3$ [see \eqref{eq:M2_eval_C}]:\begin{align}
\mathcal F_{3,2}^1(1)=\frac{\sqrt{3\cdot5}}{2}\mathcal F_{3,1}^1(1),
\end{align}which in turn, is compatible with proven instances \cite{RogersWanZucker2015,Samart2016} of Deligne's conjecture for the motive $ \mathrm M_5$:\begin{align}
\mathcal F_{3,1}^1(1)=\frac{1}{5}\zeta_{5,1}(1),\quad \mathcal F_{3,2}^1(1)=\frac{3}{4\pi} \zeta_{5,1}(2).
\end{align}Here,\begin{align}
\zeta_{5,1}(s)\colonequals {}&\prod_pL_p(\mathbb G_{m,\mathbb F_p},\Sym^5\Kl_2,s)\equiv\prod_p \frac{1}{Z_5(p,p^{-s})}=\int_{0}^\infty\frac{(2\pi y)^{s}}{\Gamma(s)}\frac{f_{3,15}(iy)\D y}{y}
\end{align}is the $L$-function associated with  the cusp form $ f_{3,15}(z)\colonequals [\eta(3z)\eta(5z)]^3+[\eta(z)\eta(15z)]^3$ of weight 3 and level 15 \cite{PetersTopVlugt1992}, satisfying a functional equation \cite[(95)]{Broadhurst2016}:
\begin{align}
\Lambda_5(s)\colonequals \left( \frac{15}{\pi^2} \right)^{s/2}\Gamma\left( \frac{s}{2} \right)\Gamma\left( \frac{s+1}{2} \right)  \zeta_{5,1}(s)=\Lambda_5(3-s).
\end{align}

Trading each element in a minor  determinant  of $\mathbf P_7$  for its counterpart in  $ \mathbf B_7^{\vphantom{\mathrm T}}\frac{\cof \mathbf P^{\vphantom{\mathrm  T}}_7}{\det\mathbf P^{\vphantom{\mathrm  T}}_7}\mathbf D_7^{-1}$,  we can also construct another sum rule \begin{align}
\det\begin{pmatrix}\mathcal F_{7,2}^1 (1)& \mathcal F_{7,2}^2(1) \\
\mathcal F_{7,4}^1(1) & \mathcal F_{7,4}^2(1) \
\end{pmatrix}={}&\frac{\sqrt{3\cdot5\cdot7}}{2^2}\det\begin{pmatrix}\mathcal F_{7,1}^1(1) & \mathcal F_{7,1}^2 (1)\\
\mathcal F_{7,3}^1(1) & \mathcal F_{7,3}^2(1) \
\end{pmatrix},\label{eq:M4_refl}
\end{align}
 which is compatible with Deligne's conjecture \cite[\S7.c]{FresanSabbahYu2020b}  for the motive
 $ \mathrm M_{9}$ and the corresponding functional equation \cite[Theorem 1.2]{FresanSabbahYu2018} for $ \zeta_{9,1}(s)$.  For an arbitrarily large $k$, there is a similar identity connecting two minor determinants of  $\mathbf P_{2k-1} $, as established in the proposition below.

\begin{proposition}[Reflection formula for minor determinants of  $\mathbf P_{2k-1} $]\label{prop:det_refl}For $k\in\mathbb Z_{>1} $, we have an identity that is equivalent to \eqref{eq:det_refl}:{\allowdisplaybreaks\begin{align}\begin{split}&
\det((-1)^{b-1}\mathcal F_{2k-1,2a}^b(1))_{1\leq a,b\leq\left\lfloor \frac{k}{2} \right\rfloor}\\={}&(-1)^{\left\lfloor \frac{k+1}{4} \right\rfloor+\left\lfloor \frac{1}{2}\left\lfloor \frac{k}{2} \right\rfloor \right\rfloor}\frac{\sqrt{[(2k+1)!!]^{2-(-1)^k}}}{2^{\left\lfloor \frac{k}{2}\right\rfloor}(k-1)!!(k!!)^{1-(-1)^{k}}}\det((-1)^{b-1}\mathcal F_{2k-1,2a-1}^b(1))_{1\leq a,b\leq\left\lfloor \frac{k+1}{2} \right\rfloor}.\label{eq:det_refl_Deligne}\end{split}
\end{align}}Here, the leading factor on the right-hand side can be rewritten as\begin{align}
(-1)^{\left\lfloor \frac{k+1}{4} \right\rfloor+\left\lfloor \frac{1}{2}\left\lfloor \frac{k}{2} \right\rfloor \right\rfloor}=\frac{\det((-1)^{b-1}\delta_{a,b})_{1\leq a,b\leq\left\lfloor \frac{k}{2} \right\rfloor}}{\det((-1)^{b-1}\delta_{a,b})_{1\leq a,b\leq\left\lfloor \frac{k+1}{2} \right\rfloor}}=\begin{cases}-1, & k\equiv 3\pmod4, \\
1, & \text{otherwise}. \\
\end{cases}
\end{align} \end{proposition}\begin{proof}Suppose that we partition an invertible matrix into four blocks $ \mathbf M=\left(\begin{smallmatrix}\mathbf A&\mathbf B\\\mathbf C&\mathbf D\end{smallmatrix}\right)$, where the top-left block $ \mathbf A$ is also invertible. The Banachiewicz formula \cite{Banachiewicz1937} provides us with a partition\begin{align}\begin{split}
\mathbf M^{-1}={}&\begin{pmatrix*}[r]\mathbf I & -\mathbf A^{-1}\mathbf B \\
 & \mathbf I \\
\end{pmatrix*}  \begin{pmatrix*}[r]\mathbf A^{-1} &  \\
 & (\mathbf D-\mathbf C\mathbf A^{-1}\mathbf B)^{-1} \\
\end{pmatrix*} \begin{pmatrix*}[l]\mathbf I &  \\
-\mathbf C\mathbf A^{-1} & \mathbf I \\
\end{pmatrix*}  \\={}&\begin{pmatrix*}[r]\mathbf{A}^{-1}+\mathbf{A}^{-1}\mathbf B(\mathbf D-\mathbf C\mathbf A^{-1}\mathbf B)^{-1}\mathbf C\mathbf A^{-1} & -\mathbf{A}^{-1}\mathbf B(\mathbf D-\mathbf C\mathbf A^{-1}\mathbf B)^{-1} \\
-(\mathbf D-\mathbf C\mathbf A^{-1}\mathbf B)^{-1}\mathbf C\mathbf A^{-1} & (\mathbf D-\mathbf C\mathbf A^{-1}\mathbf B)^{-1} \\
\end{pmatrix*} , \end{split}\label{eq:Banachiewicz}
\end{align}where $ \mathbf I$ stands for an identity matrix of appropriate size. Taking determinants on the first equality in \eqref{eq:Banachiewicz}, we have $ \det(\mathbf M)=\det(\mathbf A)\det(\mathbf D-\mathbf C\mathbf A^{-1}\mathbf B)$.

Define\begin{align}
\overset{_\frown}{\mathbf P}_{2k-1}=\begin{pmatrix*}[l]((-1)^{b-1}\mathcal F_{2k-1,2a-1}^b(1))_{1\leq a\leq\left\lfloor \frac{k+1}{2} \right\rfloor,1\leq b\leq k} \\
((-1)^{b-1}\mathcal F_{2k-1,2a}^b(1))_{1\leq a\leq\left\lfloor \frac{k}{2} \right\rfloor,1\leq b\leq k} \\
\end{pmatrix*}
\end{align}by reshuffling rows of $\mathbf P_{2k-1} $. The Broadhurst--Roberts quadratic relation can be reformulated into\begin{align}
\overset{_\frown}{\mathbf P}_{2k-1}^{\vphantom{\mathrm  T}}\mathbf D_{2k-1}^{\vphantom{\mathrm  T}}=\begin{pmatrix}\mathbf B_{2k-1}^{\mathrm o} &  \\
 & \mathbf B_{2k-1}^{{\mathrm e}} \\
\end{pmatrix}\left(\smash{\overset{_\frown}{\mathbf P}}_{2k-1}^{{\mathrm  T}}\right)^{-1},\label{eq:BR_Mk_odd_even}
\end{align}where $ \mathbf B_{2k-1}^{\mathrm o}$ and $ \mathbf B_{2k-1}^{\mathrm e}$ are defined in \eqref{eq:BkoBke}.

Suppose that $ \det((-1)^{b-1}\mathcal{F}_{2k-1,2a-1}^b(1))_{1\leq a,b\leq\left\lfloor \frac{k+1}{2} \right\rfloor}\neq0$, and take a matrix partition\begin{align}
\overset{_\frown}{\mathbf P}_{2k-1}=\begin{pmatrix*}[l]\overset{_\frown}{\mathbf A}\colonequals ((-1)^{b-1}\mathcal{F}_{2k-1,2a-1}^b(1))_{1\leq a,b\leq\left\lfloor \frac{k+1}{2} \right\rfloor} & \overset{_\frown}{\mathbf B} \\
\overset{_\frown}{\mathbf C} & \overset{_\frown}{\mathbf D} \\
\end{pmatrix*}.
\end{align}We may read off the bottom-right $ \left\lfloor \frac{k}{2} \right\rfloor\times\left\lfloor \frac{k}{2} \right\rfloor$ block from both sides of \eqref{eq:BR_Mk_odd_even}, while referring back to the Banachiewicz formula \eqref{eq:Banachiewicz}, as follows:\begin{align}
((-1)^{b-1}\mathcal{F}_{2k-1,2a}^b(1))_{1\leq a,b\leq\left\lfloor \frac{k}{2} \right\rfloor}\mathbf D_{2k-1}^{\isosc}=\mathbf B_{2k-1}^{{\mathrm e}}\left[\smash{\overset{_\frown}{\mathbf D}} ^{{\mathrm  T}}-\smash{\overset{_\frown}{\mathbf B}}^{{\mathrm  T}}\smash{\left(\smash{\overset{_\frown}{\mathbf A}} ^{-1}\right)}^{\mathrm T}\smash{\overset{_\frown}{\mathbf C}} ^{\mathrm T}\right]^{-1}.\label{eq:bottom_right_Schur}
\end{align}Here, no matter $k$ is  even or odd, the matrix  $ \mathbf D_{2k-1}^{\isosc}$ is extracted from the top-right  $ \left\lfloor \frac{k}{2} \right\rfloor\times\left\lfloor \frac{k}{2} \right\rfloor$ block of $ \mathbf D_{2k-1}$. We may quote in advance from Corollary \ref{cor:anti_diag_Dk} that   $ \mathbf D_{2k-1}^{\isosc}$ is upper-left triangular, with all its anti-diagonal elements being $ \left[\frac{(2k+1)!!}{2^{k+1}}\right]^2$. Since \begin{align} \det\left(\left[\smash{\overset{_\frown}{\mathbf D}} ^{{\mathrm  T}}-\smash{\overset{_\frown}{\mathbf B}}^{{\mathrm  T}}\smash{\left(\smash{\overset{_\frown}{\mathbf A}} ^{-1}\right)}^{\mathrm T}\smash{\overset{_\frown}{\mathbf C}} ^{\mathrm T}\right]^{-1}\right)=\frac{\det\left(\left(\smash{\overset{_\frown}{\mathbf P}}_{2k-1}^{\mathrm T}\right)^{-1}\right)}{\det\left(\left(\smash{\overset{_\frown}{\mathbf A}}^{\mathrm T}\right)^{-1}\right)} ={}&\frac{\det((-1)^{b-1}\mathcal F_{2k-1,2a-1}^b(1))_{1\leq a,b\leq\left\lfloor \frac{k+1}{2} \right\rfloor}}{(-1)^{\left\lfloor\frac{k+1}{4}\right\rfloor}\det \overset{_\frown}{\mathbf P}_{2k-1}},\end{align} we can deduce \begin{align}
\det((-1)^{b-1}\mathcal F_{2k-1,2a}^b(1))_{1\leq a,b\leq\left\lfloor \frac{k}{2} \right\rfloor}=\frac{(-1)^{\left\lfloor\frac{k+1}{4}\right\rfloor}\det \mathbf B_{2k-1}^{{\mathrm e}}}{\det \overset{_\frown}{\mathbf P}_{2k-1}\det\mathbf D_{2k-1}^{\isosc}}\det((-1)^{b-1}\mathcal F_{2k-1,2a-1}^b(1))_{1\leq a,b\leq\left\lfloor \frac{k+1}{2} \right\rfloor}
\end{align}from \eqref{eq:bottom_right_Schur}. Simplifying the right-hand side of the equation above with    $ \det \overset{_\frown}{\mathbf P}_{2k-1}=(-1)^{\left\lfloor \frac{k}{2} \right\rfloor}\prod_{j=1}^k\frac{(2j)^{k-j}}{\sqrt{(2j+1)^{2j+1}}}$, $ \det\mathbf D_{2k-1}^{\isosc}=(-1)^{\left\lfloor \frac{1}{2}\left\lfloor \frac{k}{2} \right\rfloor \right\rfloor}\left[\frac{(2k+1)!!}{2^{k+1}}\right]^{2\left\lfloor \frac{k}{2} \right\rfloor}$ and Corollary \ref{cor:detBettiMeven}, we reach  \eqref{eq:det_refl_Deligne}.

Starting from an alternative assumption that $ \det((-1)^{b-1}\mathcal F_{2k-1,2a}^b(1))_{1\leq a,b\leq\left\lfloor \frac{k}{2} \right\rfloor}\neq0$, while switching the  r\^oles of $ ((-1)^{b-1}\mathcal F_{2k-1,2a-1}^b(1))_{1\leq a\leq\left\lfloor \frac{k+1}{2} \right\rfloor,1\leq b\leq k}$ and $ ((-1)^{b-1}\mathcal F_{2k-1,2a}^b(1))_{1\leq a\leq\left\lfloor \frac{k}{2} \right\rfloor,1\leq b\leq k}$, one can repeat the procedures above to arrive at the same identity in \eqref{eq:det_refl_Deligne}.

If both $ \det((-1)^{b-1}\mathcal F_{2k-1,2a-1}^b(1))_{1\leq a,b\leq\left\lfloor \frac{k+1}{2} \right\rfloor}$ and  $ \det((-1)^{b-1}\mathcal F_{2k-1,2a}^b(1))_{1\leq a,b\leq\left\lfloor \frac{k}{2} \right\rfloor}$ vanish,  then  \eqref{eq:det_refl_Deligne} is trivially true.
\end{proof}\begin{remark}The reflection formula \eqref{eq:det_refl_Deligne} applies retroactively  to the case where $ k=1$, so long as we adopt the convention that $\det\varnothing=1 $. In this case, the classical evaluation $ \IKM(1,2;1)=\frac{\pi}{3\sqrt{3}}$ \cite[(23)]{BBBG2008} is recovered.\eor\end{remark}

The cases of $ \mathbf P_{2k}, k\in\mathbb Z_{>1}$  involve additional subtleties.

The quadratic relation   $ \mathbf P_4^{}\mathbf D_4^{}\mathbf P_4^{\mathrm T}=\mathbf B_4$ (see Table \ref{tab:Betti_deRham} for non-vanishing entries of $ \mathbf B_4$ and $ \mathbf D_4$) produces only one independent quadratic equation for Bessel moments, which is equivalent to $ \det\mathbf P_4^{}=-\frac{1}{2^{6}3^2} $. Thus, we do not have news for the 6-Bessel problem. Especially, on the algebraic side, there seem to be  no  quadratic mechanisms to relate the ratio $ [\mathcal{F}_{4,2}^1(1)]^{2}/[\mathcal{F}_{4,1}^1(1)]^2$ to a rational number; on the analytic side, the  proven instances \cite[\S4]{Zhou2017WEF} of Deligne's conjecture for the motive $ \mathrm M_6$ bring us two evaluations  $ \pi^{5/2}\mathcal F^1_{4,1}(1)=\frac{\pi^2}{2}\zeta_{6,1}(2)$ and $ \pi^{3/2}\mathcal F^1_{4,2}(1)=\frac{\pi^2}{2}\zeta_{6,1}(1)=\frac{3}{2}\zeta_{6,1}(3)$, which are not connected to each other by the functional equation
 \cite[(106)]{Broadhurst2016}\begin{align}
\Lambda_6(s)\colonequals \left( \frac{6}{\pi^2} \right)^{s/2}\Gamma\left( \frac{s}{2} \right)\Gamma\left( \frac{s+1}{2} \right)  \zeta_{6,1}(s)=\Lambda_6(4-s)
\end{align} for $ \zeta_{6,1}(s)$, the $L$-function associated with  the cusp form $ f_{4,6}(z)\colonequals [\eta(z)\eta(2z)\eta(3z)\eta(6z)]^2$ of weight 4 and level 6 \cite{Hulek2001}.

Unlike the situations discussed so far, the $ L$-function for the motive  $\mathrm M_8 $ needs an additional correction factor, in the form of \cite[\S7.6]{Broadhurst2016}\begin{align}
L(f_{6,6},s)\colonequals \int_{0}^\infty\frac{(2\pi y)^{s}}{\Gamma(s)}\frac{f_{6,6}(iy)\D y}{y}=\prod_p \frac{Z_4(p,p^{4-s})}{Z_8(p,p^{-s})},
\end{align} where the associated cusp form \cite{Yun2015}\begin{align}f_{6,6}(z)\colonequals
\frac{ [\eta (2 z) \eta (3 z)]^{9}}{[\eta (z)\eta (6 z)]^{3}}+\frac{ [\eta ( z) \eta (6 z)]^{9}}{[\eta (2z)\eta (3 z)]^{3}}
\end{align}has weight 6 and level 6. The Fres\'an--Sabbah--Yu statement \cite[\S7.c]{FresanSabbahYu2020b} of Deligne's conjecture holds for the motive  $\mathrm M_8$, according to the following identities conjectured by Broadhurst \cite[\S7.6]{Broadhurst2016} and proved by the present author (see \cite[\S5]{Zhou2017WEF} as well as  \cite[\S3.1]{Zhou2019BMdimQ}): $ \pi^{7/2}\mathcal F^1_{6,1}(1)=\pi^{7/2}\mathcal F^1_{6,3}(1)=\frac{\pi^4}{4}L(f_{6,6},2)$, $ \pi^{5/2}\mathcal F^1_{6,2}(1)=\frac{\pi^{4}}{8}L(f_{6,6},1)=\frac{9 \pi ^2}{14}L(f_{6,6},3)$.
The quadratic relation   $ \mathbf P_6^{}\mathbf D_6^{}\mathbf P_6^{\mathrm T}=\mathbf B_6^{\vphantom{\mathrm T}}$ produces three independent quadratic equations concerning on-shell Bessel moments, which in turn, can be reduced (say, by computation of Gr\"obner basis) to    the following sum rules:\begin{align}
\mathcal F_{6,1}^{1}(1)={}&\mathcal F_{6,3}^1(1),\\\mathcal F_{6,1}^{2}(1)+2\mathcal F_{6,1}^{3}(1)={}&\mathcal F_{6,3}^{2}(1)+2\mathcal F_{6,3}^3(1),\\\det \begin{pmatrix}\mathcal F_{6,2}^{1}(1) & \mathcal F_{6,2}^{2}(1)+2\mathcal F_{6,2}^3(1) \\
\mathcal F_{6,3} ^1(1)& \mathcal F_{6,3}^{2}(1)+2\mathcal F_{6,3}^3(1) \\
\end{pmatrix}={}&\frac{5}{2^{11}\cdot3}.
\end{align} All these three identities have been reported earlier \cite[(98)]{Zhou2018ExpoDESY}.  Moreover, we point out that the quadratic relation    $ \mathbf P_6^{}\mathbf D_6^{}\mathbf P_6^{\mathrm T}=\mathbf B_6^{\vphantom{\mathrm T}}$ does \textit{not} exhaust algebraic relations  [over ${\mathbb Q}(\pi)$] for the entries of $ \mathbf P_6$. Noticeably, such a relation does not entail the following identities \cite[Theorem 1.3]{Zhou2019BMdimQ}:{\allowdisplaybreaks\begin{align}
\mathcal F_{6,1}^{1}(1)+2^{3}\cdot3^2\mathcal F_{6,1}^2(1)={}&\frac{7\sqrt{\pi}}{2^{4}\cdot3},\\\mathcal F_{6,2}^{1}(1)+2^{3}\cdot3^2\mathcal F_{6,2}^{2}(1)={}&0,\label{eq:aberrant}\\\mathcal F_{6,3}^{1}(1)+2^{3}\cdot3^2\mathcal F_{6,3}^{2}(1)={}&-\frac{5\sqrt{\pi}}{2^{2}\cdot3},\\2^{8}\cdot3^3\det\begin{pmatrix}\mathcal F_{6,1}^{1}(1) & \mathcal F_{6,1}^{3}(1) \\
\mathcal F_{6,2}^{1}(1) & \mathcal F_{6,2}^{3}(1) \\
\end{pmatrix}+\frac{3^2\cdot5}{2^{4}}={}&7\sqrt{\pi}\mathcal F_{6,2}^{1}(1),
\end{align}}which are all provable by analytic properties of the cusp form $ f_{6,6}(z)$ \cite[\S3]{Zhou2019BMdimQ}.

We say that a sum rule $ P(\mathcal F_{8,1}^{1}(1),\mathcal F_{8,1}^{2}(1),\ldots,\mathcal F_{8,4}^{4}(1))=0$ for the 10-Bessel problem arises from a ``trivial factor'' if $ P(x_{1,1},x_{1,2},\ldots,x_{4,4})\in\mathbb Z[x_{1,1},x_{1,2},\ldots,x_{4,4}]$ divides the following polynomial in 16 formal variables belonging to the matrix  $ \mathbf X_4=(x_{a,b})_{1\leq a,b\leq 4}$: \begin{align}\prod_{1\leq a<b\leq 4}
\left({\mathbf X}^{\vphantom{\mathrm  T}}_4\mathbf D_8^{\vphantom{\mathrm
T}}{\mathbf X}_4^{\mathrm  T}-\mathbf  B_8^{\vphantom{\mathrm T}} \right)_{a,b}\prod_{m=1}^4\prod_{n=1}^4\left(   \mathbf B_8^{-1} {\mathbf X}^{\vphantom{\mathrm  T}}_4 \mathbf D_8^{\vphantom{\mathrm  T}}-\frac{\cof {\mathbf X}^{\vphantom{\mathrm  T}}_4 }{\det\mathbf P^{\vphantom{\mathrm  T}}_8} \right)_{m,n}.
\end{align}After constructing a  Gr\"obner basis that generates the ideal   \begin{align}  \big \langle {\mathbf X}^{\vphantom{\mathrm
T}}_4\mathbf D_8^{\vphantom{\mathrm
T}}{\mathbf X}_4^{\mathrm  T}-\mathbf  B_8^{\vphantom{\mathrm  T}}\big\rangle\colonequals{}&\left\{\left.\sum_{1\leq a<b\leq 4}\left({\mathbf X}^{\vphantom{\mathrm  T}}_4\mathbf D_8^{\vphantom{\mathrm
T}}{\mathbf X}_4^{\mathrm  T}-\mathbf  B_8^{\vphantom{\mathrm T}} \right)_{a,b}p_{a,b}\right|p_{a,b}\in\mathbb C[x_{1,1},x_{1,2},\ldots,x_{4,4}]\right\},\end{align} one can find 6 sum rules (loci of polynomials in the aforementioned ideal) that do not arise from ``trivial factors'':\begin{align}\begin{split}&
5\cdot7\det\begin{pmatrix}\mathcal F_{8,1}^{b}(1) &3 \mathcal F_{8,2}^{b}(1)-2^{2}  \mathcal F_{8,4}^{b}(1)\\
\mathcal F_{8,1}^{b'}(1) & 3 \mathcal F_{8,2}^{b'}(1)-2^{2}  \mathcal F_{8,4}^{b'}(1) \\
\end{pmatrix}-2^2\cdot3^2\det\begin{pmatrix}\mathcal F_{8,3}^{b}(1) &5 \mathcal F_{8,2}^{b}(1)-7  \mathcal F_{8,4}^{b}(1)\\
\mathcal F_{8,3}^{b'}(1) & 5 \mathcal F_{8,2}^{b'}(1)-7  \mathcal F_{8,4}^{b'}(1) \\
\end{pmatrix}\\={}&\begin{cases}0, & b=1,b'\in\{2,3\},\\[2pt]
-\frac{7}{2^{10}\cdot5}, & b=1,b'=4, \\[2pt]
-\frac{3\cdot7}{2^{10}\cdot5}, & b=2,b'=3, \\[2pt]
\frac{7\cdot59\cdot191}{2^{15}\cdot3\cdot5^3}, & b=2,b'=4, \\[2pt]
-\frac{3\cdot7\cdot47}{2^{17}\cdot5^3}, & b=3,b'=4.
\end{cases}\end{split}
\end{align} None of them resembles the situation \eqref{eq:M4_refl} in     $ \mathbf P_7$. If one rewrites  $ \det\Big(\begin{smallmatrix}\mathcal F_{8,2}^1(1) & \mathcal F_{8,2}^2 (1)\\
\mathcal F_{8,4}^1 (1)& \mathcal F_{8,4}^2(1)
\end{smallmatrix}\Big)$ \Big[or $ \det\Big(\begin{smallmatrix}\mathcal F_{8,1}^1(1) & \mathcal F_{8,1}^2(1) \\
\mathcal F_{8,3}^1(1) & \mathcal F_{8,3}^2(1)
\end{smallmatrix}\Big)$\Big] by replacing every entry of $   {\mathbf P}^{\vphantom{\mathrm  T}}_8$ with its counterpart in $  \mathbf  B_8^{\vphantom{\mathrm  T}}\frac{\cof {\mathbf P}^{\vphantom{\mathrm  T}}_8 }{\det\mathbf P^{\vphantom{\mathrm  T}}_8}\mathbf D_8^{-1}$, then one just gets  back the determinant itself, instead of a different minor.
Analytically, this is understandable, due to  the lack of a functional connection between critical $L$-values that are (conjecturally) $ \mathbb Q^\times\sqrt{\pi}^{\mathbb Z}$ multiples of the aforementioned minor determinants. Algebraically, the gadget in our proof of Proposition  \ref{prop:det_refl} no longer works here, because the row-reshuffled matrix $ \overset{_\frown}{\mathbf P}_8$ satisfies\begin{align}
\overset{_\frown}{\mathbf P}_8^{\vphantom{\mathrm  T}}\mathbf D_8^{\vphantom{\mathrm  T}}=\begin{pmatrix}&& \frac{3^{3}\cdot7}{2^{5}}& \frac{3^{3}\cdot5}{2^{5}}\\&& \frac{3\cdot5\cdot7}{2^{5}}& \frac{3^{2}\cdot5\cdot7}{2^{7}}\\-\frac{3^{3}\cdot7}{2^{5}}&-\frac{3\cdot5\cdot7}{2^{5}}&&\\-\frac{3^{3}\cdot5}{2^{5}}&-\frac{3^{2}\cdot5\cdot7}{2^{7}}&& \\
\end{pmatrix}\left(\smash{\overset{_\frown}{\mathbf P}}_8^{{\mathrm  T}}\right)^{-1},
\end{align}where the row- and column-reshuffled Betti matrix  no longer has block diagonal form---the bottom-left (as opposed to bottom-right) block of the Banachiewicz formula \eqref{eq:Banachiewicz} results in  a tautological statement about a minor determinant  $ \det\Big(\begin{smallmatrix}\mathcal F_{8,2}^1(1) & \mathcal F_{8,2}^2 (1)\\
\mathcal F_{8,4}^1 (1)& \mathcal F_{8,4}^2(1)
\end{smallmatrix}\Big)$ \Big[or $ \det\Big(\begin{smallmatrix}\mathcal F_{8,1}^1(1) & \mathcal F_{8,1}^2(1) \\
\mathcal F_{8,3}^1(1) & \mathcal F_{8,3}^2(1)
\end{smallmatrix}\Big)$\Big].

Given a positive  integer $m$, one can always construct a  Gr\"obner basis for the
Broadhurst--Roberts ideal  $ I_{m+2}^{\mathrm{BR}}= \big\langle\mathbf X_{\left\lfloor (m+1)/2 \right\rfloor}^{\vphantom{\mathrm{T}}}\mathbf D_{m}^{\vphantom{\mathrm{T}}}\mathbf X_{\left  \lfloor (m+1)/2 \right\rfloor}^{{\mathrm{T}}}-\mathbf B_{m}^{\vphantom{\mathrm  T}}\big\rangle$, using the Buchberger algorithm \cite{Buchberger2006}. For moderately large $m$, preliminary experiments in this direction have
not yet produced any new $ \mathbb Q$-linear sum rules [like \eqref{eq:aberrant}] that go beyond the relations for generalized Crandall numbers \eqref{eq:Crandall_num}.

Broadhurst and Roberts originally speculated that the quadratic equation $\mathbf P_m^{}\mathbf D_m^{}\mathbf P_m^{\mathrm T}=\mathbf B_m ^{}$  delivers a complete set of algebraic relations among entries of $  \mathbf P_m$, unless $(m-2)$ is an integer multiple of $4$.
 Fres\'an, Sabbah and Yu provided heuristic justification (see comments after \cite[Theorem 1.4]{FresanSabbahYu2020b}) of this speculation, citing Grothendieck's period conjecture that equates the transcendental degree of $ \mathbb Q(\mathbf P_m)$
with the dimension of the motivic Galois group for $\mathbf P_m $.
 The aberrant sum rule  \eqref{eq:aberrant} serves as a reminder that  there may exist very surprising algebraic relations when $ m\equiv2\pmod4$.

In view of the foregoing applications, we can say that the Broadhurst--Roberts quadratic relations convey more detailed messages about on-shell Bessel moments than the Broadhurst--Mellit determinant formulae. Perhaps paradoxically, the algebraic mechanism  underlying the former relations is essentially identical to what we have explored and exploited during the proof \cite{Zhou2017BMdet} of the latter, namely Wro\'nskian algebra for homogeneous solutions to Vanhove's differential equations.

Our off-shell generalizations (Theorem \ref{thm:W_alg}) to the Broadhurst--Roberts quadratic relations also have some interesting consequences.

If we define $ \dot{\mathcal F}^\ell_{m,j}(u)\colonequals\sqrt{u}\acute{\mathcal F}^\ell_{m,j}(u)=2uD^1\mathcal F^\ell_{m,j}(u)$, then  the Gr\"obner basis of the off-shell ideal\linebreak $ \langle\mathcal L_4^{}(u)\mathbf W_4^{}(u)\mathbf S_4^{}\mathbf W_4^{\mathrm T}(u)\mathbf V_4^{}(u)+\mathbf I_4\rangle$ offers us a relation \begin{align}
\begin{split}0={}&
  2 u (3 u^2-70 u+259)\det\left(\begin{smallmatrix}\mathcal F_{4,2}^1(u)&\mathcal F_{4,3}^1(u)\\\mathcal F_{4,2}^2(u)&\mathcal F_{4,3}^2(u)\end{smallmatrix}\right)+8 (u-13) (u-1)\det\left(\begin{smallmatrix}\mathcal F_{4,2}^1(u)&\mathcal F_{4,3}^1(u)\\ \dot{\mathcal F}_{4,2}^1(u)&\dot{\mathcal F}_{4,3}^1(u)\end{smallmatrix}\right)\\{}&+(u-25) (u-9) (u-1)\left[\det\left(\begin{smallmatrix}\mathcal F_{4,2}^1(u)&\mathcal F_{4,3}^1(u)\\ \dot{\mathcal F}_{4,2}^2(u)&\dot{\mathcal F}_{4,3}^2(u)\end{smallmatrix}\right) +\det\left(\begin{smallmatrix}\mathcal F_{4,2}^2(u)&\mathcal F_{4,3}^2(u)\\ \dot{\mathcal F}_{4,2}^1(u)&\dot{\mathcal F}_{4,3}^1(u)\end{smallmatrix}\right) \right]\end{split}
\end{align}for $u\in(0,1)$, which extends to $ u\in(-\infty,1)$ by analytic continuation. The last displayed identity reveals equivalence between two conjectural determinant formulae
of Broadhurst (see  \cite[\S3]{Broadhurst2020Ell} or \cite[\S5.3]{AcresBroadhurst2021}) \begin{align}
\det
\begin{pmatrix*}[r]
 3 \mathcal F_{4,3}^1(-7)+4 \dot{\mathcal F}_{4,3}^1(-7)  &  \mathcal F_{4,3}^1(-7)+28\mathcal F_{4,3}^2(-7)  \\
 3\mathcal F_{4,2}^1(-7)+4 \dot{\mathcal F}_{4,2}^1(-7)  &  \mathcal F_{4,2}^1(-7)+28\mathcal F_{4,2}^2(-7) \\
\end{pmatrix*}\overset?={}&-\frac{3}{32},\\\det
\begin{pmatrix*}[r]
  \mathcal F_{4,3}^1(-7) & 39\dot{\mathcal F}_{4,3}^1(-7)-427\mathcal F_{4,3}^2(-7)-112 \dot{\mathcal F}_{4,3}^2(-7) \\
  \mathcal F_{4,2}^1(-7) & 39 \dot{\mathcal F}_{4,2}^1(-7)-427\mathcal F_{4,2}^2(-7)-112 \dot{\mathcal F}_{4,2}^2(-7)  \\\end{pmatrix*}\overset?={}&\frac{3}{32},
\end{align}which in turn, are inspired by the recent work of Candelas, de la Ossa, Elmi, and van Straten  \cite{Candelas2020}. Similarly, for $ u=(\sqrt{17}-4)^2$, one can demonstrate that two determinant formulae conjectured by Broadhurst (see  \cite[\S4]{Broadhurst2020Ell} or \cite[\S5.4]{AcresBroadhurst2021}) are equivalent to each other.

It is perhaps appropriate to ask whether  there is an algebro-geometric interpretation of our
off-shell quadratic relations, in the spirit of Fres\'an--Jossen \cite{FresanJossen2020} and Fres\'an--Sabbah--Yu \cite{FresanSabbahYu2018,FresanSabbahYu2020a,FresanSabbahYu2020b}. We would like to see such an extension of the  motivic method to certain algebraic values of the off-shell parameter $u$. In the special cases of $ \mathcal F_{4,2}^1(-7)$, $ \mathcal F_{4,3}^1(-7)$, $\mathcal F_{4,2}^1((\sqrt{17}-4)^2) $ and $\mathcal F_{4,3}^1((\sqrt{17}-4)^2) $, we hope that powerful tools in algebraic geometry may eventually lead to proofs of Broadhurst's recent conjectures (see  \cite[\S\S3--4]{Broadhurst2020Ell} or \cite[\S\S5.3--5.4]{AcresBroadhurst2021}) that evaluate these off-shell Bessel moments  in terms of special $L$-values.

\subsection{Alternative representations of de Rham matrices\label{subsec:deRham_alt}}
From the invertibility of $ \mathbf P_m $ guaranteed by~\eqref{eq:detMdetN},
we know that the de Rham matrix $ \mathbf D_{m}$  satisfying the Broadhurst--Roberts quadratic relation $\mathbf P_m^{}\mathbf D_m^{}\mathbf P_m^{\mathrm T}=\mathbf B_m ^{}$   is unique. Nonetheless, there exist several different expressions for these de Rham matrices, whose mutual equivalence is not self-evident (absent the uniqueness argument). \begin{proposition}[Representations of $ \mathbf D_{m}$]\label{prop:alt_deRham}Let $\pmb{\boldsymbol\beta}_m (u)$ be the Bessel matrix defined in Proposition \ref{prop:Bessel_mat}, and abbreviate $\pmb{\boldsymbol\beta}_m (1)$ as $\pmb{\boldsymbol\beta}_m $. Pick the matrix $ \pmb{\boldsymbol \varrho}_{k}$ as given in Proposition \ref{prop:lim_W}, together with the matrix $ \pmb{\boldsymbol \vartheta}_{m}$ as described in Proposition \ref{prop:block_diag_W}.  We further introduce a matrix $ \boldsymbol\varPsi_k\in\mathbb Z^{(2k-1)\times(2k-2)}$ with entries\begin{align}
(\boldsymbol\varPsi_k)_{a,b}=\begin{cases}\delta_{a,b}(1-\delta_{a,k}), & a\in\mathbb Z\cap[1,k] ,\\
\delta_{a,b+1}, & a\in\mathbb Z\cap[k+1,2k-1]. \\
\end{cases}
\end{align}
The following representations of de Rham matrices hold for $ k\in\mathbb Z_{>1}$:{\allowdisplaybreaks\begin{align}
\frac{|\mathcal L_{2k-1}(1)|(\pmb{\boldsymbol \vartheta}^{-1}_{2k-1})^{\mathrm T}(\pmb{\boldsymbol \beta}_{2k-1}^{-1}\vphantom{\mathrm T})^{\mathrm T}\mathbf V^{\vphantom{\mathrm T}}_{2k-1}
(1)\pmb{\boldsymbol \beta}_{2k-1}^{-1}\pmb{\boldsymbol \vartheta}^{-1}_{2k-1}}{2^{2}(2k+1)(-1)^{k-1}}={}&\begin{pmatrix}\left(\frac{2}{2k+1}\right)^2\mathbf D_{2k-1} &  \\
 &  \mathbf D_{2k-3}\\
\end{pmatrix},\label{eq:deRhamD_1}\\\lim_{u\to1^-}\frac{|\mathcal L_{2k}(u)|(\pmb{\boldsymbol\vartheta}_{2k}^{-1})^{\mathrm T}\pmb{\boldsymbol \varrho}_{k}^{\phantom{\mathrm T}}(\pmb{\boldsymbol \beta}_{2k}^{-1}\vphantom{\mathrm T})^{\mathrm T} \mathbf V^{\vphantom{\mathrm T}}_{2k}
(u)\pmb{\boldsymbol \beta}_{2k}^{-1}\pmb{\boldsymbol \varrho}_{k}^{\mathrm T}\pmb{\boldsymbol\vartheta}_{2k}^{-1}}{2^{3}(k+1)(-1)^{k-1}}={}&\begin{pmatrix}\left(\frac{1}{k+1}\right)^2\mathbf D_{2k} &  \\
 &  \mathbf D _{2k-2}\\
\end{pmatrix},\label{eq:deRhamd_1}\\\lim_{u\to0^{+}}\frac{|\mathcal L_{2k}(u)|[\pmb{\boldsymbol \beta}_{2k}^{\mathrm T}(u)]^{-1} \mathbf V^{\vphantom{\mathrm T}}_{2k}
(u)[\pmb{\boldsymbol \beta}_{2k}^{\vphantom{\mathrm T}}(u)]^{-1}}{2^3(-1)^{k} }={}&\begin{pmatrix*}[r] & -\mathbf D_{2k-1} \\
\mathbf D_{2k-1}& \smash{\underset{^\circ}{\mathbf D}}_{2k-1} \\
\end{pmatrix*},\label{eq:deRhamD_2}\\\lim_{u\to0^{+}}\frac{|\mathcal L_{2k-1}(u)|\boldsymbol\varPsi_k^{\mathrm{T}}[\pmb{\boldsymbol \beta}_{2k-1}^{\mathrm T}(u)]^{-1} {\mathbf V}^{\vphantom{\mathrm T}}_{2k-1}
(u)[\pmb{\boldsymbol \beta}_{2k-1}^{\vphantom{\mathrm T}}(u)]^{-1}\boldsymbol\varPsi_k}{2^3(-1)^{k}}^{\vphantom1}={}&\begin{pmatrix*}[r] & -\mathbf D _{2k-2} \\
\mathbf D_{2k-2}& \smash{\underset{^\circ}{\mathbf D}}_{2k-2}  \\
\end{pmatrix*},\label{eq:deRhamd_2}
\end{align}}where $ \smash{\underset{^\circ}{\mathbf D}}_m$ and $ (\smash{\underset{^\circ}{\mathbf B}}_m)_{a,b}=2^{-3}(-1)^{\left\lfloor \frac{m}{2}\right\rfloor+1}(\smash{\underset{^\circ}{\pmb{\mathscr
B}}}_{m+1}^D)_{a,b}
$ [defined in \eqref{eq:B0Dmat_defn}] play r\^oles in the following  relation for $\smash{\underset{^\circ}{\mathbf P}}_m$ [defined in \eqref{eq:P0_defn}]:\begin{align}
\mathbf P^{\vphantom{\mathrm  T}}_m \smash{\underset{^\circ}{\mathbf D}}_m^{\vphantom{\mathrm  T}}\mathbf{P}_m^{\mathrm T}={}&\pi\smash{\underset{^\circ}{\mathbf B}}_m ^{\vphantom1} +\smash{\underset{^\circ}{\mathbf{P}}}_m ^{\vphantom1}\mathbf D_{m}^{\vphantom{\mathrm  T}}\mathbf{P}_m^{\mathrm  T}-\mathbf{P}_m\mathbf D_{m}^{\vphantom{\mathrm  T}}\smash{\underset{^\circ}{\mathbf{P}}}_m  ^{\mathrm T}.\label{eq:M0M_quad}\end{align}
 \end{proposition}\begin{proof}According to Proposition \ref{prop:inv_S}, the left-hand of \eqref{eq:deRhamD_1} [resp.\ \eqref{eq:deRhamd_1}] is equal to \begin{align}
\begin{pmatrix}\mathbf P^{\vphantom{\mathrm  T}}_m &
 \\[5pt] & -\mathbf P^{\vphantom{\mathrm  T}}_{m-2}  \\
\end{pmatrix}^{-1}{}&\begin{pmatrix}\left(\frac{2}{m+2}\right)^2\mathbf B_{m} &  \\
 &  \mathbf B_{m-2}\\
\end{pmatrix}\begin{pmatrix}\mathbf P_{m}^{\mathrm T} & \\[5pt]
 & -\mathbf P_{m-2}^{\mathrm T}  \\
\end{pmatrix}^{-1}\end{align}for $ m=2k-1$ (resp.\ $m=2k$), hence expressible in terms of de Rham matrices.

In the $u\to0^+$ regime, we have\begin{align}
I_0(\sqrt{u}t)={}&1+O(ut^2),\\K_0(\sqrt{u}t)+\left(\gamma_{0}+\log \frac{ \sqrt{u}t}{2} \right) I_0(\sqrt{u}t)={}&O(ut^2),\\\sqrt{u}tI_1(\sqrt{u}t)={}&O(ut^2),\\\sqrt{u}t\left[K_1(\sqrt{u}t)-\left(\gamma_{0}+\log \frac{ \sqrt{u}t}{2} \right) I_1(\sqrt{u}t)\right]={}&1+O(ut^2),
\end{align}where $ \gamma_0$ is the Euler--Mascheroni constant.

With a reshuffling-rescaling matrix $ \pmb{ \boldsymbol\xi}_{2k}\in\mathbb Q^{2k\times2k}$  defined by \eqref{eq:Rmat}, we may paraphrase \cite[Proposition 4.7]{Zhou2017BMdet} into\begin{align}
\lim_{u\to0^+}\pmb{\boldsymbol \beta}_{2k}^{\vphantom{\mathrm T}}(u)\mathbf W^{\vphantom{\mathrm T}}_{2k}(u)\pmb{ \boldsymbol\xi}_{2k}^{\mathrm T}\pmb{\boldsymbol \gamma}_{2k}^{\mathrm T}(u)=\begin{pmatrix*}[r]\sqrt{\pi}\mathbf P_{2k-1}^{\mathrm T} & -\frac{1}{\sqrt{\pi}}\smash{\underset{^\circ}{\mathbf {P}}}_{2k-1} ^{\mathrm T}\\[5pt]
 & -\frac{1}{\sqrt{\pi}} \mathbf P_{2k-1}^{\mathrm T}\\
\end{pmatrix*},
\end{align}where $\pmb{\boldsymbol \gamma}_{2k}(u)=\Big(\begin{smallmatrix*}[l]\mathbf I_{k} &  \\
\frac{1}{\pi}\left(\gamma_{0}+\log \frac{ \sqrt{u}}{2} \right) \mathbf I_{k} & \mathbf I_{k}\end{smallmatrix*}\Big) $,  and $ \smash{\underset{^\circ}{\mathbf P}}_{2k-1}$ is described in \eqref{eq:P0_defn}. Meanwhile, by direct computation, we have\begin{align}\begin{split}
(-1)^k2^3\begin{pmatrix*}[r] & \mathbf B_{2k-1}\\
-\mathbf B_{2k-1} & \smash{\underset{^\circ}{\mathbf B}}_{2k-1} \\
\end{pmatrix*}={}&\pmb{ \boldsymbol\xi}_{2k}^{\vphantom{\mathrm T}}\mathbf S_{2k}^{-1}(u)\pmb{ \boldsymbol\xi}_{2k}^{\mathrm T}= \pmb{\boldsymbol \gamma}_{2k}^{\vphantom{\mathrm T}}(u)\pmb{ \boldsymbol\xi}_{2k}^{\vphantom{\mathrm T}}\mathbf S_{2k}^{-1}(u)\pmb{ \boldsymbol\xi}_{2k}^{\mathrm T}\pmb{\boldsymbol \gamma}_{2k}^{\mathrm T}(u)\\={}&|\mathcal L_{2k}(u)| \pmb{\boldsymbol \gamma}_{2k}^{\vphantom{\mathrm T}}(u)\pmb{ \boldsymbol\xi}_{2k}^{\vphantom{\mathrm T}}\mathbf W^{{\mathrm T}}_{2k}(u)\mathbf V^{\vphantom{\mathrm T}}_{2k}
(u)\mathbf W^{\vphantom{\mathrm T}}_{2k}(u)\pmb{ \boldsymbol\xi}_{2k}^{\mathrm T}\pmb{\boldsymbol \gamma}_{2k}^{\mathrm T}(u).\end{split}
\end{align} Therefore, the limit in \eqref{eq:deRhamD_2} exists, and is equal to\begin{align}\begin{split}&\begin{pmatrix*}[r]\sqrt{\pi}\mathbf P_{2k-1}^{\vphantom{\mathrm T}} &  ^{\vphantom{\mathrm T}}\\[5pt]
-\frac{1}{\sqrt{\pi}}\smash{\underset{^\circ}{\mathbf {P}}}_{2k-1} & -\frac{1}{\sqrt{\pi}} \mathbf P_{2k-1}^{\vphantom{\mathrm T}}\\
\end{pmatrix*}^{-1}\begin{pmatrix*}[r] & \mathbf B_{2k-1}\\
-\mathbf B_{2k-1} & \smash{\underset{^\circ}{\mathbf B}}_{2k-1} \\
\end{pmatrix*}
\begin{pmatrix*}[r]\sqrt{\pi}\mathbf P_{2k-1}^{\mathrm T} &- \frac{1}{\sqrt{\pi}}\smash{\underset{^\circ}{\mathbf {P}}}_{2k-1}
^{\mathrm T}\\[5pt]
 & -\frac{1}{\sqrt{\pi}} \mathbf P_{2k-1}^{\mathrm T}\\
\end{pmatrix*}^{-1}\\={}&\begin{pmatrix*}[r] & -\mathbf D_{2k-1} \\
\mathbf D_{2k-1} & \mathbf P_{2k-1}^{-1}\left[\pi\smash{\underset{^\circ}{\mathbf B}}_{2k-1} ^{\vphantom1} +\smash{\underset{^\circ}{\mathbf {P}}}_{2k-1} ^{\vphantom1}\mathbf P_{2k-1}^{-1}\mathbf B_{2k-1}^{\vphantom1}-\mathbf B_{2k-1}^{\vphantom1}\left(\mathbf P_{2k-1}^{\mathrm T}\right)^{-1}\smash{\underset{^\circ}{\mathbf {P}}}_{2k-1}^{\mathrm T}\right]\left(\mathbf P_{2k-1}^{\mathrm T}\right)^{-1}\\
\end{pmatrix*}.\end{split}
\end{align}Recalling the Broadhurst--Roberts quadratic relation $ \mathbf P_{2k-1}^{}\mathbf D_{2k-1}^{\vphantom{\mathrm  T}}\mathbf P_{2k-1}^{\mathrm  T}=\mathbf B_{2k-1}$, we may also rewrite $ \smash{\underset{^\circ}{\mathbf {P}}}_{2k-1} ^{\vphantom1}\mathbf P_{2k-1}^{-1}\mathbf B_{2k-1}^{\vphantom1}-\mathbf B_{2k-1}^{\vphantom1}\left(\mathbf P_{2k-1}^{\mathrm T}\right)^{-1}\smash{\underset{^\circ}{\mathbf {P}}}_{2k-1}^{\mathrm T}=\smash{\underset{^\circ}{\mathbf {P}}}_{2k-1} ^{\vphantom1}\mathbf D_{2k-1}^{\vphantom{\mathrm  T}}\mathbf P_{2k-1}^{\mathrm  T}-\mathbf P_{2k-1}\mathbf D_{2k-1}^{\vphantom{\mathrm  T}}\smash{\underset{^\circ}{\mathbf {P}}}_{2k-1}^{\mathrm T}$, thereby confirming \eqref{eq:M0M_quad} when $m=2k-1$ is an odd number.

When $m=2k$ is an even number, the ideas behind \eqref{eq:M0M_quad}  are essentially similar, except that one needs a variation on Lemma \ref{lm:margin_v} to show that the $k$-th row and the $k$-th column of $ \lim_{u\to0^+}|\mathcal L_{2k-1}(u)|[\pmb{\boldsymbol \beta}_{2k-1}^{\mathrm T}(u)]^{-1} \linebreak{\mathbf V}^{\vphantom{\mathrm T}}_{2k-1}
(u)[\pmb{\boldsymbol \beta}_{2k-1}^{\vphantom{\mathrm T}}(u)]^{-1}$ are filled with zeros. We leave the details to diligent readers.
  \end{proof}
\begin{remark}Equating the different representations of de Rham matrix  $ \mathbf D_{m}$, we may deduce recursive relations connecting a party of integers $D^n \ell_{m,j}(1),j\in\mathbb Z\cap[0,m]$ to another party  $D^n \ell_{m-2,j}(1),j\in\mathbb Z\cap[0,m-2]$ [see \eqref{eq:deRhamD_1} and \eqref{eq:deRhamd_1}], as well as relations connecting both parties to   $D^n \ell_{m-1,j}(0),j\in\mathbb Z\cap[0,m-1]$ [see \eqref{eq:deRhamD_2} and \eqref{eq:deRhamd_2}]. These recursions on $ D^n \ell_{m,j}$ are effectively just restatements of the following facts under suitably chosen bases: $ I_0(t)K_1(t)+I_1(t)K_0(t)=\frac{1}{t}$,  $ \lim_{u\to0^+}I_0(\sqrt{u}t)=1$ and $ \lim_{u\to0^+}\sqrt{u}tK_1(\sqrt{u}t) =1$. Direct combinatorial proofs of such recursions may appear hard, though not infeasible.     \eor\end{remark}\begin{corollary}[Anti-diagonal structure of $ \mathbf D_{2k-1}$]\label{cor:anti_diag_Dk}The de Rham matrix  $ \mathbf D_{2k-1}$ is upper-left triangular, with  its main anti-diagonal  being occupied by   $ \left[\frac{(2k+1)!!}{2^{k+1}}\right]^2$.\end{corollary}\begin{proof}After inspecting the highest order derivatives and highest order moments in $2k$ equations generated from rows of  \eqref{eq:Bessel_mat_inv}, we immediately  know that  $[\pmb{\boldsymbol \beta}_{2k}^{\mathrm T}(u)]^{-1}$ has the following structure:\begin{align}
([\pmb{\boldsymbol \beta}_{2k}^{\mathrm T}(u)]^{-1} )_{a,b}=\begin{cases}0,&a=1,b\in\mathbb Z\cap[2,2k],\\(-4u)^{1-a},&a\in\mathbb Z\cap[1,k],b= 2a-1,\\
0, & a\in\mathbb Z\cap[2,k],b\in\mathbb Z\cap[1,2a-1), \\
\frac{1}{2u} (-4u)^{k+1-a},& a\in\mathbb Z\cap[k+1,2k],b= 2(a-k), \\
0, & a\in\mathbb Z\cap[k+1,2k],b\in\mathbb Z\cap[1,2(a-k)).\ \\
\end{cases}
\end{align}Since $ \mathbf V_{2k}
(u)$ is upper-left triangular, with anti-diagonal elements $ (\mathbf V_{2k}
(u))_{a,2k+1-a}=(-1)^{a-1}$, we can show that\begin{align}
([\pmb{\boldsymbol \beta}_{2k}^{\mathrm T}(u)]^{-1} \mathbf V^{\vphantom{\mathrm T}}_{2k}(u))_{a,b}={}&\begin{cases}(-4u)^{1-a},&a\in\mathbb Z\cap[1,k],b=2(k-a),\\
0, & a\in\mathbb Z\cap[2,k],b\in\mathbb Z\cap(2(k-a),2k], \\
-\frac{1}{2u} (-4u)^{k+1-a},& a\in\mathbb Z\cap[k+1,2k],b= 2(2k-a)+1, \\
0, & a\in\mathbb Z\cap[k+1,2k],b\in\mathbb Z\cap( 2(2k-a)+1,2k].\ \\
\end{cases}
\end{align}This further leads us to a block partition  $ [\pmb{\boldsymbol \beta}_{2k}^{\mathrm T}(u)]^{-1} \mathbf V^{\vphantom{\mathrm T}}_{2k}
(u)[\pmb{\boldsymbol \beta}_{2k}^{\vphantom{\mathrm T}}(u)]^{-1}=\Big(\begin{smallmatrix*}[r]*&\boldsymbol U\\-\boldsymbol U^{\mathrm T}&*\end{smallmatrix*}\Big)$ where $ \boldsymbol U\in\mathbb Q(u)^{k\times k}$ is an upper-left triangular matrix whose  main anti-diagonal  is occupied by  $ \frac{(-1)^{k-1}}{2^{2k-1}u^k}$. (Here, an asterisk denotes a block with irrelevant details.) Referring back to \eqref{eq:deRhamD_2}, we know that $ \mathbf D_{2k-1}$ is an upper-left triangular matrix whose main anti-diagonal is filled with \begin{align}
\lim_{u\to0^{+}}\frac{(-1)^{k-1}}{2^{2k-1}u^k}\frac{|\mathcal L_{2k}(u)|}{2^3(-1)^{k-1}}=\left[ \frac{(2k+1)!!}{2^{k+1}} \right]^2,
\end{align}as claimed.\end{proof}\begin{remark}In \cite[Theorem 1.4]{FresanSabbahYu2020b}, Fres\'an, Sabbah and Yu  have characterized (effectively) the inverse de Rham matrices $ \mathbf D_{2k-1}^{-1},k\in\mathbb Z_{>0}$ and $ \mathbf D_{2k}^{-1},k\in2\mathbb Z_{>0}$, in the form of lower-right triangular matrices, with explicit formulae for elements on the main anti-diagonal. Paraphrasing their result with matrix inversions and suitable normalizations, we can also confirm the anti-diagonal behavior of   $ \mathbf D_{2k-1}$.

Extending the proof the proposition above, one can show that the de Rham matrix  $ \mathbf D_{2k}$ is always upper-left triangular. When  $k$ is odd, the center of  the main anti-diagonal in the skew-symmetric matrix $ \mathbf D_{2k}^{\vphantom{\mathrm  T}}=\mathbf P^{-1}_{2k}\mathbf B_{2k}
^{\vphantom{\mathrm  T}}\big(\mathbf P_{2k}^{\mathrm  T}\big)^{-1} $ must be zero, hence $ \det  \mathbf D_{2k}=0$. (Alternatively, one may deduce this fact from $ \det\mathbf B_{2k}=0$, which has been proved in Corollary \ref{cor:detBetti}.)

To avoid inverting a singular de Rham matrix, Fres\'an--Sabbah--Yu \cite[\S3.b]{FresanSabbahYu2020b} constructed a quadratic relation that was different from  $ \mathbf P_{2k}^{}\mathbf D_{2k}^{}\mathbf P_{2k}^{\mathrm T}=\mathbf B_{2k} ^{}$   for odd $k$, without involving the last row in $ \mathbf P_{2k}$.  Accordingly, the Gr\"obner basis generating their quadratic relation does not contain linear sum rules derivable from   generalized Crandall numbers (see \eqref{eq:Crandall_num}, which recapitulates \cite[(65)]{Zhou2018ExpoDESY}).  \eor\end{remark}

The uniqueness of de Rham matrices gives an indirect proof that our constructions of $ \mathbf D_{2k-1}^{},k\in\mathbb Z_{>0}$ and $\mathbf  D_{2k},k\in2\mathbb Z_{>0}$ are equivalent to Fres\'an--Sabbah--Yu \cite[\S3.b]{FresanSabbahYu2020b}. Without resorting to  a uniqueness argument, we may interpret such equivalence directly through the following identities (cf.~\cite[(4.22)]{Zhou2017BMdet} or \cite[(2.10]{Zhou2019BMdimQ}): \begin{align}
\left\{ \begin{array}{c}
\widetilde L_m(uD^2+D^1)\frac{I_0(\sqrt{u}t)}{t}=\frac{(-1)^m}{2^{m+2}}L^*_{m+2}\frac{I_0(\sqrt{u}t)}{t}, \\
\widetilde L_m(uD^2+D^1)\frac{K_0(\sqrt{u}t)}{t}=\frac{(-1)^m}{2^{m+2}}L^*_{m+2}\frac{K_0(\sqrt{u}t)}{t}. \\
\end{array} \right.
\label{eq:VanhoveBMW}\end{align}Here, the operator $ L^*_{m+2}$ is formal adjoint to the Borwein--Salvy operator $ L_{m+2}$ \cite[Lemma 3.3]{BorweinSalvy2007}, which  is the $ (m+1)$-st symmetric  power of the Bessel differential operator $ \left(t\frac{\D}{\D t}\right)^2-t^{2}\left( t\frac{\D }{\D t} \right)^0$, annihilating each member in the set $ \{[I_0(t)]^j[K_0(t)]^{m+1-j}|j\in\mathbb Z\cap[0,m+1]\}$. The Borwein--Salvy operator $ L_{m+2}=\mathscr L_{m+2,m+2}$ is constructed   recursively by the Bronstein--Mulders--Weil algorithm \cite[Theorem 1]{BMW1997}:\begin{align}
\begin{cases}\mathscr L_{m+2,0}=\left( t\frac{\D }{\D t} \right)^0,\mathscr L_{m+2,1}=t\frac{\D }{\D t}, &  \\
\mathscr L_{m+2,k+1}=t\frac{\D }{\D t}\mathscr L_{m+2,k}-k(m+2-k)t^{2}\mathscr L_{m+2,k-1}, &          \forall k\in\mathbb Z\cap[1,m+1]. \\
\end{cases}\label{eq:BMW}
\end{align}To find out the de Rham matrix that fits into the relation $\mathbf P_{m}\mathbf D_{m}=\mathbf B_{m}\frac{\cof \mathbf P_{m}}{\det \mathbf P_{m}} $,  we have differentiated off-shell Bessel moments with respect to $u$, and have evaluated these derivatives in the  $u\to1^-$ limit. To attain the same goal, Fres\'an, Sabbah and Yu \cite[Proposition 3.4]{FresanSabbahYu2020b} have integrated by parts with respect to $t$, relating Bessel moments $ \IKM(a,b;n)$ of different orders   $n$ by a recursive algorithm \cite[(3.9)]{FresanSabbahYu2020b} that emulates  Bronstein--Mulders--Weil
 \cite[Theorem 1]{BMW1997} and Borwein--Salvy \cite[Lemma 3.3]{BorweinSalvy2007}. In other words, the equivalence between our Wro\'nskian approach to de Rham matrices and the non-Wro\'nskian approach of Fres\'an--Sabbah--Yu \cite[\S3.b]{FresanSabbahYu2020b} ultimately roots from the duality relation \eqref{eq:VanhoveBMW} between Vanhove and Borwein--Salvy operators, at $u=1$. We will not further elaborate on such a duality, as it is not essential to the theme of the current work.

Broadhurst and Roberts   have proposed a set of combinatorial formulae \cite[(5.7)--(5.12)]{BroadhurstRoberts2018} that
conjecturally generate the de Rham matrices  $\mathbf D_{m},m\in\mathbb Z_{>0}$ through a recursion on $m$. The uniqueness argument alone does not allow us to settle their conjecture. What we currently know is that various representations of   $\mathbf D_{m}$ in Proposition  \ref{prop:alt_deRham} do provide us with recursive relations concerning special values of the polynomials $ D^n \ell_{m,j}$ with three consecutive values of $m$, which are reminiscent of  the  three-term recursion for the Broadhurst--Roberts bivariate polynomials \cite[(5.10)]{BroadhurstRoberts2018}. Perhaps the former three-term relations can be translated into the  Broadhurst--Roberts recursions for   $\mathbf D_m$, through some combinatorial efforts.

It is our hope that the recursive structures of $\mathbf D_{m}$, once completely understood, will shed new light on the algebraic and arithmetic nature of the Broadhurst--Roberts  ideal $ I_{m+2}^{\mathrm{BR}}= \big\langle\mathbf X_{\left\lfloor (m+1)/2 \right\rfloor}^{\vphantom{\mathrm{T}}}\mathbf D_{m}^{\vphantom{\mathrm{T}}}\mathbf X_{\left  \lfloor (m+1)/2 \right\rfloor}^{{\mathrm{T}}}-\mathbf B_{m}^{\vphantom{\mathrm  T}}\big\rangle$ and the corresponding Broadhurst--Roberts variety $ V_{m+2}^{\mathrm{BR}}\colonequals \Big\{\mathbf X_{\left\lfloor (m+1)/2 \right\rfloor}^{\vphantom{\mathrm{T}}}\in\mathbb C^{\left\lfloor \frac{m+1}{2} \right\rfloor\times \left\lfloor \frac{m+1}{2} \right\rfloor}\Big|\mathbf X_{\left\lfloor (m+1)/2 \right\rfloor}^{\vphantom{\mathrm{T}}}\mathbf D_{m}^{\vphantom{\mathrm{T}}}\linebreak\mathbf X_{\left\lfloor (m+1)/2 \right\rfloor}^{{\mathrm{T}}}-\mathbf B_{m}^{\vphantom{\mathrm{T}}}=0\Big \}$.

\end{document}